\documentclass[twoside,letterpaper,12pt,centertags]{amsart}

\usepackage[T1]{fontenc}
\usepackage[utf8]{inputenc}
\usepackage{subfiles}
\usepackage{multicol}

\makeatletter
\usepackage[section]{placeins}

\AtBeginDocument{%
	\expandafter\renewcommand\expandafter\subsection\expandafter{%
		\expandafter\@fb@secFB\subsection
	}%
}
\makeatother
\makeatletter
\@namedef{subjclassname@2020}{%
	\textup{2020} Mathematics Subject Classification}
\makeatother

\usepackage{amsfonts,amssymb,amsmath,amsthm}
\usepackage{mathrsfs}
\usepackage{nicefrac}
\usepackage{latexsym}
\usepackage{mathtools}
\usepackage[noadjust]{cite}
\usepackage{paralist}
\usepackage{longtable}
\usepackage{float}
\usepackage{pdfsync}
\usepackage{tikz,tikz-cd}
\usetikzlibrary{patterns}
\usepackage{hhline}
\usepackage{import}

\usepackage{enumitem}

\usepackage[colorlinks=true,citecolor=magenta,linkcolor=cyan,backref]{hyperref}
\AtBeginDocument{}

\numberwithin{equation}{section}

\theoremstyle{plain}
\newtheorem{thm}{Theorem}[section]
\newtheorem{lem}[thm]{Lemma}
\newtheorem{cor}[thm]{Corollary}
\newtheorem{prop}[thm]{Proposition}

\theoremstyle{definition}
\newtheorem{defn}[thm]{Definition}
\newtheorem{rem}[thm]{Remark}

\usepackage{color}

\DeclareMathOperator*{\PV}{PV}

\newcommand{\IC}{\mathbb{C}}

\newcommand{\IN}{\mathbb{N}}

\newcommand{\R}{\mathbb{R}}

\newcommand{\cA}{\mathcal{A}}
\newcommand{\cT}{\mathcal{T}}
\newcommand{\cD}{\mathcal{D}} 
\newcommand{\cF}{\mathcal{F}} 
\newcommand{\cG}{\mathcal{G}}
\newcommand{\cH}{\mathcal{H}}

\newcommand{\cK}{\mathcal{K}} 
\newcommand{\cL}{\mathcal{L}}

\newcommand{\cS}{\mathcal{S}} 

\newcommand{\cU}{\mathcal{U}}

\newcommand{\loc}{\operatorname{loc}}
\renewcommand{\L}{\operatorname{L}} 
\newcommand{\C}{\operatorname{C}} 
\renewcommand{\H}{\operatorname{H}} 
\newcommand{\W}{\operatorname{W}}
\newcommand{\B}{\operatorname{B}}
\newcommand{\F}{\operatorname{F}}

\newcommand{\V}{\operatorname{V}}
\newcommand{\X}{\operatorname{X}} 
\newcommand{\Y}{\operatorname{Y}}
\newcommand{\Z}{\operatorname{Z}}
 \newcommand{\U}{\operatorname{U}}
\DeclareRobustCommand{\Hdot}{\dot{\H}\protect{\vphantom{H}}
} 
\DeclareRobustCommand{\Wdot}{\dot{\W}\protect{\vphantom{W}}} 
\DeclareRobustCommand{\Ydot}{\dot{\V}\protect{\vphantom{V}}} 
\DeclareRobustCommand{\Fdot}{\dot{\F}\protect{\vphantom{F}}} 
\DeclareRobustCommand{\Xdot}{\dot{\X}\protect{\vphantom{X}}} 
\DeclareRobustCommand{\Ydot}{\dot{\Y}\protect{\vphantom{Y}}} 
\DeclareRobustCommand{\Zdot}{\dot{\Z}\protect{\vphantom{Z}}} 
 
\DeclareRobustCommand{\Udot}{\dot{\U}\protect{\vphantom{K}}} 
 
\DeclareRobustCommand{\cFdot}{\dot{\cF}\protect{\vphantom{\cF}}} 

\DeclareRobustCommand{\cLdot}{\dot{\cL}\protect{\vphantom{\cL}}} 
\DeclareRobustCommand{\cTdot}{\dot{\cT}\protect{\vphantom{\cT}}}
\DeclareRobustCommand{\cUdot}{\dot{\cU}\protect{\vphantom{\cU}}}
\DeclareRobustCommand{\cGdot}{\dot{\cG}\protect{\vphantom{\cG}}}

\newcommand{\eps}{\varepsilon} 
\renewcommand\Re{\operatorname{Re}}

\def\angle#1#2{\langle #1,#2 \rangle} 

\renewcommand{\iint}{\int_{}\kern-.34em \int} 
\renewcommand{\iiint}{\iint_{}\kern-.34em \int} 

\newcommand{\dd}{\, \mathrm{d} }
\newcommand{\dv}{\, \mathrm{d} v}
\newcommand{\dx}{\, \mathrm{d} x}
\newcommand{\ds}{\, \mathrm{d} s}

\newcommand{\dt}{\, \mathrm{d} t}

\title[Weak solutions to Kolmogorov--Fokker--Planck equations]{Weak solutions to Kolmogorov--Fokker--Planck equations: regularity, existence and uniqueness}
\author{Pascal Auscher}
\address{Universit\'e Paris-Saclay, CNRS, Laboratoire de Math\'{e}matiques d'Orsay, 91405 Orsay, France\\ and 
French-Australian  Mathematical Sciences and Interactions, Australian National University-CNRS, Canberra, ACT, 2601, Australia}
\email{pascal.auscher@universite-paris-saclay.fr}
\author{Cyril Imbert}
\address{D\'epartement de Math\'ematiques et Applications, \'Ecole normale sup\'erieure, Universit\'e PSL, CNRS, 75005 Paris, France }
\email{cyril.imbert@ens.psl.eu}
\author{Lukas Niebel}

\address{Institut f\"ur Analysis und Numerik, Westf\"alische Wilhelms-Universit\"at M\"unster\\
Orl\'eans-Ring 10, 48149 M\"unster, Germany.}
\email{lukas.niebel@uni-muenster.de}

\date{December 5, 2025
}
\thanks{The third author is funded by the Deutsche Forschungsgemeinschaft (DFG, German Research Foundation) under Germany's Excellence Strategy EXC 2044 --390685587, Mathematics M\"unster: Dynamics--Geometry--Structure. A CC-BY 4.0 \url{https://creativecommons.org/licenses/by/4.0/} public copyright license has been applied by the authors to the present document and will be applied to all subsequent versions up to the Author Accepted Manuscript arising from this submission.}



\begin{document}

\begin{abstract} 
We prove existence, uniqueness and regularity of weak solutions of Kolmogorov--Fokker--Planck equations with either local or non-local diffusion in the velocity variable and rough diffusion coefficients or kernels. Our results cover the Cauchy problem and allow a broad class of source terms under minimal assumptions. The core of the analysis is a set of sharp kinetic embeddings \`a la Lions and transfer-of-regularity results \`a la Bouchut--H\"ormander. We formulate these tools in a homogeneous, scale-invariant form, available for a large range of regularity parameters. 
\end{abstract}

\maketitle

\tableofcontents
 
\section{Introduction} 

Motivated by the study of non-linear kinetic equations, this work is concerned with proving the existence and uniqueness of weak solutions for linear kinetic Kolmogorov--Fokker--Planck equations of both local and non-local (fractional) type, with rough diffusion coefficients or kernels.
To the best of our knowledge, this work establishes for the first time the well-posedness of the non-local Kolmogorov--Fokker--Planck equation with a rough diffusion kernel. 
An essential ingredient in our proof is a kinetic transfer-of-regularity result for weak solutions, which appears here for the first time in the non-local case and improves on the results of H\"ormander \cite{MR0222474} and Bouchut \cite{MR1949176} in the local case.
In all cases, we obtain the regularity exponents implied by the kinetic scaling. Kinetic regularity also allows us to deal with large classes of source terms. 
In particular, this precise solution theory allows us to define a notion of fundamental solution for such equations (presented in a subsequent work \cite{AIN2}).

In this introduction, we first define the class of equations and their weak solutions, and we state our main existence and uniqueness result together with some consequences.
Then we present our main instrumental result on kinetic regularity. We next highlight a few aspects of our proofs and present tools that we use. We also explain the recourse to homogeneous Sobolev spaces to obtain estimates that are invariant under kinetic scaling. 
We continue with a short review of the literature and the state of the art, and we describe the organisation of the article.
This work is the first of a series, and we also describe subsequent papers.

\subsection{Kolmogorov--Fokker--Planck equations}

The equations take the following general form:
\begin{equation}
 \label{eq:weaksol}
 (\partial_t + v \cdot \nabla_x ) f +\cA f = S, 
\end{equation}
where the unknown function $f=f(t,x,v)$, depending on time $t$, position $x$, and velocity $v$, is defined on $\Omega = \R \times \R^d \times \R^d$, with $d \ge 1$, and the source term $S$ is given. The operator $\cA$ is a diffusive linear operator acting on the velocity variable $v$. 
Later, we shall also consider restrictions of time to finite intervals, still with $x$ and $v$ in $\R^d$.

A first example is the kinetic Fokker--Planck equation corresponding to
\begin{equation}
 \label{e:fk}
 \cA f(t,x,v) = -\nabla_v \cdot ( \mathbf{A} \nabla_v f)(t,x,v) 
\end{equation}
where $\mathbf{A}=\mathbf{A}(t,x,v)$ is a real symmetric matrix satisfying for a.e. $(t,x,v) \in \Omega$,
\[
  \text{ eigenvalues of } \mathbf{A} (t,x,v) \text{ lie in } [\lambda, \Lambda] \]
for some $0 < \lambda \le \Lambda<\infty$. We shall see that coercivity and boundedness of the matrix $\mathbf{A}$ are sufficient (real symmetry plays no role), and no regularity other than measurability of the coefficients is required. The case $- \cA=\Delta_{v}$, the Laplace operator in the variable $v$, yields the well-known Kolmogorov equation.

Another important class of examples consists of kinetic equations with integral diffusion. They correspond to
\begin{equation}
 \label{e:A-kernel}
 \cA f (t,x,v) = \PV \int_{\R^d} (f(t,x,v) - f(t,x,v')) k (t,x,v,v') \dv'
\end{equation}
where the function $k \colon \Omega \times \R^d \to \R$ is non-negative and satisfies appropriate ellipticity conditions.
The main example is $\cA f (t,x,v) = (-\Delta_v)^\beta f (t,x,v)$ for which $k (t,x,v,v') = c_{d,\beta} |v'-v|^{-d-2\beta}$, $0<\beta<1$, for some
explicit constant $c_{d,\beta}>0$ and its regularity theory will be instrumental. 
In the general integro-differential case, since the kernel $k$ can be singular,
the meaning of $\cA f$ has to be understood in the principal value ($\PV$) sense.
The simplest form of the ellipticity condition on the kernel $k$ is the following:
\[
 \begin{cases}
   k (t,x,v,v') = k(t,x,v',v), \\
   \frac{\lambda}{|v'-v|^{d+2\beta}} \le k (t,x,v,v') \le \frac{\Lambda}{|v'-v|^{d+2\beta}}.
 \end{cases}
\]
These conditions are considered in many articles concerned with the regularity of parabolic equations with integral diffusion, see for instance \cite{MR2448308,MR2784330,MR2494809}.
Our approach allows us to deal with a larger class of kernels, satisfying weaker ellipticity conditions.
More precisely, our results apply to the class of kinetic equations with integral diffusion introduced in \cite{MR4049224}.
They naturally appear in the context of the study of the Boltzmann equation without Grad's cut-off assumption.
See Remark~\ref{rem:appli-kernel} below.

\subsection{Weak solutions: definition, main result and consequences}

Not only do we want to solve Cauchy problems on finite time intervals, but also to obtain estimates that are scale invariant. Thus, we start with homogeneous assumptions to state our main result. 

For the Kolmogorov--Fokker--Planck equation~\eqref{eq:weaksol} with $\cA= (-\Delta_{v})^{\beta}$, distributional solutions can be built via Fourier transform. For rough coefficients, we shall use a form method.
 
 The inhomogeneous Sobolev spaces $\H^\beta_{\vphantom{v} v}$ on $\R^d$ will be defined in a standard way. By using the Fourier transform, the homogeneous semi-norms
$\|f\|_{\Hdot^\beta_{v}}=\|D_{v}^\beta f\|_{\L^2_{\vphantom{v} v}}$, with $D_{v}=(-\Delta_{v})^{1/2}$, make sense for all $\beta\in \R$ when, say, $f\in \L^2_{v}$ to start with. A completion of the space $f\in \L^2_{v}$ for the semi-norm $\|f\|_{\Hdot^\beta_{v}}$  can always be obtained within the space of tempered distributions modulo polynomials. But this is not convenient when dealing with partial differential equations. Instead, one looks for a completion contained in the space of tempered distributions, which we denote by $\Hdot^\beta_{v}$ and refer to as a homogeneous Sobolev space. When $\beta<d/2$, 
this completion exists as a Hilbert space (see Section~\ref{s:homog-kolmo} for a definition). 
Otherwise, if $\beta\geq d/2$, such a completion exists but only as a semi-normed space, which poses technical difficulties in proofs and statements, so we decide not to consider them in this article,\footnote{With this adequate definition when $\beta\ge d/2$, one can simplify the definition of weak solutions below. Actually, most of our results here in which the threshold $d/2$ appears can be improved to results with a threshold $\mathsf{K}/2$ where $\mathsf{K}=2\beta+(2\beta+2)d$ is the homogeneous dimension for the scaling. We refer to the forthcoming article \cite{AN} for details.} although the norms above are still defined when, in addition, $f\in \L^2_{v}$. 
 
In any case for $\beta>0$, $\Hdot^\beta_{v}\cap \L^2_{v}=\H^\beta_{\vphantom{v} v}$. Hence, the constraint $\beta<d/2$ will arise in some proofs for technical reasons, and the case $\beta\ge d/2$ needs to be treated differently. We thought it interesting to present things like this as it directly covers the new cases of interest when $d=2,3$ and $0<\beta<1$. 

 As suggested above, we want that for all $(t,x) \in \R \times \R^d$, 
$\cA$ maps $\H^\beta_{\vphantom{v} v}$ to $\H^{-\beta}_{\vphantom{v} v}$ in a uniform fashion and in a scale-invariant way.
To make it rigorous and valid for all $\beta>0$, we assume that $\R\times \R^d\ni (t,x)\mapsto a_{t,x}$ is a family of continuous sesquilinear forms $a_{t,x}\colon \H^{\beta}_{v} \times \H^\beta_{\vphantom{v} v} \to \IC$  with the homogeneous estimate
\begin{equation}
 \label{e:ellip-upper}
 |a_{t,x}(f,g)|\le \Lambda \|f\|_{\Hdot^\beta_{v}} \|g\|_{\Hdot^\beta_{v}}
\end{equation}
for some $\Lambda <\infty$, uniformly for $(f,g) \in \H^{\beta}_{v} \times \H^\beta_{\vphantom{v} v}$, $(t,x)\in \R\times \R^d$, and that $(t,x)\mapsto a_{t,x}(f,g)$ is measurable (here, and throughout, ambient measure in $(t,x,v)$ is Lebesgue measure). We also assume homogeneous coercivity, namely that there exists $\lambda>0$ such that for all $(t,x)\in \R\times \R^d$ and all $f\in \H^\beta_{v}$, 
\begin{equation}
 \label{e:ellip-lower}
 \Re a_{t,x}(f,f) \ge \lambda \|f\|_{\Hdot^\beta_{v}}^2.
\end{equation}
 The mixed space 
\[\F^\beta=\L^2_{t,x}\H^{\beta}_{\vphantom{t,x} v}\]
(and its homogeneous version $\Fdot^\beta=\L^2_{t,x}\Hdot^{\beta}_{\vphantom{t,x} v}$ when $\beta<d/2$) plays a central role and 
we set
\begin{equation}
 \label{e:a-defi}
 \forall f,g \in \F^\beta, \qquad a(f,g)=\iint_{\R\times \R^d} a_{t,x}(f(t,x),g(t,x))\, \dt\dx.
\end{equation}
For short, we write $a_{t,x}(f(t,x),g(t,x))$ as $a_{t,x}(f,g)$. 
We define by $\cA: \F^\beta \to (\F^\beta)'$, the complex dual of $\F^\beta$ realized in the extended $\L^2_{t,x,v}$ inner product, the operator associated to $a$ by 
$$ \angle{\cA f}{g}= a(f,g), \quad f,g\in \F^\beta.$$
When $\beta<d/2$, this can be extended to $ \Fdot^\beta$: 
In this case, $a$ (identified with its extension) becomes a continuous and coercive sesquilinear form on $\Fdot^\beta \times \Fdot^\beta$ and $\cA: \Fdot^\beta \to (\Fdot^\beta)'$ is an isomorphism. We shall not distinguish notation.

\begin{rem} \label{rem:appli-kernel} We can work with weaker inhomogeneous formulations of boundedness and ellipticity conditions, see \eqref{e:ellip-upper-inhom} and \eqref{eq:weakelliptic-bis}. The situation is well-understood when $\cA$ is given by \eqref{e:fk} plus lower order terms. 
When $\cA$ is given by a diffusion kernel \eqref{e:A-kernel}, the second author and L.~Silvestre give sufficient conditions in \cite{MR4049224} on the kernel for obtaining the weaker forms: see Condition~(1.3) with $\bar R = +\infty$ for ellipticity and Conditions~(1.4), (1.5), (1.6), (1.7) with $\bar R = +\infty$ for boundedness. The interested reader is referred to \cite[Theorem~4.1]{MR4049224} and \cite[Theorem~2.1]{MR4688651}.
\end{rem}

We now define our classes of weak solutions to the equation $(\partial_{t}+v\cdot\nabla_{x})f +\cA f = S$ and state our main result concerned with their construction for fairly general source terms $S$ in appropriate spaces. 
Define 
\begin{equation*}
\Udot^\beta=
\begin{cases}
 \L^2_{t,x}\Hdot^{\beta}_{v} &,\ \mathrm{if}\ \beta <d/2,\\
 \L^2_{\loc, t}\L^2_{\vphantom{\loc, t,x} x,v}\cap \L^2_{t,x}\Hdot^{\beta}_{v}&, \ \mathrm{if}\ \beta\ge d/2, 
\end{cases}
\end{equation*} 
and set
\begin{equation*}
\|f\|_{\Udot^\beta_{}}= \|f\|_{ \L^2_{ t,x}\Hdot^{\beta}_{\vphantom{t,x} v} }= \| D_{v}^{\beta} f \|_{\L^2_{\vphantom{t,x} t,x,v}}
\end{equation*}
which makes sense for all $\beta>0$ by the above remarks.
The threshold $d/2$ is only there because our definition here of the homogeneous Sobolev spaces has this restriction.

\begin{defn}[Weak solutions]
\label{defn:weaksol}
Let $S \in \cD'(\Omega)$. A distribution $f \in \cD'(\Omega)$ is said to be a \emph{weak solution} to \eqref{eq:weaksol} if $f\in \Udot^\beta$ and for all $h\in \cD(\Omega)$, 
\[
 -\iiint_{\Omega} f (\partial_{t}+v\cdot\nabla_{x}) \overline h \dt \dx \dv + a(f,h)
 = \angle{S}{h}.
\]
\end{defn}

\begin{rem}
 We may also consider the adjoint equation 
 \begin{equation}
 \label{eq:weaksolback}
 -(\partial_t + v \cdot \nabla_x ) f +\cA^* f = S, 
\end{equation}
where $\cA^*$ is the adjoint of $\cA$.
We refer to it as the backward kinetic equation, and we define its weak solutions similarly. The theory we develop will need to consider both equations \eqref{eq:weaksol} and \eqref{eq:weaksolback}.
\end{rem}

\begin{rem} We do not assume $\L^\infty_{\loc, t}\L^2_{x,v}$ in this definition (compare, for example, \cite{MR4049224,MR3923847,MR4653756,MR4453413}). The choice of $ \L^2_{\loc, t}\L^2_{\vphantom{\loc, t} x,v}$ when $\beta\ge d/2$ is somewhat arbitrary. It would suffice to have a subspace of $\L^1_{\loc}(\Omega)$ that embeds in $\L^1_{\loc,t}\,\cS'_{x,v}$ (the space of locally integrable functions on $\R^{}_{t}$ valued in the Schwartz space $\cS'_{x,v}=\cS(\R^{2d})$) and on which we can compute $\|D_{v}^{\beta}f\|_{\L^2_{t,x,v}}$. This is because we use the partial Fourier transform in $(x,v)$ in our arguments. In the end, the weak solutions we construct belong to $\C^{}_{0}(\R^{}_{t}\, ; \L^2_{x,v} )$ anyway.
\end{rem}

\begin{thm}[Existence and uniqueness of weak solutions]
\label{thm:HomExUn} Let $\beta>0.$
Assume that the family $a$ of sesquilinear forms satisfies \eqref{e:ellip-upper} and \eqref{e:ellip-lower}. 
Let $S = S_1 + S_2 + S_3$ with $S_1 \in \L^2_{t,x}\Hdot^{-\beta}_{\vphantom{t,x} v}$, $S_2 \in \L^2_{t,v}\Hdot_{\vphantom{t,v} x}^{-\frac{\beta}{2\beta+1}}$ and
$S_3 \in \L^1_{ t}\L^2_{\vphantom{t,x} x,v}$. There exists a unique weak solution $f\in \Udot^\beta$
to \eqref{eq:weaksol} in the sense of Definition~\ref{defn:weaksol}.
Moreover, $f\in \C^{}_{0}(\R^{}_{t}\, ; \L^2_{x,v} )$ and $f\in \L^2_{t,v}\Hdot^{\frac{\beta}{2\beta+1}}_{\vphantom{t,v}x}$
with \begin{multline*}
 \sup_{t\in \R} \|f(t)\|_{\L^2_{x,v}} + \|D_{\vphantom{D_{v}} x}^{\frac{\beta}{2\beta+1}} f \|_{\L^2_{t,x,v}} + \| D_{v}^{\beta} f \|_{\L^2_{ t,x,v}}  \\
 \lesssim \|D_{v}^{-\beta}S_1 \|_{\L^2_{t,x,v}}+\|D_{\vphantom{D_{v}} x}^{-\frac{\beta}{2\beta+1}}S_2 \|_{\L^2_{t,x,v}} + \|S_3 \|_{ \L^1_{t}\L^2_{x,v}}
\end{multline*}
for an implicit cons\-tant depending only on $d,\beta, \lambda, \Lambda$, and where $D_x$ and $D_v$ denote $(-\Delta_x)^{\frac12}$ and $(-\Delta_v)^{\frac12}$, respectively. In addition, the solution satisfies the energy equality: for any $s,t \in \R$ with $s<t$, 
\begin{align} \label{e:intro:energy}
\nonumber \|f(t)\|^2_{\L^2_{x,v}} + 2\Re& \int_{s}^t\int_{\R^d} a_{\tau,x}({f},{f})\dx\dd \tau = 
\\ \|f(s)\|^2_{\L^2_{x,v}}  &+ 2\Re \int_{s}^t\bigg(\int_{\R^d} \angle {S_{1}}{f} \dx + \int_{\R^d} \angle {S_{2}}{f} \dv + \iint_{\R^{2d}} {S_{3}}\overline f \dx \dv \bigg)\dd \tau.
 \end{align} 
 
\end{thm}
\begin{rem}\label{rem:commentmainthm} Let us explain the statement. 
 The notation $\C^{}_{0}(\R;X)$ is the space of continuous functions on $\R$ taking values in $X$ with zero limits at $\pm\infty$. Note that since the solution $f(t)\in \L^2_{x,v}$ at all times, the expressions $D_{v}^{\beta} f$ and $D_{\vphantom{D_{v}}x}^{\frac{\beta}{2\beta+1}} f$ make sense, so that $f\in \H_{v}^{\beta}$ and $f\in \H_{ \vphantom{\H_{v}} x}^{\frac{\beta}{2\beta+1}}$, at least almost everywhere in the other variables. But the point of this formulation is that estimates are homogeneous. In \eqref{e:intro:energy}, the bracket $\angle {S_{1}}{f}$ denotes the extension of the sesquilinear duality on $\L^2_{v}$ to $\Hdot^{-\beta}_{v} \times \H^{\beta}_{v}$, giving an integrable function of $(t,x)$, and the second one $\angle {S_{2}}{f}$ refers to the extension of the sesquilinear duality on $\L^2_{x}$ to $\Hdot_{\vphantom{\L^2_{x}} x}^{-\frac{\beta}{2\beta+1}}\times \H_{\vphantom{\L^2_{x}} x}^{\frac{\beta}{2\beta+1}}$, giving an integrable function of $(t,v)$. Integrals are in the sense of Lebesgue. 
\end{rem}
\begin{rem}
As the problem is linear, all sources that belong to any interpolation space that can be obtained from the above three ones give rise to a weak solution. We shall not attempt to write such statements. 

Note that the source terms $S_{1}$ have been considered in the literature (see below), but having the source terms $S_{2}$ and $S_{3}$ is new.  
\end{rem}
\begin{rem} 
We also consider the kinetic Cauchy problem on half-infinite time intervals and construct solutions with homogeneous estimates; see Theorem~\ref{thm:homCP}. 
\end{rem}
\begin{rem} 
 We extend our approach to inhomogeneous spaces and construct solutions for the kinetic Cauchy problem in the corresponding inhomogeneous framework on finite time intervals, under weaker ellipticity conditions \eqref{e:ellip-upper-inhom} and \eqref{eq:weakelliptic-bis}, see Theorem~\ref{thm:CP0T}. 
\end{rem}

\begin{rem}
\label{rem:backward}
 The same statements and remarks hold for the backward equations corresponding to the operator in \eqref{eq:weaksolback}.
\end{rem}

Keeping in mind Remark~\ref{rem:backward}, let us mention a corollary of Theorem~\ref{thm:HomExUn}, that
concerns the boundedness of the following two operators
 \begin{equation}
\label{eq:KAdefintro}
\cK_{\cA}^\pm=(\pm(\partial_{t}+v\cdot \nabla_{x})+\cA)^{-1}
\end{equation}
 defined as follows: For $S\in \cD'(\Omega)$, $\cK_{\cA}^\pm S$ is the unique weak solution in $\Udot^\beta$, whenever it exists, to $\pm(\partial_{t}+v\cdot \nabla_{x})f+\cA f=S$ in the sense
 of Definition~\ref{defn:weaksol}, respectively. Remark that as $\cK_{\cA}^\pm S$ satisfy an equation and $\cA\cK_{\cA}^\pm S\in \L^2_{t,x}\Hdot^{-\beta}_{v}$, the transport term $(\partial_{t}+v\cdot \nabla_{x})\cK_{\cA}^\pm S$ is a distribution (tempered if $S\in \cS'(\Omega)$). This motivates the introduction of a kinetic space with only conditions on the transport and regularity with respect to the velocity.
\begin{cor}[Boundedness]
\label{cor:boundedness} Let $\beta>0$. Let 
\begin{align*}
 \Zdot^{\beta}&:= \L^2_{t,x}\Hdot^{-\beta}_{v} + \L^2_{t,v}\Hdot_{\vphantom{t,v} x}^{-\frac{\beta}{2\beta+1}}
+ \L^1_{t}\L^2_{x,v},
\\
\Ydot^\beta&:= \L^2_{t,x}\Hdot^{\beta}_{v} \cap \L^2_{t,v}\Hdot_{\vphantom{t,v} x}^{\frac{\beta}{2\beta+1}} \cap \C^{}_{0}(\R^{}_{t}\, ; \L^2_{x,v}),\\
\cLdot^\beta_{\beta}&:=\{f\in \Udot^\beta\, ;\, (\partial_t + v \cdot \nabla_x)f \in \Zdot^{\beta}\},
\end{align*}
and equip them with the natural norms. 
Under the assumptions of Theorem~\ref{thm:HomExUn} on $\cA$, the maps
\begin{align}
 \label{eq:KAintro}
  \cK_{\cA}^\pm: \Zdot^{\beta}\to \Ydot^\beta
\end{align}
 are continuous, and 
\[ \pm(\partial_{t}+v\cdot\nabla_{x}) + \cA : \cLdot^\beta_{\beta} \to \Zdot^{\beta}\]
are isomorphisms. In particular, 
 \begin{equation}
\label{eq:maxreg1}
\|\pm(\partial_{t}+v\cdot \nabla_{x})f+\cA f \|_{\Zdot^{\beta}} \sim 
 \|f\|_{\Udot^\beta}+ \|(\partial_{t}+v\cdot \nabla_{x})f \|_{\Zdot^{\beta}}, 
 \qquad f\in \cLdot^\beta_{\beta}
\end{equation}
 and, 
 \begin{equation}
\label{eq:maxreg2}
\| \cK_{\cA}^\pm S\|_{\Ydot^\beta} \lesssim \|S \|_{\Zdot^{\beta}} \sim 
 \| \cK_{\cA}^\pm S\|_{\Udot^\beta}+ \|(\partial_{t}+v\cdot \nabla_{x})\cK_{\cA}^\pm S \|_{\Zdot^{\beta}}, \qquad S\in \Zdot^{\beta}.
\end{equation}
The implicit constants depend only on $\beta, d, \lambda, \Lambda$.
\end{cor}

\begin{rem} 
 Note that as we have not defined homogeneous Sobolev spaces with large exponents, it would be more rigorous to write $\Ydot^\beta= 
 \Udot^\beta\cap \L^2_{t,v}\Hdot_{\vphantom{t,v} x}^{\frac{\beta}{2\beta+1}} \cap \C^{}_{0}(\R^{}_{t}\, ; \L^2_{x,v})$ but this amounts to the same in terms of norms (as $\beta/(2\beta+1)<1/2\le d/2$, $\Hdot_{\vphantom{t,v} x}^{\frac{\beta}{2\beta+1}} $ is well-defined).  We make the abuse of notation for clarity. This space is a closed subspace of $\L^2_{t,x}\Hdot^{\beta}_{v} \cap \L^2_{t,v}\Hdot_{\vphantom{t,v} x}^{\frac{\beta}{2\beta+1}} \cap \L^\infty_{\vphantom{t,v} t}\L^2_{x,v}$, the dual space to $\Zdot^{\beta}$ for the duality extending the $\L^2_{t,x,v}$ duality. 
 \end{rem}

\begin{rem} 
The formal adjoint of $\cK_{\cA}^+$ is $(\cK_{\cA}^+)^*= (-(\partial_{t}+v\cdot \nabla_{x})+\cA^*)^{-1} =\cK_{\cA^*}^-$.
\end{rem}

There is a particularly important special case which is, in fact, proved along the way. 

\begin{cor}[The fractional Kolmogorov equation]\label{cor:equality} Let $\beta>0$ and assume $\cA=(-\Delta_{v})^\beta\ $\footnote{Sstrictly speaking we should write $I_{t,x}\otimes(-\Delta_{v})^\beta $ with $I_{t,x}$ being the identity on $\L^2_{t,x}$.}. For $S\in \Zdot^{\beta}$, the weak solutions $\cK_{(-\Delta_{v})^\beta}^\pm S$ are the Kolmogorov solutions $($noted $\cK_{\beta}^\pm S)$ given by the fundamental solution for the forward and backward operators $\pm (\partial_{t}+v\cdot\nabla_{x}) +(-\Delta_{v})^{\beta}$.
 \end{cor}
 
 Of course, the action of the fundamental solution is meant through an extension from a dense class of sources for which the usual formula via the Fourier transform can be used.

\begin{rem} 
Let us make the following comment, which illustrates our strategy. The story begins with the model operators studied by Kolmogorov, and we only have access to Kolmogorov solutions $\cK_{\beta}^\pm S$. We first establish (new) bounds for them, that is, the bounds \eqref{eq:KAintro} for $\cK_{\beta}^\pm$. Next, we prove a novel uniqueness statement for each equation $\pm (\partial_{t}+v\cdot\nabla_{x})f +(-\Delta_{v})^{\beta}f=0$ when $f\in \Udot^\beta$ that implies Lemma~\ref{lem:distributionalsolutions}. These two facts already show Corollary~\ref{cor:equality}. But we can say more: they imply that $\cK_{\beta}^\pm$ is an isomorphism between the source space and the kinetic space, materialised by the equivalences in \eqref{eq:maxreg1} and \eqref{eq:maxreg2}.  This isomorphism property is at the heart of kinetic embeddings, yielding in particular a posteriori continuity valued in $\L^2_{x,v}$ of weak solutions for rough $\cA$ and that are eventually injected into the proof of Theorem~\ref{thm:HomExUn}. 
\end{rem}

\subsection{Kinetic embedding and transfer of regularity}

As just pointed out, our main tool in establishing the above existence and uniqueness statement is the following result that is independent of any equation. This result is both reminiscent of Bouchut's and H\"ormander's transfer of regularity results in \cite{MR1949176,MR0222474} (only proved for weak solutions in the local case there) and Lions' embedding theorem for time continuity \cite{Lions} (used essentially for parabolic equations, which is sometimes called the homogeneous case in kinetic theory as it means no dependence in the position variable). The fractional setting is new, and we also treat a broader class of source terms than has previously been considered in the literature.

\begin{thm}[Kinetic embedding and transfer of regularity]\label{thm:optimalregularity}
 Let $\beta >0$. Let $f\in \cD'(\Omega)$ be such that $f\in \Udot^\beta$
 and $(\partial_{t}+v\cdot\nabla_{x})f =S_{1}+S_{2}+S_{3}$, 
 with $S_1 \in \L^2_{t,x}\Hdot^{-\beta}_{\vphantom{t,x} v}$, $S_2 \in \L^2_{t,v}\Hdot_{\vphantom{t,x} x}^{-\frac{\beta}{2\beta+1}}$ and $S_3 \in \L^1_{\vphantom{t,x} t}\L^2_{\vphantom{t,x} x,v}$, that is $f\in \cLdot^\beta_{\beta}$. Then, we have
\begin{enumerate}
\item $f\in \C^{}_{0}(\R^{}_{t}\, ; \L^2_{x,v} )$ and 
$f\in \L^2_{t,v}\Hdot^{\frac{\beta}{2\beta+1}}_{\vphantom{t,v}x}$
with 
\begin{multline*}
\|D_{x \vphantom{D_{v}}}^{\frac{\beta}{2\beta+1}} f \|_{\L^2_{t,x,v}}+ \sup_{t\in \R} \|f(t)\|_{\L^2_{x,v}} \\\lesssim_{d,\beta} \|D_{v}^\beta f\|_{\L^2_{\vphantom{t,x,v} t,x,v}}+ \|D_{v}^{-\beta}S_1 \|_{\L^2_{t,x,v}}+ \|D_{\vphantom{D_{v}} x}^{-\frac{\beta}{2\beta+1}}S_2 \|_{\L^2_{t,x,v}} + \|S_3 \|_{\L^1_{\vphantom{t,x} t}\L^2_{\vphantom{t,x,v} x,v}}.
\end{multline*}
\item The map $t\mapsto\|f(t)\|_{\L^2_{x,v}}^2$ is absolutely continuous on $\R$, and for a.e. $t\in \R$,
\begin{equation}\label{e:abs-cont}
\frac{\mathrm{d} }{\mathrm{d}t }\|f(t)\|_{\L^2_{x,v}}^2 =  2\Re \bigg(\int_{\R^d} \angle {S_{1}}{f} \dx + \int_{\R^d} \angle {S_{2}}{f} \dv + \iint_{\R^{2d}} {S_{3}}\overline f \, \dx\dv \bigg).
\end{equation}
\end{enumerate}

\end{thm}
\begin{rem} The same comments as in Remark~\ref{rem:commentmainthm} apply.
 \end{rem}
 
\begin{rem} Note that estimates are invariant under the kinetic scaling $(t,x,v)\mapsto (\delta^{2\beta} t, \delta^{2\beta+1} x, \delta v)$.
\end{rem}
\begin{rem} In the local case $\beta=1$, it is possible to replace the assumption $f\in \Udot^1$ by $f\in \L^2_{t,x}\Wdot^{1,2}_v$ with $\|\nabla_{v}f\|_{\L^2_{t,x,v}}< \infty$. See Section~\ref{sec:localcase}. 
\end{rem}
\begin{rem}
 We shed light on the fact that we do not assume \textit{a priori} control on $f(t) \in \L^2_{x,v}$ at any time when $\beta<d/2$, which may seem surprising. When $\beta\ge d/2$, we make this assumption only qualitatively (one could even take a much weaker one), and it does not appear in the estimate.
 \end{rem}
\begin{rem}
We allow the free transport of $f$ to belong to a larger space than just $\L^2_{t,x}\Hdot^{-\beta}_{\vphantom{t,x} v}$. 
\end{rem}
\begin{rem}
 Not only do we obtain continuity of $t\mapsto f(t)$ in $ \L^2_{x,v}$ but we also get absolute continuity of $t\mapsto\|f(t)\|_{\L^2_{x,v}}^2$ and zero limit at infinity.  Both facts are crucial when dealing with weak solutions. 
 \end{rem}

\begin{rem}
 Theorem~\ref{thm:optimalregularity} is a special case of a more general result, see Theorem~\ref{thm:optimalregularitygeneralisation}, which itself is a consequence of embeddings obtained for a scale of kinetic spaces, see Theorem~\ref{thm:homkinspace}. Such generalisations could be useful for further developments in the study weak solutions. As their proofs follow a similar pattern, we shall prove them all in this article.
\end{rem}

\begin{rem}
  Theorem~\ref{thm:optimalregularity} implies the sharp kinetic Sobolev embedding as
  \begin{equation*}
    \|f\|_{\L^{2\kappa}_{\vphantom{t,x,v} t,x,v}} \lesssim_{d,\beta} \|D_{v}^\beta f\|_{\L^2_{\vphantom{t,x,v} t,x,v}}+ \|D_{v}^{-\beta}S_1 \|_{\L^2_{t,x,v}}+ \|D_{\vphantom{D_{v}} x}^{-\frac{\beta}{2\beta+1}}S_2 \|_{\L^2_{t,x,v}} + \|S_3 \|_{\L^1_{\vphantom{t,x} t}\L^2_{\vphantom{t,x,v} x,v}}
  \end{equation*}
  with the optimal gain of integrability $\kappa = 1+\frac{\beta}{(\beta+1)d}$. This is a direct consequence of the Sobolev embedding for the anisotropic space $  \L^2_{x}\H_v^\beta \cap \L^2_{v}\H_{\vphantom{\H_v} x}^{\frac{\beta}{2\beta+1}}$ at every fixed time.  Integrating this estimate and combining it with the $\L^\infty_t\L^2_{x,v}$ estimate yields the claim. This strategy was also used in \cite{MR4049224,MR3923847} in order to prove a gain of integrability for weak subsolutions. We refer to \cite{dmnz_kinegeom_2025} for an alternative proof in the case $\beta = 1$ based on kinetic trajectories. 
\end{rem}

\subsection{Highlights and observations} 

 The theory of kinetic equations related to this work has been much developed, and we shall provide a short review. We wish to highlight some points in our contribution and make a few observations on tools that we do or do not use. 

The Hilbertian framework proposed by J.-L.~Lions to construct weak solutions had already been used for some particular kinetic equations. Our approach shows that it applies in a general setting, thanks to the complete description of the kinetic embedding and transfer of regularity for scales of adapted kinetic spaces: it allows us to obtain uniqueness under minimal assumptions and also to obtain existence with more source terms. 
 
 A second highlight concerns the strategy of proof of the kinetic embedding and transfer of regularity: it boils down to estimates and appropriate uniqueness of distributional solutions of the constant coefficients Kolmogorov equation $(\partial_{t}+v\cdot\nabla_{x})f +(-\Delta_{v})^\beta f=S$ and its adjoint. This also shows the isomorphism properties of the integral Kolmogorov operators. This is interesting in its own right and new at this level of generality. 
 
 As for the techniques of proofs, they only use partial Fourier transformation for the position and velocity variables and the Galilean change of variables. Using homogeneous norms also provides us with a very convenient setup in which estimates are quite elementary and do not require subtle decompositions in Fourier space.
 
We also mention that there are common tools in the field that we do not need. We neither use the Galilean group law associated with the fields $(\partial_{t}+v\cdot\nabla_{x}), \nabla_{v}$ nor rely on commutator techniques \`a la H\"ormander in the context of hypoelliptic equations.
Our final observation is that we do not use the recently developed local regularity theory of weak solutions, such as the De Giorgi-Nash-Moser estimates. In other words, the construction of weak solutions does not require knowing that weak solutions have local regularity properties. 

\subsection{From Lions embedding to kinetic embeddings.}

We think it could be useful for the reader to focus on the genesis of our kinetic embedding and the subtlety of using homogeneous spaces by describing the parabolic case, where there is no transport and no position variable $x$.
 
 The starting point of J.-L.~Lions in his 1957 article \cite{Lions} is his famous abstract embedding. Three Hilbert spaces $(\V,\H,\V')$ form a \emph{Gelfand triple} if there are continuous and dense inclusions $\V\hookrightarrow \H\hookrightarrow \V'$. With such a triple, Lions established that if $f\in \L^2(0,T\, ; \, \V)$ and $f \in \H^1(0,T\, ; \, \V')$ then $f\in \C([0,T]\, ; \, \H)$ and $t\mapsto \|f(t)\|_{\H}^2$ is absolutely continuous. His proof uses the \textit{a priori} knowledge that $f(t)$ exists in $\H$ almost everywhere, approximation procedures that are compatible with the topologies of the three spaces $\V,\H,\V'$, and integration theory for vector-valued functions.

For example, with $\V=\H^1(\R^d)$ the classical Sobolev space, this shows that if $f\in \L^2(0,T\, ; \, \H^1(\R^d))$ and $f \in \H^1(0,T\, ; \, \H^{-1}(\R^d))$ then $f\in \C([0,T]\, ; \, \L^2(\R^d))$. 
Given the natural scaling, the embedding could be envisioned with $\H^1, \H^{-1}$ replaced by their homogeneous versions $\Hdot^1, \Hdot^{-1}$: however, it is not true when $T<\infty$, for the simple reason that both embeddings $\Hdot^1(\R^d) \hookrightarrow \L^2(\R^d) \hookrightarrow \Hdot^{-1}(\R^d)$ fail. Indeed, take any $g \in \Hdot^1(\R^d)$ which is neither an element of $\L^2(\R^d)$ nor an element of $\Hdot^{-1}(\R^d)$ and consider $f(t) \equiv g \in \L^2(0,1\, ; \Hdot^1(\R^d)) \cap \Hdot^1(0,1\, ; \, \Hdot^{-1}(\R^d))$ which cannot be an element of $\C([0,1]\, ; \, \L^2(\R^d))$. In other words, there is no possible temporal trace theory with homogeneous spaces on finite time intervals.

 This failure is not an obstacle when $T=\infty$. Indeed, it is proved by the first author with S.~Monniaux and P.~Portal in \cite{amp} that the embedding for homogeneous spaces is true provided one works on an infinite interval.  The strategy is to prove first a uniqueness result for the backward heat equation so that any distribution $f$ that belongs to $\L^2(0,\infty\, ; \Hdot^1(\R^d)) \cap \Hdot^1(0,\infty\, ; \, \Hdot^{-1}(\R^d))$ is represented as a Duhamel integral for the backward heat equation (up to a constant, depending on how homogeneous Sobolev spaces are defined) with source term $\partial_{t}f +\Delta f$. As this integral is shown to belong to $ \C^{}_{0}([0,\infty)\, ; \, \L^2(\R^d))$, this yields the desired conclusion for the embedding. This can also be done on $\R$.

With another strategy using half-time derivatives, it was shown by the first author and M. Egert in \cite{ae} that the condition $f \in \Hdot^1(\R\, ; \, \Hdot^{-1}(\R^d))$, that is, $\partial_{t}f\in \L^2(\R\, ; \, \Hdot^{-1}(\R^d))$, can even be relaxed to larger spaces, with the aim of allowing more source terms in equations. Again, what was new was to work with homogeneous spaces.

Our proof of continuity in time in Theorem~\ref{thm:optimalregularity} will rely on using the strategy of \cite{amp} in a kinetic context: working on infinite intervals is required to obtain the scale-invariant estimates; and we allow as general source terms as possible, in the spirit of \cite{ae}. We also make it work for fractional diffusion, which is new even in the homogeneous/parabolic case.

\subsection{Review of literature}

This work is concerned with the study of kinetic equations with both local and integral diffusion. It is related to the existing literature when $0< \beta\le 1$ stemming from various trends of research. 

\subsubsection{Kolmogorov equations}

The Kolmogorov equation was proposed by A.~Kolmogorov in 1934 in \cite{MR1503147}. He derived explicitly the fundamental solution of $(\partial_t + v \cdot \nabla_x ) - \Delta_v$, which allows one to observe the regularising effect of the, at first glance, degenerate diffusion equation. This was the starting point of H\"ormander's hypoellipticity theory \cite{MR0222474}. One of the examples in his paper gave birth to a trend of research devoted to a class of equations now referred to as ultra-parabolic equations or equations of Kolmogorov type. It was launched by E.~Lanconelli and S.~Polidoro \cite{MR1289901}. In this article, they first considered such equations with constant coefficients, exhibited a Lie group structure associated to the equation (extending the Galilean invariance of the original Kolmogorov equation) and established a Harnack inequality. 

\subsubsection{Construction of weak solutions}
Existence (and uniqueness) of weak solutions has been addressed in the literature devoted to kinetic theory in various settings. 

 P.~Degond used Lions' existence theorem to treat the kinetic Cauchy problem on $[0,T]$ for the Fokker--Planck equation ($\beta=1$) with constant diffusion coefficients and a bounded drift term with bounded divergence and source terms in $ \L^2_{t,x}\H^{-1}_v$ \cite{MR875086}. He attributes the idea to a previous work of S.~Baouendi and P.~Grisvard for the related stationary problem in one dimension \cite{MR252817}. 
   J.~A.~Carrillo used the same idea in bounded domains \cite{MR1634851}. More recently,  D.~Albritton, S.~Armstrong, J.~C.~Mourrat and M.~Novack \cite{albritton2021variational} used this idea in a context of equations with confinement, and they worked with a Gaussian weight $\gamma$ in the velocity variable. More precisely, they defined a kinetic space made of functions $f \in \L^2_{t,x}\H^1_v(\gamma \mathrm{d} v)$ such that $(\partial_t+ v \cdot \nabla_x)f \in \L^2_{t,x}\H^{-1}_v(\gamma \mathrm{d} v)$. It is proven that these functions are continuous in time with values in $\L^2_{x,v}(\gamma \mathrm{d} v)$. They also used another route via a minimisation procedure. 

K.~Nystr\"om and M.~Litsg\aa rd considered the Cauchy-Dirichlet problem for the Fokker--Planck equation ($\beta=1$) in possibly unbounded domains in \cite{MR4312909}. Weak solutions are constructed on the full space in the case where coefficients are allowed to be rough using a minimisation procedure of a convex functional. For the uniqueness result, coefficients are allowed to be rough only in a compact set; they are assumed to be constant in the rest of the domain. See also the work of F.~Anceschi and A.~Rebucci in  \cite{anceschi2023obstacle} for a follow-up on an obstacle problem.

In \cite{MR4350284}, the third author together with R.~Zacher 
introduced and developed a notion called \textit{kinetic maximal $\L^2$-regularity}. This concept is applied to distributional solutions 
of the Kolmogorov operators with constant diffusion coefficients when $ 0 < \beta \le 1$. The Cauchy problem is studied on $[0,T]$ and in inhomogeneous spaces in $(x,v)$-variables on the full space.
Continuity in time is proven for the space of solutions, and the temporal trace space is characterised in terms of an anisotropic Sobolev space.
There are similarities and differences with our approach of kinetic embedding that we shall explain in the last section. 

\subsubsection{Transfer of regularity}

The regularising effect of equations combining free transport $(\partial_t+ v\cdot \nabla_x)$ and diffusion $-\Delta_{v}$ in the velocity variable $v$ was first exhibited by A.~Kolmogorov in \cite{MR1503147}. It can be seen from the explicit computation of the fundamental solution. This transfer of regularity has been quantified in H\"ormander's work, although only in a crude way for the regularity exponents. 

An important step in the study of the regularity of kinetic equations is the observation that the means of solutions in the velocity variable $v$ are more regular than the solutions themselves.
This was observed in particular by F.~Golse, B.~Perthame and R.~Sentis \cite{MR0808622} (see also \cite{MR0923047}). Results of this type are nowadays called \textit{averaging lemmas}. The theory was further developed by the study of the regularity of the solution itself, exhibiting the \textit{transfer of regularity} property of the free transport. In this perspective, F. Bouchut's article \cite{MR1949176} was influential as it was among the first to provide optimal and global estimates for this phenomenon, which is sometimes also called \textit{kinetic regularisation}. We emphasise again that the results of \cite{MR1949176} apply to weak solutions only when $\beta = 1$ and in inhomogeneous spaces. 

This regularising effect can also be seen for the weak solutions constructed in \cite{albritton2021variational} by working with a commutator method inspired by the work of H\"ormander. Sharp estimates are proven in adapted Besov spaces.

\subsubsection{Regularity of weak solutions}

Although we do not use the local regularity of weak solutions at all, recent developments in such topics have been crucial to our understanding and definition of a weak solution. In the case of rough and real coefficients satisfying an ellipticity condition, but without any smoothness assumptions on coefficients, a local $\L^\infty$ bound for weak solutions was obtained through Moser's iterative method by A.~Pascucci and S.~Polidoro \cite{MR2068847}. Then W.~Wang and L.~Zhang \cite{MR2514358} proved a local H\"older regularity thanks to a weak Poincar\'e inequality satisfied by weak subsolutions. In both works, weak solutions $f$ together with $\nabla_v f$ and $(\partial_t + v \cdot \nabla_x)f$ are assumed to be locally square integrable. 
It makes sense for $f$ and $\nabla_v f$, because of the natural energy estimate, but the free transport term has no reason to be square integrable in the case of rough diffusion coefficients. Another proof of the local H\"older estimate was proposed in \cite{MR3923847} by F.~Golse, the second author, C.~Mouhot and A.~Vasseur, where a Harnack inequality is established. This latter work focuses on \eqref{eq:weaksol} with \(\cA f = -\nabla_v \cdot ( \mathbf{A} \nabla_v f)\) and weak solutions are assumed to be locally in $\L^\infty_t \L^2_{x,v} \cap \L^2_{t,x} \H^1_v$ and such that $(\partial_t + v \cdot \nabla_x) f
\in \L^2_{t,x}\H^{-1}_v$. See also the works of J.~Guerand with the second author and with C.~Mouhot, and also the work of the third author with R. Zacher \cite{MR4653756,MR4453413,niebel2023kinetic} for some further developments. An explanation on how to prove estimates for weak (sub-) supersolutions belonging to the space $\L^2_{\loc,t,x}\L^2_v \cap \L^2_{t,x} \Hdot^1_v$ can be found in the recent work of H.~Dietert, C.~Mouhot, the last author and R.~Zacher \cite{dmnz_kinegeom_2025}.

We conclude this review of the literature by mentioning that we focused mainly on literature concerning weak solutions. For an overview of the literature on strong $\L^p$ solutions or strong H\"older solutions, we refer to \cite{MR4336467,MR4299838} and \cite{MR4275241,MR4229202} and the references therein. See also the pioneering work of L. Rothschild and E. Stein \cite{MR0436223}. 

\subsection{Future developments} 

We announce some subsequent works. We have used the material of this article to provide a construction of the fundamental solution for non-constant diffusion coefficients, see \cite{AIN2}. In the local case, we have also obtained sharp decay estimates. We shall establish self-improvement of the regularity properties of local weak solutions with improved source terms, both in the local and non-local cases. The authors also plan to treat $\L^p$ estimates in the spirit of kinetic maximal $\L^p$-regularity as in \cite{MR4336467} for weak solutions of the (fractional) Kolmogorov equation with constant (and even possibly rough) coefficients. We will also develop further regularity theory along the characteristics with some new variational methods, with consequences for weak solutions.
Finally, the setup here will be adapted to allow lower-order terms with unbounded coefficients. 

Among possible other developments is the construction of weak solutions on bounded domains, but this would require other methods, as we use the Fourier transform.

It is an interesting question how far one can stretch the methods of the article. While it is natural to apply them to operators of Kolmogorov type as considered in \cite{MR1289901}, another problem is to tackle H\"ormander sum of squares with a drift term. 

\subsection{Organisation of the article}

In Section~\ref{s:homog-kolmo}, homogeneous kinetic spaces are introduced, and estimates for Kolmogorov operators are obtained. Section~\ref{s:uniqueness-embeddings} concerns kinetic embeddings, which rely on an appropriate uniqueness result for the equation associated with Kolmogorov operators. Weak solutions for equations with rough coefficients are constructed in Section~\ref{s:weak-rough}. The kinetic Cauchy problem is treated in Section~\ref{s:cauchy}. In particular, inhomogeneous kinetic spaces are introduced, and they are used to solve the kinetic Cauchy problem on finite time intervals. In Section~\ref{sec:localcase}, we focus specifically on the case of local diffusion operators ($\beta=1$). 
Section~\ref{sec:kinmax} contains an alternative definition of homogeneous kinetic spaces that permits us to relax the constraint $\gamma \le 2\beta$ and allows us to compare our findings with previously known kinetic embeddings. We also explain in this final section how to lift the limitation $\beta < d/2$.

\subsection{Notation}
The full space $\R \times \R^d \times \R^d$ is denoted by $\Omega$ while $\Omega_+$ denotes $(0,\infty) \times \R^d \times \R^d$.
The set of distributions in $\Omega$ (resp. tempered distributions in $\Omega$) is denoted by $\cD' (\Omega)$ (resp. $\cS'(\Omega)$). 
The operators $(-\Delta_{v})^{1/2}$ and $(-\Delta_x)^{1/2}$ are denoted by $D_v$ and $D_x$ respectively.
We work with functions of $(t,x,v)$. As we use mixed spaces, we shall put variables in the index, \textit{e.g.} $\L^2_{t,x} \Hdot^\beta_{\vphantom{t,x} v}$. 
The Galilean change of variables is denoted by $\Gamma f (t,x,v)= f(t,x+tv,v)$. 
Given a Banach space $\B$, $\C^{}_{0}(\R\, ; \, \B)$ denotes the space of continuous functions valued into $\B$, with limit zero at $\pm\infty$.
Given two Banach spaces $X$ and $Y$, $X \hookrightarrow Y$ means that $X$ is continuously embedded in $Y$.
We use the same bracket notation for the duality pairing in different contexts with subsequent explanations, hoping this does not create any confusion.

\section{Homogeneous kinetic spaces and Kolmogorov operators}
\label{s:homog-kolmo}

\subsection{Main functional spaces}

Before defining rigorously the homogeneous kinetic spaces that we will work with, we give here a short list and state the main result for the reader's convenience. The ranges of $\gamma,\beta$ will be specified below.
\begin{itemize}
\item The space $\Hdot^{{\alpha}}$ is the homogeneous Sobolev space on $\R^d$.
\item The space $\Fdot^\beta$ is $\L^2_{t,x} \Hdot^\beta_v$ and $\Udot^\beta$ is $ \L^2_{\loc, t}\L^2_{\vphantom{\loc, t} x,v} \cap \L^2_{t,x}\Hdot^{\beta}_{v}$.
\item The space $ \Xdot^{\gamma}_{\beta}$ equals $\L^2_{x}\Hdot_{\vphantom{\beta} v}^\gamma \cap \L^2_{\vphantom{\beta} v}\Hdot_{\vphantom{\beta} x}^{\frac{\gamma}{2\beta+1}}$ when $\gamma \ge 0$ and $\Xdot^{\gamma}_{\beta}= \L^2_{x}\Hdot^{\gamma}_{\vphantom{\beta} v} + \L^2_{v}\dot\H_{\vphantom{\beta} x}^{\frac{\gamma}{2\beta+1}}$ if $\gamma\le 0$. 
\item The space $\cFdot^\gamma_{\beta}$ is made of $f \in \L^2_{t,x}\Hdot^{\gamma}_{v}$ such that $(\partial_t + v \cdot \nabla_x)f \in \L^2_{t,x} \Hdot^{\gamma -2\beta}_{\vphantom{t,x} v}$.
\item The space $\cGdot^\gamma_{\beta}$ is made of  $f \in \L^2_{t,x}\Hdot^{\gamma}_{v}$ such that $(\partial_t + v \cdot \nabla_x)f \in \L^2_{t} \Xdot^{\gamma -2\beta}_\beta$.
\item The space $\cLdot^\gamma_{\beta}$ is made of  $f \in \L^2_{t,x}\Hdot^{\gamma}_{v}$ such that $(\partial_t + v \cdot \nabla_x)f \in \L^2_{t} \Xdot^{\gamma -2\beta}_\beta+ \L^1_{t} \Xdot^{\gamma-\beta}_{\beta} $.
\end{itemize}

We use calligraphic letters whenever we are concerned with kinetic spaces, that is, spaces which also describe the regularity in terms of the kinetic operator $\partial_t +v \cdot \nabla_x$.

 \begin{rem}
 	The calligraphic scales are adapted to the fractional Laplacian of order $\beta$: formally, $(-\Delta_{v})^\beta$ maps $\cLdot^\gamma_{\beta}$ into $ \L^2_{t,x}\Hdot^{\gamma-2\beta}_{v} \subset \L^2_{t} \Xdot^{\gamma -2\beta}_\beta$ when $\gamma-2\beta\le 0$. The spaces $\cFdot^\beta_{\beta}$ and $ \cLdot^\beta_{\beta}$ are the most important ones here. Having full scales is important for the potential extension of the theory (and proofs are not much different anyway). The spaces $\cGdot^\gamma_{\beta}$ are also related to other kinetic spaces defined in Section~\ref{sec:kinmax}.  
 \end{rem}

\begin{thm}[Kinetic embeddings and transfers of regularity: homogeneous case]
\label{thm:homkinspace}
Assume $\beta>0$,  $0\le \gamma\le 2\beta$ with $\gamma <d/2$. 
\begin{enumerate} 
\item  $\cLdot^\gamma_{\beta} \hookrightarrow \L^2_{\vphantom{\beta} t,v}\Hdot_{\vphantom{\beta} x}^{\frac{\gamma}{2\beta+1}}$ with 
\(
\|D_{\vphantom{\beta}x}^{\frac{\gamma}{2\beta+1}} f\|_{\L^2_{t,x,v}} \lesssim_{d,\beta,\gamma} \|f\|_{\cLdot^\gamma_{\beta}}.
\)
\item $\cLdot^\gamma_{\beta} \hookrightarrow \C^{}_{0}(\R^{}_{t}\, ; \, \Xdot^{\gamma-\beta}_{\beta} )$ with
\(
\sup_{t\in \R} \|f(t)\|_{\Xdot^{\gamma-\beta}_{\beta}} \lesssim_{d,\beta,\gamma} \|f\|_{\cLdot^\gamma_{\beta}}.
\) 
\item There is a subspace of $\L^2_{t,x,v}$ $($even of $\cS(\Omega)$ if $d>2${}$)$  dense in all $\cFdot^\gamma_{\beta}, \cGdot^\gamma_{\beta}$ and $\cLdot^\gamma_{\beta}$.
\item The families of spaces $(\cFdot^\gamma_{\beta})_{\gamma}$ and $(\cGdot^\gamma_{\beta})_{\gamma}$ have the complex interpolation pro\-perty. 
\end{enumerate}
When $\gamma=\beta\ge d/2$,  $\mathrm{(i), (ii), (iii)}$ hold provided we define the kinetic space $\cLdot^\beta_{\beta}$ with $\Udot^\beta$ replacing $\L^2_{t,x}\Hdot^{\beta}_{v}$.
\end{thm}
\begin{rem}
It is not clear whether the complex interpolation property is satisfied by $(\cLdot^\gamma_{\beta})_{\gamma}$ or not.
\end{rem}

\begin{rem}
If we impose $f\in \Udot^\gamma$ in the definition of the kinetic spaces when $\gamma\ge d/2$ and $\gamma\ne \beta$, then we can keep (i), (ii), and (iii), but it is not clear whether they are Banach spaces, and we may lose property (iv).
\end{rem}

\subsection{Homogeneous kinetic spaces}

{We come to precise definitions.}

\subsubsection{Homogeneous Sobolev spaces $\Hdot^\alpha (\R^d)$.}

To define homogeneous Sobolev spaces, we use the Fourier transform on tempered distributions in $\R^d$, equipped with Euclidean norm $|x|$ and inner product $x\cdot y$. 

Assume first $0\le\alpha<d/2$.  We define the homogeneous Sobolev space $\Hdot^\alpha(\R^d)$ as
\begin{align}
 \label{eq:homSobbetapos}
 \Hdot^\alpha(\R^d)= \{f\in \cS'(\R^d)\, ; \, \exists g \in \L^2(\R^d),  \hat f = |\xi|^{-\alpha}\hat g\}
\end{align}
and we set $(-\Delta)^{\alpha/2}f= g$ and $\|f\|_{\Hdot^\alpha}=\|g\|_{2}$. The key point is not only $|\xi|^{-\alpha}\hat g \in \cS'(\R^d)$, it is also that  $|\xi|^{-\alpha}\hat g\in \L^1_{\loc}(\R^d)$, so that the elements of $\Hdot^\alpha(\R^d)$ have locally integrable Fourier transform. 

With this definition, $\Hdot^\alpha(\R^d)$ is a Hilbert space and $\cS(\R^d)\subset \Hdot^\alpha(\R^d)\subset \cS'(\R^d)$ with continuous and dense inclusions. Elements of $\Hdot^\alpha(\R^d)$ are in fact $\L^q$-functions for $q=q_{\alpha}= \frac{2d}{d-2\alpha}$ the Sobolev conjugate exponent for the Sobolev inequality $\|f\|_{q}\le C(d, \alpha)\|f\|_{\Hdot^\alpha} $ and $\Hdot^\alpha(\R^d)$ agrees with 
$ \{f\in  \L^q(\R^d)\, ; \, (-\Delta)^{\alpha/2}f \in \L^2(\R^d)\}$. For $\alpha=0$, $\Hdot^0(\R^d)=\L^2(\R^d)$.

\begin{rem}
 Literature on kinetic equations is mostly concerned with $\alpha\in (0,1]$. We observe that the requirement $0<\alpha < d/2$ is only to have a space of functions that is a Hilbert space.
 It forces $\alpha<1$ when $d=2$ and $\alpha<1/2$ when $d=1$. But we can also manage the case $\alpha\ge d/2$ differently. We refer the reader to forthcoming work \cite{AN} for further explanations. 
 The $d/2$ threshold never shows up when using inhomogeneous functional spaces, see Section~\ref{s:cauchy}.
\end{rem}
\begin{rem}
When $0<\alpha<1$, an equivalent norm on $\Hdot^\alpha(\R^d)$ is given by the Gagliardo or Beppo-Levi (semi-)norm (see \cite{MR2768550} for example), but we shall not use it. When $\alpha=1$, then our definition differs from the space $\Wdot^{1,2}(\R^d)=\{f\in \cD'(\R^d)\, ;\, \nabla f\in \L^2(\R^d)\}$ as this space contains constants. We will indicate how this changes the statement and proof of some of our results for kinetic equations, see Section~\ref{sec:localcase}. 
\end{rem}

Define for $\alpha<0$,
\[
 \Hdot^{\alpha}(\R^d)= \{f\in \cS'(\R^d)\, ; \, \exists g \in \L^2(\R^d),  \hat f = |\xi|^{-\alpha}\hat g\}
\]
with norm $\|f\|_{\Hdot^{\alpha}}=\|g\|_{2}$. This is also a Hilbert space, contained in $\cS'(\R^d)$ {and with elements having $\L^2_{\loc}(\R^d)$ Fourier transforms}.
When $-d/2 <\alpha<0$, $\Hdot^{\alpha}(\R^d)$ identifies to the dual of $\Hdot^{-\alpha}(\R^d)$ for the duality induced by the $\L^2(\R^d)$ inner product and we also have $\cS(\R^d)\subset \Hdot^{\alpha}(\R^d)$ as a dense subset. If $\alpha\le -d/2$, then Schwartz functions whose Fourier transforms are supported away from the origin form a dense set.

We eventually come to the proper interpretation of fractional powers of the Laplacian. If $f\in \Hdot^\alpha(\R^d)$, $\alpha\in \R$ and $\alpha<d/2$, and $\gamma\in \R$ with $\alpha-2\gamma<d/2$, then $(-\Delta)^\gamma f$ is the tempered distribution whose Fourier transform agrees with $|\xi|^{2\gamma-\alpha}\hat g$ where $\hat f=|\xi|^{-\alpha}\hat g$ and $g\in \L^2(\R^d)$. Thus, $(-\Delta)^\gamma f \in \Hdot^{\alpha-2\gamma}(\R^d)$ and has Fourier transform $|\xi|^{2\gamma} \hat f$.

\subsubsection{Homogeneous weak kinetic spaces}
\label{sec:homweakkineticspaces}

We work with time, position and velocity variables $(t,x,v)\in \Omega=\R\times\R^d\times \R^d$. As we use mixed spaces, we shall put the variables in the index, e.g. $\L^2_{t,x}\Hdot^{-\beta}_{\vphantom{t,x} v}$ for $\L^2(\R^{}_{t}\times \R_{x}^d\, ; \,\Hdot^{-\beta}(\R_{v}^d))$. Without indication, Lebesgue spaces are with respect to the Lebesgue measure. We also indicate the variables on differentiation operators. We recall that we use the notation $D_{v}$ for $(-\Delta_{v})^{1/2}$ and $D_{x}$ for $(-\Delta_{x})^{1/2}$. In what follows, $\beta>0$ is fixed. 
\bigskip

We introduce next for $\gamma \in [0,d/2)$, 
 \begin{align}
 \label{eq:dotXgammabeta}
 \Xdot^{\gamma}_{\beta}= \L^2_{\vphantom{\beta} x}\Hdot_{\vphantom{\beta} v}^\gamma \cap \L^2_{\vphantom{\beta} v}\Hdot_{\vphantom{\beta} x}^{\frac{\gamma}{2\beta+1}}
\end{align}
with Hilbertian norm $\|f\|_{\Xdot^{\gamma}_{\beta}}$ given by
\begin{align*}
 \|f\|_{\Xdot^{\gamma}_{\beta}}^2= \|f\|_{\L^2_{\vphantom{\beta} x}\Hdot^\gamma_{\vphantom{\beta} v}}^2 + \|f\|_{\L^2_{\vphantom{\beta} v}\Hdot_{\vphantom{\beta} x}^{\frac{\gamma}{2\beta+1}}}^2.
\end{align*}
If $\gamma<0$, we set 
\begin{align}
\label{eq:dotX-gammabeta}
 \Xdot^{\gamma}_{\beta}= \L^2_{\vphantom{\beta} x}\Hdot^{\gamma}_{\vphantom{\beta} v} + \L^2_{\vphantom{\beta} v}\dot\H_{\vphantom{\beta} x}^{\frac{\gamma}{2\beta+1}}
\end{align}
with Hilbertian norm $ \|f\|_{\Xdot^{\gamma}_{\beta}}$ given by
\begin{align*}
 \|f\|_{\Xdot^{\gamma}_{\beta}}^2= \inf_{f=f_{1}+f_{2}} \|f_{1}\|_{\L^2_{\vphantom{\beta} x}\Hdot^{\gamma}_{\vphantom{\beta} v}}^2 + \|f_{2}\|_{ \L^2_{\vphantom{\beta} v}\dot\H_{\vphantom{\beta} x}^{\frac{\gamma}{2\beta+1}}}^2.
\end{align*}
\begin{rem}
\label{rem:Xdotgammabeta}
For $\gamma<d/2$, $\Xdot^{\gamma}_{\beta}$ are Hilbert spaces. They are subspaces of $\cS'_{x,v}=\cS'(\R^{2d})$, and subspaces of $\L^2_{\loc,x,v} {=\L^2_{\loc}(\R^{2d})}$ when $\gamma\ge 0$. The subspace of $\cS_{x,v}=\cS(\R^{2d})$ of Schwartz functions whose Fourier transforms (in $\cS_{\varphi,\xi}$)  have compact support away from $\varphi=0$ and $\xi=0$ is a dense subspace for $-\infty < \gamma < d/2$.
\end{rem}
\begin{rem}
For $\gamma \in [0,d/2)$, the space $\Xdot^{-\gamma}_{\beta}$ identifies with the dual of $ \Xdot^{\gamma}_{\beta}$ for the duality extending the $\L^2_{x,v}$ inner product. 
\end{rem}

We introduce next for $\gamma<d/2$, 
 \begin{align}
 \label{eq:Fdotgammabeta}
 \cFdot^\gamma_{\beta}= \{f\in \cD'(\Omega)\, ; \, f \in \L^2_{t,x}\Hdot^{\gamma}_{v} \ \& \ (\partial_t + v \cdot \nabla_x)f \in \L^2_{t,x} \Hdot^{\gamma -2\beta}_{\vphantom{t,x} v}\}
\end{align}
with Hilbertian norm $ \|f\|_{\cFdot^\gamma_{\beta}}$ defined by
\begin{align}
 \label{eq:Fdotgammabetanorm}
 \|f\|_{\cFdot^\gamma_{\beta}}^2= \| D_{v}^{\gamma} f\|_{\L^2_{t,x,v}}^2 + \|(\partial_t + v \cdot \nabla_x)f\|_{\L^2_{t,x}\Hdot^{\gamma-2\beta}_{\vphantom{t,x} v}}^2,
 \end{align}
\begin{align}
 \label{eq:Kdotgammabeta}
 \cGdot^\gamma_{\beta}= \{ f \in \cD'(\Omega)\, ; \, f \in \L^2_{t,x}\Hdot^{\gamma}_{v} \ \& \ (\partial_t + v \cdot \nabla_x)f \in \L^2_{t} \Xdot^{\gamma -2\beta}_\beta\}
\end{align}
with Hilbertian norm $ \|f\|_{\cGdot^\gamma_{\beta}}$ defined by
\begin{align*}
 \|f\|_{\cGdot^\gamma_{\beta}}^2= \| D_{v}^{\gamma} f\|_{\L^2_{t,x,v}}^2 + \|(\partial_t + v \cdot \nabla_x)f\|_{\L^2_{t}\Xdot^{\gamma-2\beta}_\beta}^2, 
 \end{align*}
 and
 \begin{align}
 \label{eq:Ldotgammabeta}
 \cLdot^\gamma_{\beta}= \{ f \in \cD'(\Omega)\, ; \, f \in \L^2_{t,x}\Hdot^{\gamma}_{v} \ \& \ (\partial_t + v \cdot \nabla_x)f \in \L^2_{t} \Xdot^{\gamma -2\beta}_\beta+\L^1_{t} \Xdot^{\gamma-\beta}_{\beta}\}
\end{align}
with norm $ \|f\|_{\cLdot^\gamma_{\beta}}$ defined by
\begin{align*}
 \|f\|_{\cLdot^\gamma_{\beta}}= \| D_{v}^{\gamma} f\|_{\L^2_{t,x,v}} + \|(\partial_t + v \cdot \nabla_x)f\|_{\L^2_{t}\Xdot^{\gamma-2\beta}_\beta+\L^1_{t} \Xdot^{\gamma-\beta}_{\beta}}.
 \end{align*}
 
 The action of the field $\partial_t + v \cdot \nabla_x$ is taken in the sense of distributions on $\Omega$.
\begin{rem} \label{rem:inclusions}
 $ \cFdot^\gamma_{\beta}\subset \cGdot^\gamma_{\beta} \subset \cLdot^\gamma_{\beta} \subset \L^1_{\loc,t}\cS'_{x,v}$.
\end{rem}

Eventually, we can repeat \textit{verbatim} the definitions on $
\Omega_{+}= (0,\infty)\times \R^d \times \R^d$.

\subsection{Kolmogorov operators}
\label{sec:estimates}

\subsubsection{Galilean change of variables}

The Galilean change of variables is defined by
\[\Gamma f(t,x,v)=f(t,x+tv,v)\]
and also for fixed $t$,
\[\Gamma(t)f(x,v)= f(x+tv,v)\]
which are isometries on $\L^2_{t,x,v}$  and $\L^2_{x,v}$ respectively. Furthermore, $(\Gamma(t))_{t\in \R}$ is a strongly continuous group on $\L^2_{x,v}$ \cite[Lemma~5.2]{MR4350284}. These two properties imply that the mapping $\Gamma \colon \C^{}_{0}(\R^{}_{t}\, ; \, \L^2_{x,v}) \to \C^{}_{0}(\R^{}_{t}\, ; \, \L^2_{x,v})$ is continuous. 

\subsubsection{Partial Fourier transform}

Our definition of (partial) Fourier transform in the variables $(\varphi,\xi)$ dual to $(x,v)$ is
\[
\widehat f(t,\varphi,\xi)=\iint_{\R^d\times \R^d} e^{-i(x\cdot \varphi+v\cdot\xi)}f(t,x,v) \dx \dv.
\]

 This technical lemma relating the Fourier transform and the Galilean change of variables $\Gamma$ will be used frequently. To simplify the presentation 
we introduce the weights
\begin{align}
\label{eq:weight}
 W(t,\varphi,\xi)= \sup(|\xi-t \varphi|, |\varphi|^{\frac{1}{1+2\beta}}),
\end{align}
and 
\[w(\varphi,\xi)= \sup(|\xi|, |\varphi|^{\frac{1}{1+2\beta}})
\]
which depend on the fixed parameter $\beta>0$.

We define $\cD_{t}\,\cS_{x,v}$ as the subspace of functions in $\cS(\Omega)$ having compact support in the $t$-variable, identified to compactly supported functions in $t$ valued in the Schwartz space in $(x,v)$ variables. Its dual is denoted $\cD'_{t}\,\cS'_{x,v}$ and identified with the space of distributions valued in the space of tempered distributions. 

\begin{lem}\label{l:galilean}
 Let $\gamma\in \R$ with $\gamma<d/2$. Then the map $f\mapsto \widehat{\Gamma f}$ is an isomorphism 
\begin{enumerate}
\item $\cS(\Omega)$ onto $\cS(\Omega)$, and $\cD_{t}\,\cS_{x,v}$ onto $\cD_{t}\,\cS_{\varphi,\xi}$.
\item $\cS'(\Omega)$ onto $\cS'(\Omega)$, and $\cD_{t}'\,\cS'_{x,v}$ onto $\cD'_{t}\,\cS'_{\varphi,\xi}$.
\item $\L^p_{t \vphantom{\varphi,\xi}}\Xdot^{\gamma}_{\beta} $ onto $\L^p_{t \vphantom{\varphi,\xi}}\L^2_{ \varphi,\xi}\big(W^{2\gamma} \dd \varphi \dd \xi \big)$, $1\le p <\infty$.
\item $\C^{}_{0}(\R^{}_{t}\, ;\, \Xdot^{\gamma}_{\beta})$ onto $\C^{}_{0}\big(\R^{}_{t}\, ; \, \L^2_{\vphantom{t} \varphi,\xi}\big(W^{2\gamma} \dd \varphi \dd \xi \big)\big)$.
\item $\L^2_{t}\Xdot^{\gamma}_{\beta}$ onto $\L^2_{t,\varphi,\xi}\big(W^{2\gamma} \dt \dd \varphi \dd \xi \big)$.
\end{enumerate}
\end{lem}
\begin{proof} We have $\widehat {\Gamma f}(t,\varphi,\xi)= \widehat f(t,\varphi,\xi-t\varphi)$.
The first item is rather easy, and we skip the details. The second item follows by duality using the first one, replacing $\Gamma$ by $\Gamma^*=\Gamma^{-1}$. For the third and fourth items, Plancherel's theorem tells us that 
$u\in \Xdot^{\gamma}_{\beta}$ if and only if 
$\widehat u \in \L^2_{\varphi,\xi}(w^{2\gamma} \dd\varphi \dd\xi\big)$. Apply this to $u(x,v)=f(t,x,v)$ for fixed $t$ and then use the change of variables induced by $\Gamma(t)$. This proves the third item. For the fifth item, apply further Fubini's theorem when $p=2$. Concerning the fourth item, the right-hand space should be understood as the image under the change of variable $\widehat \Gamma$ of 
$\C^{}_{0}\big(\R^{}_{t}\, ; \, \L^2_{\varphi,\xi}(w^{2\gamma} \dd\varphi \dd\xi\big)\big)$ where $\widehat \Gamma$ is the conjugate of the Galilean change of variables through the Fourier transform, that is $\widehat \Gamma {h}(t,\varphi,\xi)={h}(t,\varphi,\xi-t\varphi)$.
 \end{proof}

\subsubsection{Forward and backward Kolmogorov operators of order $\beta$}

The forward and backward Kolmogorov (integral) operators of order $\beta$ are denoted by $\cK^{\pm}_\beta$. They are formally defined as inverse maps of $\pm (\partial_{t}+v\cdot\nabla_{x}) +(-\Delta_{v})^{\beta}$, and this is what we rigorously prove under specific assumptions. To define them for now, we use Lemma~\ref{l:galilean} and for
$S\in \cD'_{t}\, \cS'_{x,v}$
(in particular $S \in \cS'(\Omega)$ is possible) 
we set 
\[ f =\cK^{+}_{\beta} S \] whenever
\begin{equation}
 \label{e:conjugaison}
 \widehat{\Gamma f}= T(\widehat {\Gamma S}) \in \cD'_{t}\, \cS'_{\varphi,\xi}, 
\end{equation}
with
\begin{equation}
 \label{eq:kolmogorovoperator}
Th(t,\varphi,\xi)= \int_{-\infty}^t K(t,s,\varphi,\xi)h(s,\varphi,\xi)\ds,
\end{equation}
where for $s,t\in \R$ and $(\varphi,\xi)\in \R^{2d}$, 
\begin{equation}
 \label{eq:kolmogorovkernel}
K(t,s,\varphi,\xi)= \exp\bigg({-\int_{s}^t |\xi-\tau \varphi|^{2\beta}\, d\tau}\bigg).
\end{equation}
Of course, the integral for $Th$ above makes sense only on certain classes of measurable functions $h$ of $(t,\varphi,\xi)$ and is thought of as the Duhamel formula for the ordinary differential equation,
 \[
g'(t)+ |\xi-t\varphi|^{2\beta}g(t)= h(t), \qquad t\in \R. 
\]
But the corresponding $S$ are distributions.

The operator $\cK^{-}_{\beta}$ is the formal adjoint to $\cK^{+}_{\beta}$. It is associated in the same fashion to the adjoint of $T$ given by the following formula,
\begin{equation}
 \label{eq:kolmogorovoperatoradjoint}
T^*h(s,\varphi,\xi)= \int^{+\infty}_{s} K(t,s,\varphi,\xi)h(t,\varphi,\xi)\, dt.
\end{equation}

\begin{rem}
We insist that at this stage that $\cK^{\pm}_{\beta}$ are defined through integral representations, that is, with the help of what is usually called the Kolmogorov fundamental solution in Fourier variables. That $\cK^{\pm}_{\beta} S$ are distributional solutions of the corresponding PDE under appropriate assumptions needs to be proved. It is after we prove uniqueness results for such PDEs that we will be able to connect these solutions to weak solutions.
\end{rem}

\subsubsection{Density}

In order to establish, for instance, the energy equality (see Lemma~\ref{lem:energyequalityweaksol}), it is useful to have density results at hand. This will also be needed to prove (iii) in Theorem~\ref{thm:homkinspace} and for various estimates. As we use a homogeneous setup, we cannot use convolution techniques in the spatial variables. We use Fourier techniques instead.

\begin{lem}[Density] \label{lem:density} 
Assume $d\ge 1$. Let $\cS_{K}$ be the subset of those $S\in \cD_{t}\, \cS_{x,v}$ whose partial Fourier transform has compact support in $\R^{}_{t}\times (\R^{d}_{\varphi}\setminus \{0\})\times (\R^{d}_{\xi}\setminus \{0\})$.   Then $\cS_{K}$ is dense in $\L^p_{t}\Xdot^{\gamma}_{\beta}$ and in $\C^{}_{0}(\R^{}_{t}\, ;\, \Xdot^{\gamma}_{\beta})$ for all $1\le p<\infty$ and $-\infty<\gamma<d/2$ and $\Gamma \cK^{\pm}_{\beta} (\tilde\cS_{K}) \subset \H^1_{t}(\L^2_{x,v})$. 
\end{lem}

\begin{proof}
 By Remark~\ref{rem:Xdotgammabeta} after the definition of the spaces $\Xdot^{\gamma}_{\beta}$, one can see that any element of the listed spaces can be approximated by a function $S\in \cS_{K}$. 
 
 It remains to show that for $S\in \cS_{K}$, $\Gamma\cK_{\beta}^\pm S$ and $\partial_{t}\, \Gamma \cK_{\beta}^\pm S$ belong to $\L^2_{t,x,v}$. 
 We do it for $\cK_{\beta}^+$, the proof for $\cK_{\beta}^-$ being the same.
 
 Let $h=\widehat {\Gamma S}$. This function is $\C^\infty$ and has support in $$\{(t,\varphi,\xi) \in \Omega\, ; \, |t|\le M, m\le |\varphi|\le M, m\le |\xi-t\varphi|\le M\}$$ for some $0<m\le M<\infty$. 

 We now look at $g=Th$ which is to be $g= \widehat{\Gamma\cK_{\beta}^+S}$.  
 Hence it is enough to show that $g$ and $\partial_{t}g$ 
 belong to $ \L^2_{t,\varphi,\xi}$.

 Since
 \[
 g(t,\varphi,\xi)= \int_{-M}^t K(t,s,\varphi,\xi)h(s,\varphi,\xi)\ds
 \] 
and $|\xi|\le M+M^2$ on the support of $h$, we have that the support of $g$ is contained in $\{(t, \varphi, \xi)\in \R^{2d+1}\, ; t\ge -M, \, m\le |\varphi| \le M, |\xi| \le M+M^2\}$, and that 
 $g$ is continuous and bounded. In addition, if $t\ge 2M$ and using $|\xi-s\varphi|\ge m$ and $s\le M$ on the support of $h(s,\varphi,\xi)$,  we have 
$|K(t,s,\varphi,\xi)| \lesssim e^{-c_{\beta}m^{2\beta} (t-M)}$ (see Lemma~\ref{lem:comp} below). Hence $e^{c|t-M|}g$ is also bounded for some $c>0$. All this implies that $g, |\xi-t\varphi|^{2\beta}g \in \L^2_{t,\varphi,\xi}$.

 Eventually, using the support of $h$ and the estimates for $\partial_{t}K$ allows us to differentiate $g$ under the integral sign and we find the equation
 $\partial_{t}g + |\xi-t\varphi|^{2\beta}g=h$. From the above properties of $g$, we conclude that $\partial_{t}g\in  \L^2_{t,\varphi,\xi}$ as well. 
 \end{proof}

\begin{rem}
 For $S\in \cS_{K}$, $\cK^{\pm}_{\beta}S$ are strong $\L^2$ solutions of the forward and backward Kolmogorov equations respectively.
\end{rem}
 
We remark that we can reach a stronger density result when $d\ge 2$. We include this result for its own sake.

 \begin{lem}[Stronger density] \label{lem:strongerdensity} 
 Assume $d\ge2$. There exists a subset $\tilde\cS_{K}$ of $\cS(\Omega)$ which is dense in $\L^p_{t}\Xdot^{\gamma}_{\beta}$ and in $\C^{}_{0}(\R^{}_{t}\, ;\, \Xdot^{\gamma}_{\beta})$ for all $1\le p<\infty$ and $-\infty<\gamma<d/2$ and such that $\cK^{\pm}_{\beta} (\cS_{K}) \subset \cS(\Omega)$. 
\end{lem}

\begin{proof}
We begin with the construction of $\tilde\cS_{K}$, arguing again on the Fourier side. Let $d(\xi, \R \varphi)$ denote the distance from $\xi$ to the line $\R\varphi$. When $\varphi\ne 0$, we have $d(\xi, \R \varphi)^2= \big|\xi- ({\xi}\cdot{\varphi})\, \frac \varphi {|\varphi|^2}\big|^2$ is a $\C^\infty$ function of $(\varphi,\xi)$ on $ \R^d_{\varphi} \setminus \{0\} \times\R^d_{\xi}$. It is not identically zero as $d\ge2$. Define the non-empty open set
\[\Omega_{K}=\{(t,\varphi,\xi)\in \Omega\, ; \, |\varphi|>0, d(\xi, \R \varphi)>0 \}.\] We say that 
$S\in \tilde\cS_{K}$ if and only if $h=\widehat{\Gamma S} \in \cS(\Omega)$ and has compact support in $\Omega_{K}$.  

Let us prove the density of $\tilde\cS_{K}$ in $\L^p_{t}\Xdot^{\gamma}_{\beta}$ and $\C^{}_{0}(\R^{}_{t}\, ;\, \Xdot^{\gamma}_{\beta})$. By Lemma~\ref{lem:density},
we may assume that $S\in \cS_{K}$, so that $h=\widehat {\Gamma S}$ has support in $\{(t,\varphi,\xi) \in \Omega\, ; \, |t|\le M, m\le |\varphi|\le M, m\le |\xi-t\varphi|\le M\}$ for some $0<m\le M<\infty$. On this set, the function 
$(t,\varphi,\xi)\mapsto d(\xi, \R \varphi)^2$ is $\C^\infty$ and
 if we let $\chi\in \C^\infty(0,\infty)$ with $\chi(y)=1$ if $y\ge 4$ and $\chi(y)=0$ if $y\le 1$, then $h_{\varepsilon}:(t,\varphi,\xi)\mapsto \chi(d(\xi, \R \varphi)^2/\varepsilon^2)h(t,\varphi,\xi)$ is $\C^\infty$ with compact support in $\Omega_{K}$. 
 For the convergence of $h_{\varepsilon}$ to $h$, we have that $h$ has support in $[-M,M]\times \{m\le |\varphi|\le M\} \times \{|\xi|\le M+M^2\}$. In addition $h-h_{\varepsilon}=0$ when $d(\xi, \R \varphi)\ge 2\varepsilon$. If $d(\xi, \R \varphi)\le 2\varepsilon$, then $\xi $ belongs to the cylinder $ \R\varphi+B(0,2\varepsilon)$. Thus, using all the support information, $$
\bigg\| \bigg(\iint |h- h_{\varepsilon}|^2 \, W^{2\gamma} \dd\xi \dd\varphi\bigg)^{1/2} \bigg\|_{\L^p_{t}}  \lesssim_{m,M, \gamma, \beta,p}  \|h\|_{\L^\infty}\, \varepsilon^{1/2}. 
$$

We turn to proving that the Kolmogorov operators map $\tilde\cS_{K}$ to $\cS(\Omega)$. We do it for $\cK_{\beta}^+$, the proof for $\cK_{\beta}^-$ being the same. Let $S\in \cS_{K}$ and set $h=\widehat{\Gamma S}$. We have to show that $g=Th\in \cS(\Omega)$. Since $h$ has compact support in $\Omega_{K}$, there exists $0<m<M<\infty$ such that $h$ is supported in $[-M,M]\times F$, with
\[
 F=\{(\varphi,\xi)\in \R^{2d}\, ;  \, m\le |\varphi|\le M, m\le d(\xi,\R \varphi) , |\xi|\le M\}.
\]  
On $\R\times (\R^{2d} \setminus F) \cup ((-\infty, -M] \times F)$, $g$ is clearly 0, that is, $g$ is supported in $ [-M,\infty)\times F $. 
Let $0<m'<m$ and $M<M'<\infty$ and $O=\{(\varphi,\xi)\in \R^{2d}\, ;  \, m'< |\varphi|<M', m'< d(\xi,\R \varphi) , |\xi|< M'\}$. When $m' \le d(\xi,\R \varphi) $ and $t\ge s$, $\int_{s}^t|\xi-\tau\varphi|^{2\beta}\, d\tau \ge (t-s) (m')^{2\beta}$. In particular $(\varphi,\xi)\mapsto \exp (- \int_{s}^t|\xi-\tau\varphi|^{2\beta}\, d\tau)$ is $\C^\infty$ on $O$ and $g$ is $\C^\infty$ with respect to $(\varphi,\xi)\in O$ for any $t> -M'$, and all its partial derivatives with respect to $(\varphi,\xi)\in O$ have exponential decay in $t\to \infty$, uniformly. Next, as $\partial_{t}g= -|\xi-t\varphi|^{2\beta} g+h$, the same is true for $\partial_{t}g$. By induction, we obtain that $g$ is $\C^\infty$ with respect to $(t,\varphi,\xi)$. Since $g$ has support in $ [-M,\infty)\times F $, with all partial derivatives having exponential decay in time, uniformly in $(\varphi,\xi)$, we conclude that $g\in \cS(\Omega)$. 
\end{proof}

Remark that $\tilde\cS_{K}$ is even contained in $ \cD_{t}\, \cS_{x,v}$ but elements in  $\cK^{\pm}_{\beta} (\tilde\cS_{K})$ do not have compact support in $t$.

\subsubsection{Estimates}

The main technical result related to Kolmogorov operators is the following a priori estimates. It makes precise how Kolmogorov operators act on various functional spaces in $(t,x,v)$
. 

\begin{prop}[A priori estimates for Kolmogorov operators] \label{p:boundsKolmoapriori} Let $\beta>0$ and $\gamma\in \R$. Assume that $S\in \cS_{K}$ of Lemma~\ref{lem:density}. In particular, $\cK^{\pm}_{\beta}S$ belongs to $\L^2_{t,x,v}$, so all fractional derivatives below make sense allowing the $\Xdot^\alpha_{\beta}$ and $\Hdot^\alpha_{v}$ norms to be computed. Then $\cK^{\pm}_{\beta}S\in \L^2_{t}\Xdot^\gamma_{\beta} \cap \C^{}_{0}(\R^{}_{t}\, ; \, \Xdot^{\gamma-\beta}_{\beta})$ with bound
\begin{align*}
 \max\{
 \|\cK^{\pm}_{\beta}S\|_{\L^2_{t}\Xdot^\gamma_{\beta}}, 
\sup_{t\in \R}\|\cK^{\pm}_{\beta}S\|_{\Xdot^{\gamma-\beta}_{\beta}}
\}
\lesssim
 \min\{\|S\|_{\L^2_{t}\Xdot^{\gamma-2\beta}_{\beta}}, \|S\|_{\L^1_{t}\Xdot^{\gamma-\beta}_{\beta}}
 \}
\end{align*}
where the implicit constants depend only on $\beta,\gamma,d$. 
In particular, if $0\le \gamma\le 2\beta$, we have the following four inequalities
\begin{align*}
 \max\{
 \|D_{v}^\gamma\cK^{\pm}_{\beta}S\|_{\L^2_{t,x,v}},
\|D_{\vphantom {D_{v}} x}^{\frac{\gamma}{2\beta+1}}\cK^{\pm}_{\beta}S\|_{\L^2_{t,x,v}} 
\}
\lesssim
 \min\{ \|D_{v}^{\gamma-2\beta} S\|_{\L^2_{t,x,v}}, \|D_{\vphantom {D_{v}} x}^\frac{\gamma-2\beta}{2\beta+1}S\|_{\L^2_{t,x,v}}
 \},
\end{align*}
and if $\beta\le \gamma\le 2\beta$,  
\begin{align*}
\max\bigg\{&\sup_{t\in \R} \|D_{v}^{\gamma-\beta}\cK^{\pm}_{\beta}S(t)\|_{\L^2_{x,v}},
\sup_{t\in \R} \|D_{\vphantom {D_{v}} x}^{\frac{\gamma-\beta}{2\beta+1}}\cK^{\pm}_{\beta}S(t)\|_{\L^2_{x,v}} \bigg\}
\\
&
\lesssim
\min\{\|D_{v}^{\gamma-2\beta} S\|_{\L^2_{t,x,v}}, \|D_{\vphantom {D_{v}} x}^\frac{\gamma-2\beta}{2\beta+1}S\|_{\L^2_{t,x,v}}, \|S\|_{\L^1_{t}\Xdot^{\gamma-\beta}_{\beta}}
\}.
\end{align*}
\end{prop}

 Using our definition of spaces in the restricted range of parameters and density,  this immediately gives us the following boundedness results.
 
\begin{prop}[Estimates for Kolmogorov operators] \label{p:boundsKolmo}
Let $\beta>0$ and $\gamma\in \R$. We have the following bounds. 
\begin{enumerate} 
\item For $\gamma <d/2$, the operators $\cK^{\pm}_{\beta}$ map boundedly $\L^2_{t}\Xdot^{\gamma-2\beta}_{\beta}$ to $\L^2_{t}\Xdot^\gamma_{\beta}$. 
\item For $\gamma-\beta <d/2$, the operators
$\cK^{\pm}_{\beta}$ map boundedly $\L^2_{t}\Xdot^{\gamma-2\beta}_{\beta}$ to $\C^{}_{0}(\R^{}_{t}\, ; \, \Xdot^{\gamma-\beta}_{\beta})$.
\item For $\gamma <d/2$, the operators $\cK^{\pm}_{\beta}$ map boundedly $\L^1_{t}\Xdot^{\gamma-\beta}_{\beta}$ to $\L^2_{t}\Xdot^\gamma_{\beta}$.
\item For $\gamma-\beta <d/2$, $\cK^{\pm}_{\beta}$ map boundedly $\L^1_{t}\Xdot^{\gamma-\beta}_{\beta}$ to $\C^{}_{0}(\R^{}_{t}\, ;\, \Xdot^{\gamma-\beta}_{\beta})$.
\end{enumerate}
If $0\le \gamma\le 2\beta$ and $\gamma<d/2$,  then we have boundedness on $\L^2_{t,x,v}$ of
\begin{equation*}
  D_{v}^\gamma \cK^{+}_{\beta} D_{v}^{2\beta-\gamma}, 
D_{\vphantom {D_{v}} x}^{\frac{\gamma}{2\beta+1}} \cK^{+}_{\beta} D_{v}^{2\beta-\gamma}, D_{v}^\gamma \cK^{+}_{\beta} D_{\vphantom {D_{v}} x}^{\frac{2\beta-\gamma}{2\beta+1}}, D_{\vphantom {D_{v}} x}^{\frac{\gamma}{2\beta+1}} \cK^{+}_{\beta} D_{\vphantom {D_{v}} x}^{\frac{2\beta-\gamma}{2\beta+1}}.
\end{equation*}
\end{prop}

\begin{rem} 
  The bounds into $\C^{}_{0}$ spaces are a key aspect of our approach. The only restriction on $\gamma$ is the one imposed by our definition of Sobolev spaces. See Proposition~\ref{prop:estimatesT} below, where the Fourier estimates are proven without restriction.
\end{rem}
We isolate an important boundedness result for later use in the case $\gamma=\beta$, and that is valid when $\beta>0$.

\begin{cor}\label{cor:Kbeta} 
If $\beta>0$, the operator \(\cK_{\beta}^\pm\) is a bounded linear map
 \begin{equation}
 \label{eq:Kbeta}
 \cK_{\beta}^\pm: \L^2_{t,x}\Hdot^{-\beta}_{v} + \L^2_{t,v}\Hdot_{\vphantom{t,v} x}^{-\frac{\beta}{2\beta+1}}
+ \L^1_{t}\L^2_{x,v}\to \L^2_{t,x}\Hdot^{\beta}_{v} \cap \L^2_{t,v}\Hdot_{\vphantom{t,v} x}^{\frac{\beta}{2\beta+1}} \cap \C^{}_{0}(\R^{}_{t}\, ; \L^2_{x,v}).
\end{equation}
\end{cor}

\begin{proof} 
Suppose that $S$ belongs to the source space. As $-\beta<0<d/2$, by Lemma~\ref{lem:density} one can find $S_{j}\in \cS_{K}$ converging to $S$ in the source space norm. By Proposition~\ref{p:boundsKolmoapriori},  $\cK_{\beta}^\pm S_{j}$ satisfies the estimates in the sum of norms of the target space. In particular, $\cK_{\beta}^\pm S\in \C^{}_{0}(\R_{\vphantom{t,v} t}\, ; \L^2_{x,v})$ at the limit. Hence, we can also take limits in the other norms and obtain the conclusion.  
\end{proof}

Our goal is now to establish Proposition~\ref{p:boundsKolmoapriori}.
Thanks to \eqref{e:conjugaison} and Lemma~\ref{l:galilean}, its proof is done in the Fourier side, using Plancherel in the end. 
The method consists of working with fixed Fourier variables $(\varphi,\xi)$ and proving weighted temporal estimates by integrating with respect to $t$. It makes explicitly clear two facts: on the one hand, the uniformity (or dependency) of the constants with respect to $(\varphi,\xi)$, yielding in the end homogeneous estimates, and on the other hand, that time intervals can be infinite.

 Making the variables $(\varphi,\xi)$ implicit, we set
\[
  K(t,s)=K(t,s,\varphi,\xi), \quad Th(t)= \int_{-\infty}^t K(t,s)h(s)\ds, \quad W(t)=W(t,\varphi,\xi)
\]
 for the Kolmogorov kernel, the associated operator and the weight defined in \eqref{eq:kolmogorovkernel}, \eqref{eq:kolmogorovoperator} and \eqref{eq:weight}.

\subsubsection{Weighted temporal estimates}

In order to prove Proposition~\ref{p:boundsKolmoapriori}, we fix $(\varphi,\xi)\ne (0,0)$, define weighted temporal spaces and study how the map $T$ acts on them.

For ${\eps} \in \R$ and $1\le p < \infty$, we use the notation $\L^p_{t}(W^{\varepsilon p})$ for the space $\L^p(\R^{}_{t}\, ;\,W^{\varepsilon p} \dt)$ of measurable functions $g$ such that 
$W^{\varepsilon}g\in \L^p_{t}$ with norm $\|g\|_{\L^p_{t}(W^{\varepsilon p})}= \|W^{\varepsilon}g\|_{\L^p_{t}}$, that is 
 \[
 \|g\|_{\L^p_{t}(W^{\varepsilon p})}^p: = \int_{\R} |g(t)|^p\,\big(\sup(|\xi-t \varphi|, |\varphi|^{\frac{1}{1+2\beta}})\big)^{\varepsilon p} \dd t.\] 
The following proposition will use elementary estimates on the kernel $K$. By abuse of notation, we also denote by $W^\varepsilon$ the operator of multiplication by $W^\varepsilon(t)$. 
\begin{prop}[Weighted temporal estimates]
\label{prop:estimatesT} Let $\beta >0$ and $\gamma\in \R$. Then with bounds for the norms that only depend on $\beta, \gamma, d$, 
 \begin{align*}
T &: \L^2_{t}(W^{2\gamma -4\beta}) \to \L^2_{t}(W^{2\gamma})\ \mathrm{ is\ bounded.}
\\
 T&: \L^1_{t}(W^{\gamma -\beta}) \to \L^2_{t}(W^{2\gamma})\ \mathrm{ is\ bounded.}
\\
 W^{\gamma-\beta}T&: \L^2_{t}(W^{2\gamma -4\beta}) \to \C^{}_{0}(\R^{}_{t})\ \mathrm{ is\ bounded.} 
 \\
 W^{\gamma-\beta} T&: \L^1_{t}(W^{\gamma-\beta}) \to \C^{}_{0}(\R^{}_{t})\ \mathrm{ is\ bounded.}
\end{align*}
The same estimates are valid for $T^*$.
\end{prop}

\subsubsection{Elementary estimates on the function $K$}

To prove Proposition~\ref{prop:estimatesT}, we begin with two elementary lemmas about the function $K$.
\begin{lem}\label{lem:comp} Let $\beta>0$.
 There are constants $0<c_{\beta}<C_{\beta}<\infty$ such that for all $s,t \in \R$ with $s<t$ and $(\varphi,\xi)\ne (0,0)$, 
 \[
  c_{\beta} \le \frac{\int_{s}^t |\xi-\tau \varphi|^{2\beta}\dd\tau}{(t-s) (|\xi-s \varphi|^{2\beta}+ ((t-s)|\varphi|)^{2\beta})} \le C_{\beta}
\]
and 
\[ c_{\beta} \le \frac{\int_{s}^t |\xi-\tau \varphi|^{2\beta}\dd\tau}{(t-s) (|\xi-t \varphi|^{2\beta}+ ((t-s)|\varphi|)^{2\beta})}\le C_{\beta}.
\]
\end{lem}
\begin{proof}
 Set $\xi'= \xi-s\varphi$ and $\varphi'=(t-s)\varphi$, then a change of variable yields
 $$
\int_{s}^t |\xi-\tau \varphi|^{2\beta}\, \dd\tau= (t-s) \int_{0}^1 |\xi'-\tau \varphi'|^{2\beta}\, \dd\tau.
$$
Now the function 
$$
(\xi',\varphi')\mapsto \int_{0}^1 |\xi'-\tau \varphi'|^{2\beta}\,\dd\tau
$$
is continuous on $\R^{2d}$, hence bounded above and below by positive numbers $C_{\beta}, c_{\beta}$ on the compact set defined by
$|\xi'|^{2\beta}+ |\varphi'|^{2\beta}=1$. One can conclude the first inequality by using that it is also homogeneous of order $2\beta$. 
One can do similarly with $\xi'=\xi-t\varphi$ writing
$$
\int_{s}^t |\xi-\tau \varphi|^{2\beta}\, \dd\tau= (t-s) \int_{0}^1 |\xi'+\tau \varphi'|^{2\beta}\, \dd\tau
$$
and conclude as before. 
\end{proof}

\begin{lem}
\label{lem:estimatesKts} Let $\beta>0$, $\varepsilon\in \R$, $\alpha>0$ and $(\varphi,\xi)\ne (0,0)$. One has the following estimates with implicit constants depending on $\beta, \varepsilon, \alpha$ and independent of $s,t,\varphi,\xi$.  
\begin{align}
 \label{eq:K1} \int_{-\infty}^t K(t,s)^\alpha \ds &\simeq_{\beta,\alpha} W^{-2\beta}(t),
 \\
 \label{eq:K2} \int_{s}^\infty K(t,s)^\alpha \, \dt & \simeq_{\beta,\alpha} W^{-2\beta}(s),
\\
 \label{eq:K4} \sup\big( W^\varepsilon(t), W^\varepsilon(s)\big) K(t,s) & \lesssim_{\beta,\varepsilon} \inf\big( W^\varepsilon(t), W^\varepsilon(s)\big) , \quad s\le t.
 \\
 \label{eq:K5} 
 \int_{-\infty}^t K(t,s) W^{2\beta -\varepsilon}(s) \ds
 &\lesssim_{\beta,\varepsilon} W^{- \varepsilon}(t),  \\
 \label{eq:K6} 
 \int_{s}^\infty K(t,s)W^{2\beta -\varepsilon}(t) \dt
 &\lesssim_{\beta,\varepsilon} W^{-\varepsilon}(s).  \end{align}

 \end{lem}
\begin{proof}
Proof of \eqref{eq:K1} and \eqref{eq:K2}. 
It is easily checked for all $A, B \ge 0$, and $c>0$, 
\begin{align}
\label{eq:equivAB}
 \int_{0}^\infty e^{-cu^{1+2\beta}A^{2\beta}} e^{-cuB^{2\beta}} \, \dd u \simeq_{c,\beta} \inf\big(B^{-2\beta}, A^{-\frac{2\beta}{1+2\beta}}\big)= \big(\sup\big(B, A^{\frac{1}{1+2\beta}}\big)\big)^{-2\beta}.
\end{align}
 Thus,  use the comparisons in Lemma~\ref{lem:comp}, the change of variable $t-s=u$ and \eqref{eq:equivAB}.
 
Proof of \eqref{eq:K4}. First, there is nothing to do if $s=t$, and we assume $s<t$. Also for $\varepsilon=0$, we see that $0<K(t,s)\le 1$, so we are done. Assume next $\varepsilon\ne 0$. Set $ A=(t-s)^{\frac 1 {2\beta}} |\xi-t\varphi|$, $B=(t-s)^{\frac 1 {2\beta}} |\xi-s\varphi|$ and $C= (t-s)^{ \frac{2\beta+1}{2\beta}} |\varphi|$. We have $C>0$, $A,B\ge 0$ and $|A-B|\le C$. Observe that using both inequalities in Lemma~\ref{lem:comp} yields
 \[
 K(t,s) \le \exp\big( -\frac{c_{\beta}}2\big(A^{2\beta} + B^{2\beta} + 2C^{2\beta}\big)\big)
 \]
 so that after multiplication by $(t-s)^{\frac 1 {2\beta}} $, the desired inequality follows from 
 \[
 \sup(\rho^\varepsilon, \rho^{-\varepsilon})\exp\big( -\frac{c_{\beta}}2\big(A^{2\beta} + B^{2\beta} + 2C^{2\beta}\big)\big) \le C_{\beta,\varepsilon}, 
 \]
where 
\[\rho= \frac {\sup\big(A, C^{\frac 1 {2\beta+1}}\big)}{ \sup\big(B, C^{\frac 1 {2\beta+1}}\big)}. 
\]
 The case $\varepsilon<0$ reduces to the case $\varepsilon>0$ by symmetry. Thus, we assume now $\varepsilon>0$. 
If $\sup\big(A, C^{\frac 1 {2\beta+1}}\big)=C^{\frac 1 {2\beta+1}}$, then $\rho\le 1$. If $\sup\big(A, C^{\frac 1 {2\beta+1}}\big)=A$, then, using $A\le B+C$,  \[\rho\le 1 + \frac {C}{\sup\big(B, C^{\frac 1 {2\beta+1}}\big)} \le 1 + \big(\sup\big(B, C^{\frac 1 {2\beta+1}}\big)\big)^{2\beta}.\] 
Similarly, we have $\rho^{-1} \le 1 + \big(\sup\big(A, C^{\frac 1 {2\beta+1}}\big)\big)^{2\beta}$.
Hence, the term $\sup(\rho^\varepsilon, \rho^{-\varepsilon})$ can be absorbed by the exponential, and this finishes the proof of \eqref{eq:K4}. 

Proof of \eqref{eq:K5}. Using the pointwise inequality  \eqref{eq:K4} and the integral one \eqref{eq:K2}, 
\begin{align*}
 \int_{-\infty}^t K(t,s)W^{2\beta -\varepsilon}(s) \ds &= \int_{-\infty}^t K(t,s)^{\frac 1 2} (W^{4\beta -2\varepsilon}(s)K(t,s))^{\frac 1 2} \ds
 \\
 & \lesssim_{\beta, 2\beta-\varepsilon} W^{2\beta -\varepsilon}(t) \int_{-\infty}^t K(t,s)^{\frac 1 2} \ds
 \\
 & \lesssim_{\beta,\varepsilon} W^{-\varepsilon}(t).
\end{align*}

The proof of \eqref{eq:K6} is the same. 
 \end{proof}

\subsubsection{An intermediate result}

Before proving Proposition~\ref{prop:estimatesT}, we prove estimates for $T$ in $\L^1$ and $\L^\infty$ spaces. 
\begin{prop}[$\L^1$ and $\L^\infty$ estimates for $T$]
 \label{prop:L1Linftyestimates}
 Let $\beta >0$ and $\varepsilon\in \R$. Then with bounds for the norms that only depend on $\beta$ and $ \varepsilon$, 
\begin{align*}
 W^\varepsilon T W^{-\varepsilon} &: \L^1_{t} \to \C^{}_{0}(\R^{}_{t}) \ \mathrm{ is \ bounded.}
 \\
W^\varepsilon T W^{2\beta-\varepsilon} &: \L^1_{t} \to \L^1_{t} \ \mathrm{ is \ bounded.}
\\W^\varepsilon T W^{2\beta-\varepsilon} &: \L^\infty_{t} \to \L^\infty_{t}\ \mathrm{ is\ bounded.}
\end{align*} 
The same estimates are valid for $T^*$.
\end{prop}
\begin{proof}
We only treat $T$ as the proofs for $T^*$ are similar. We start with the $\L^\infty_{t}$ bound for $W^\varepsilon T(W^{-\varepsilon}f)$ when $f\in \L^1_{t}$. We have by \eqref{eq:K4}, 
\[ W^\varepsilon(t)|T(W^{-\varepsilon}f)(t)|\le \int_{-\infty}^t W^\varepsilon(t)K(t,s)W^{-\varepsilon}(s)|f(s)|\ds \lesssim_{\beta,\varepsilon}\int_{-\infty}^t |f(s)|\ds,
\] 
which is controlled by $\|f\|_{\L^1_{t}}$. This also proves the fact $W^\varepsilon(t)|T(W^{-\varepsilon}f)(t)|\to 0$ when $t\to -\infty$. 

By a density argument and closedness of $\C^{}_{0}(\R^{}_{t})$ in $\L^\infty_{t}$, it is now enough to prove the continuity of $W^{\varepsilon}T(W^{-\varepsilon}f)$ and limit 0 at $+\infty$ assuming $f$ to be compactly supported in $[-R,R]$ for some fixed $R>0$ and continuous. For the limit, the above estimate rewrites for $t\ge R$,
\begin{align*}
W^{\varepsilon}(t) |T(W^{-\varepsilon}f)(t)| \le \int_{-R}^R W^\varepsilon(t)K(t,s)W^{-\varepsilon}(s) |f(s)|\ds
\end{align*}
and clearly the integral converges to $0$ by dominated convergence as $t\to +\infty$. 

For the continuity, we use that $W^{\varepsilon}$ and $W^{-\varepsilon}$ are continuous and also that $T(W^{-\varepsilon}f)$ is locally Lipschitz in time since, as $f$ is continuous with compact support, it is easy to see that $g=T(W^{-\varepsilon}f)$ solves the ordinary differential equation 
$
g'(t)+ |\xi-t\varphi|^{2\beta}g(t)= (W^{-\varepsilon}f)(t) 
$
in the classical sense. 

Next, the $\L^1_{t}-\L^1_{t}$ and $\L^\infty_{t}-\L^\infty_{t}$ bounds of $W^\varepsilon T W^{2\beta-\varepsilon}$
follow from \eqref{eq:K5} and \eqref{eq:K6} by Fubini's theorem.
\end{proof}

\subsubsection{Proof of Proposition~\ref{prop:estimatesT}}

We are now ready to prove the temporal weighted estimates for the operator $T$.

\begin{proof}[Proof of Proposition~\ref{prop:estimatesT}] The last boundedness statement is already contained in Proposition~\ref{prop:L1Linftyestimates}.
We prove the other three properties in the order they are stated. 

First, interpolating $\L^1_{t}-\L^1_{t}$ and $\L^\infty_{t}-\L^\infty_{t}$ boundedness of $W^\gamma T W^{2\beta-\gamma}$, we obtain $\L^2_{t}-\L^2_{t}$ boundedness which is equivalent to the first item. 

Secondly, we have $W^\gamma T :\L^1_{t}(W^{\gamma-2\beta})\to \L^1_{t}$ and $W^\gamma T :\L^1_{t}(W^{\gamma})\to \L^\infty_{t}$. Using Stein-Weiss interpolation with change of measure (see \cite[Theorem 5.4.1]{MR0482275}), we obtain $W^\gamma T :\L^1_{t}(W^{\gamma-\beta})\to \L^2_{t}$ is bounded, which is equivalent to the second item. 

Thirdly, the second case applied to $T^*$ rewrites $W^\gamma T^* W^{\beta-\gamma}: \L^1_{t}\to \L^2_{t}$ is bounded, so by duality and changing $\beta-\gamma$ to $\gamma-\beta$, we have 
 $W^{\gamma-\beta}TW^{2\beta-\gamma}$ maps $\L^2_{t}$ into $\L^\infty_{t}$. Note that if we apply this to continuous and compactly supported functions $f$ then we also have $W^{\gamma-\beta}TW^{2\beta-\gamma}f= W^{\gamma-\beta}TW^{\beta-\gamma}(W^\beta f)\in \C^{}_{0}(\R^{}_{t}) $ as $W^\beta f \in \L^1_{t}$. 
By density of such functions $f$ in $\L^2_{t}$ and closedness of $\C^{}_{0}(\R^{}_{t})$ in $\L^\infty_{t}$, we conclude that $W^{\gamma-\beta}TW^{2\beta-\gamma}$ maps $\L^2_{t}$ into $\C^{}_{0}(\R^{}_{t})$ which is equivalent to the third item.
\end{proof}

\subsubsection{Proof of Proposition~\ref{p:boundsKolmoapriori}}

We can now obtain the desired a priori estimates for the Kolmogorov operators.
\begin{proof}[Proof of Proposition~\ref{p:boundsKolmoapriori}] Let us concentrate on $\cK^{+}_{\beta}$, the proof for $\cK^{-}_{\beta}$ being similar. We come back to the definition of $T$ depending on the three variables and recall that the weight $W$ also depends on the three variables. By the Fourier transform correspondence $S\mapsto \widehat {\Gamma S}$, Lemma~\ref{l:galilean} and the construction in Lemma~\ref{lem:density}, the space $\widehat {\Gamma \cS_{K}}$ is dense in 
$\L^p_{\vphantom{\varphi,\xi}t}\L^2_{\varphi,\xi}(W^\alpha)$ for all $\alpha\in \R$ and $1\le p<\infty$, where $\L^2_{\varphi,\xi}(W^{\alpha})$ is here the weighted $\L^2$ space with respect to $W^{\alpha}(t,\varphi,\xi)\dd\varphi\dd\xi$, for fixed $t$.
 Using additionally Fubini's theorem, the first assertion for $\cK^{+}_{\beta}$ is equivalent to the weighted boundedness
$$
T: \L^2_{\varphi,\xi}\L^2_{t}(W^{2\gamma-4\beta}) \to \L^2_{\varphi,\xi}\L^2_{t}(W^{2\gamma}).
$$
Using  the conjugate $\widehat \Gamma$ of the Galilean change of variables through the Fourier transform, and using the fact that this map preserves $\C^{}_{0}(\R^{}_{t}\, ; \, \L^2_{\varphi,\xi})$, the second is equivalent to 
$$
W^{\gamma-\beta}\, T: \L^2_{\varphi,\xi}\L^2_{t}(W^{2\gamma-4\beta}) \to \C^{}_{0}(\R^{}_{t}\, ; \, \L^2_{\varphi,\xi}).
$$
The third is equivalent to 
$$
T: \L^1_{t}\L^2_{\varphi,\xi}(W^{2\gamma-2\beta})\to \L^2_{\vphantom{t} \varphi,\xi}\L^2_{\vphantom{t} t}(W^{2\gamma}).
$$ 
With the same idea as for the second assertion, the fourth one is equivalent to 
$$
W^{\gamma-\beta}\, T: \L^1_{t}\L^2_{\varphi,\xi}(W^{2\gamma-2\beta})\to \C^{}_{0}(\R^{}_{t}\, ; \, \L^2_{\varphi,\xi}).
$$ 

The first two assertions on $T$ are direct consequences of Proposition~\ref{prop:estimatesT}.

For the third one, we have by Proposition~\ref{prop:estimatesT} and by Minkowski integral inequality
 \begin{align*}
 \|\|W^{\gamma}\, Th\|_{\L^2_{t}}\|_{\L^2_{ \varphi,\xi}}&\lesssim_{\beta,\gamma} \|\|W^{\gamma-\beta}h\|_{\L^1_{t}}\|_{{\L^2_{\varphi,\xi}}} \\
&\le \|\| W^{\gamma-\beta}h\|_{{\L^2_{\varphi,\xi}}}\|_{\L^1_{t}}.
\end{align*}
For the $\L^\infty$ bound in the fourth one, write
\begin{align*}
 \|W^{\gamma-\beta}(t)\, Th(t)\|_{\L^2_{\varphi,\xi}}&\le  \int_{-\infty}^t \| W^{\gamma-\beta}(t)K(t,s)h(s)\|_{{\L^2_{\varphi,\xi}}}\, \dd s \\
&\lesssim_{\beta,\gamma} \int_{-\infty}^t \| W^{\gamma-\beta}(s)h(s)\|_{{\L^2_{\varphi,\xi}}}\, \dd s
\end{align*}
using Minkowski integral inequality and \eqref{eq:K4} directly, rather than Proposition~\ref{prop:estimatesT}. Continuity and limits follow the argument in Proposition~\ref{prop:L1Linftyestimates}. We leave such details to the reader. 

We turn to the specific case $0\le \gamma\le 2\beta$.   The first weigthed boundedness above reformulates as  the $\L^2_{t,\varphi,\xi}$ boundedness of 
$W^{\gamma-2\beta}TW^\gamma$. Using the definition of $W$ in \eqref{eq:weight} and $\gamma, 2\beta-\gamma\ge 0$, we can split this into the $\L^2_{t,\varphi,\xi}$ boundedness of four operators. Coming back to the Kolmogorov operators and noting that the operators $D_{v}^\alpha$ and $D_{x}^\alpha$ preserve $\cS_{K}$, we obtain the $\L^2_{t,x,v}$ control of
 $$D_{v}^\gamma \cK^{+}_{\beta} D_{v}^{2\beta-\gamma}S, \,
D_{\vphantom {D_{v}} x}^{\frac{\gamma}{2\beta+1}} \cK^{+}_{\beta} D_{v}^{2\beta-\gamma}S , \, D_{v}^\gamma \cK^{+}_{\beta} D_{\vphantom {D_{v}} x}^{\frac{2\beta-\gamma}{2\beta+1}}S, \, D_{\vphantom {D_{v}} x}^{\frac{\gamma}{2\beta+1}} \cK^{+}_{\beta} D_{\vphantom {D_{v}} x}^{\frac{2\beta-\gamma}{2\beta+1}}S$$
by $\|S\|_{\L^2_{t,x,v}}$ when $S\in \cS_{K}$. 
We argue similarly for the final inequality when $\beta\le \gamma \le 2\beta$. 

\end{proof}

\begin{rem}
 Note that $D_{x}$ commutes with $\cK^{+}_{\beta}$, but not $D_{v}$. 
\end{rem}
\section{Uniqueness and kinetic embeddings}
\label{s:uniqueness-embeddings}

\subsection{Uniqueness}

The kinetic embeddings rely on a crucial uniqueness result. 
\begin{lem}[Uniqueness]
 \label{lem:uniqueness}
  Let $\beta>0$ and $\gamma>0 $. If $f\in \Udot^\gamma $  satisfies $\pm (\partial_{t}+v\cdot\nabla_{x})f +(-\Delta_{v})^{\beta}f=0$ in $\cD'(\Omega)$, then $f=0$.
\end{lem}
Recall that $\Udot^\gamma$ is defined just before Definition~\ref{defn:weaksol}.
\begin{proof}
 We only consider $ (\partial_{t}+v\cdot\nabla_{x})f +(-\Delta_{v})^{\beta}f=0$, while the backward equation can be treated by a similar argument. When $f\in \Udot^\gamma$, $f$ and $(-\Delta_{v})^{\beta}f$ belong to $\L^1_{\loc,t}\cS'_{x,v}$, with partial Fourier transform in $\L^1_{\loc}(\Omega)$, and the equation above can be interpreted in $\cD'_{t}\, \cS'_{x,v}$. Setting $G=\widehat{\Gamma f}$, 
 we have $G \in \L^1_{\loc}(\Omega)$ and
$
\partial_{t}G+ |\xi-t\varphi|^{2\beta} G=0
$ in $\cD'_{t}\, \cS'_{x,v}\subset \cD'(\Omega)$.  From $\partial_{t}G=-|\xi-t\varphi|^{2\beta} G \in \L^1_{\loc}(\Omega)$, we deduce that $G$ has a representative in $ \C(\R^{}_{t}; \L^1_{\loc, \varphi,\xi})$ and then in $\C^1(\R^{}_{t}; \L^1_{\loc, \varphi,\xi})$. Identifying $G$ with this representative, the equation can be interpreted as a first-order differential equation valued in $\L^1_{\loc,\varphi,\xi}$. Thus, there exists $G_{0}\in \L^1_{\loc,\varphi,\xi}$ such that 
$G= K(t,0)G_{0}$ for all $t\in \R$, where $K(t,0)$ is the kernel defined in \eqref{eq:kolmogorovkernel}. As $(-\Delta_{v})^{\gamma/2}f\in \L^2_{t,x,v}$, we have that $|\xi-t\varphi|^\gamma G \in \L^2_{t,\varphi,\xi}$, that is 
\begin{equation}
\label{eq:control}
\iiint_{\Omega} |\xi-t\varphi|^{2\gamma} \exp\bigg({-2\int_{0}^t |\xi-\tau \varphi|^{2\beta}\, d\tau}\bigg) |G_{0}(\varphi,\xi)|^2\, \dt \dd \varphi \dd \xi <\infty.
\end{equation}
We can use Fubini's theorem and notice that the $t$-integral on $\R$ is infinite when $(\varphi,\xi)\ne (0,0)$ because of the exponential growth when $t\to -\infty$. This proves that $G_{0}=0$ almost everywhere. Hence, $G=0$ and so is $f$.
 \end{proof}

\begin{rem} 
\label{rem:uniqueness} The argument works for a larger class of $f$. From the argument, we see that it suffices that $f, (-\Delta_{v})^\beta f \in \L^1_{\loc,t}\cS'_{x,v}$ with $G=\widehat {\Gamma f}\in \L^1_{\loc}(\Omega)$ such that there is some measurable function $m$ for which $mG\in \L^1_{t}\L^1_{\loc,\varphi,\xi}$ and 
\[
\int_{\R} |m(t,\varphi,\xi)| \exp\bigg({-\int_{0}^t |\xi-\tau \varphi|^{2\beta}\, d\tau}\bigg) \, \dt =\infty
\]
almost everywhere. For example this works for $f\in \L^2_{t,v}\Hdot_{\vphantom{t,v} x}^{\frac{\gamma}{2\beta+1}}$ if $\frac{\gamma}{2\beta+1}<d/2$, with additionally $f\in \L^1_{\loc, t}\L^2_{x,v}$ otherwise. This also works for sums of spaces for which these conditions hold. 
 \end{rem}

\begin{rem}\label{rem:inverttime}
 The argument shows that it is only the behaviour at $t \to-\infty$ that forces $f$ to be 0. For the backward equation, it is the behaviour at $t\to +\infty$ that counts.
\end{rem}

\subsection{Isomorphism properties}
 
A consequence of uniqueness is the isomorphism properties of the Kolmogorov operators for our scales of spaces in an appropriate range of exponents.

\begin{lem}\label{lem:distributionalsolutions}
Assume $\beta>0$,  $0\le \gamma \le2\beta$ with $\gamma<d/2$. If $S\in \L^2_{t} \Xdot^{\gamma -2\beta}_\beta+ \L^1_{t} \Xdot^{\gamma-\beta}_{\beta}$, then
$\pm (\partial_{t}+v\cdot\nabla_{x})\cK^\pm_\beta S + (-\Delta_{v})^{\beta}\cK^\pm_\beta S=S$ in $\cS'(\Omega)$, hence in $\cD'(\Omega)$. 
If $\gamma=\beta$, one can drop the condition $\gamma<d/2$. 
\end{lem}

\begin{proof} We do the argument for $\cK_{\beta}^+$ and the one for $\cK_{\beta}^-$ is similar. 
 For $S$ in the density class $\cS_{K}$ of Lemma~\ref{lem:density}, the proof of that lemma shows that for $g=Th$ with $h=\widehat{\Gamma S}$, we have $g, \partial_{t}g, |\xi-t\varphi|^{2\beta}g \in \L^2_{t,\varphi,\xi}$ and $\partial_{t}g+ |\xi-t\varphi|^{2\beta}g=h$. It easily follows that the equation holds in $\cS'(\Omega)$, hence the formula holds by applying the correspondence via the Fourier transform. Using the bounds for the operators $\cK_{\beta}^+$ in Proposition~\ref{p:boundsKolmoapriori}, one can deduce the formula in $\cS'(\Omega)$ for general $S$ by density, hence it holds also in $\cD'(\Omega)$. 
 If $\gamma=\beta$, then we already know the bounds of $\cK^\pm_\beta$ by Corollary~\ref{cor:Kbeta} for all $\beta>0$, and we conclude by the same density argument. 
 \end{proof}

We have already established that Kolmogorov solutions are the unique weak solutions.

\begin{proof}[Proof of Corollary~\ref{cor:equality}] 
Lemma~\ref{lem:distributionalsolutions} and Corollary~\ref{cor:Kbeta} prove that the Kolmogorov solutions $\cK_{\beta}^\pm S$ are weak solutions to $\pm (\partial_{t}+v\cdot\nabla_{x})f + (-\Delta_{v})^{\beta}f=S$ in $\cD'(\Omega)$ respectively when $S\in \Zdot^{\beta}$. Moreover, they are the only ones by Lemma~\ref{lem:uniqueness}. 
 \end{proof}

In fact, we have obtained more, and this is the key ingredient for developing our theory of weak solutions.

\begin{lem}[Isomorphisms] \label{lem:isom-embed}
  Assume $0\le \gamma \le2\beta$ with $\gamma<d/2$.
  The Kolmogorov operators $\cK^\pm_\beta$ are isomorphisms from 
  $\L^2_{t,x} \Hdot^{\gamma -2\beta}_{\vphantom{t,x} v}$ onto $\cFdot^\gamma_{\beta}$, from  $\L^2_{t} \Xdot^{\gamma -2\beta}_{\beta}$ onto $\cGdot^\gamma_{\beta}$ and from $\L^2_{t} \Xdot^{\gamma -2\beta}_\beta+ \L^1_{t} \Xdot^{\gamma-\beta}_{\beta}$ onto $\cLdot^\gamma_{\beta}$, with inverses given by $\pm (\partial_{t}+v\cdot\nabla_{x}) + (-\Delta_{v})^{\beta}$.
  If $\gamma=\beta\geq d/2$, then the statement holds provided we define the kinetic spaces with $f\in \Udot^\beta$ replacing $f\in\L^2_{t,x}\Hdot^{\beta}_{v}$.
\end{lem}
\begin{proof}

We do the proof of the last statement for $\cK^+_\beta$ as the other cases are all the same. 

First we assume $0\le \gamma \le2\beta$ with $\gamma<d/2$. It follows from Proposition~\ref{p:boundsKolmo} that $\cK^+_\beta$ maps continuously $S\in \L^2_{t} \Xdot^{\gamma -2\beta}_{\beta}+ \L^1_{t} \Xdot^{\gamma-\beta}_{\beta}$ into $\cLdot^\gamma_{\beta}$.
 Indeed, $\cK^+_\beta S\in \L^2_{t,x} \Hdot^{\gamma}_{\vphantom{t,x} v}$ as $\gamma\ge 0$, and by the equation in Lemma~\ref{lem:distributionalsolutions}, $ (\partial_{t}+v\cdot\nabla_{x})\cK^+_\beta S=S - (-\Delta_{v})^{\beta}\cK^+_\beta S \in \L^2_{t} \Xdot^{\gamma -2\beta}_{\beta}+ \L^1_{t} \Xdot^{\gamma-\beta}_{\beta}$ since $(-\Delta_{v})^{\beta}\cK^+_\beta S \in 
 \L^2_{t,x} \Hdot^{\gamma -2\beta}_{\vphantom{t,x} v}\subset \L^2_{t} \Xdot^{\gamma -2\beta}_{\beta}$ when $\gamma - 2 \beta \le 0$.
 
Lemma~\ref{lem:distributionalsolutions} shows that $\cK^+_\beta$ is injective. 
 
 Next, we show the ontoness of $\cK^+_\beta$. Consider $f \in \cLdot^\gamma_{\beta}$ and set $S= (\partial_{t}+v\cdot\nabla_{x})f + (-\Delta_{v})^{\beta}f$. By definitions of the spaces $\cLdot^\gamma_{\beta}$ and $\Xdot^{\gamma-2\beta}_\beta$ with $\gamma - 2 \beta \le 0$, we have $S \in \L^2_{t}\Xdot^{\gamma-2\beta}_\beta + \L^1_{t} \Xdot^{\gamma-\beta}_{\beta}$ with norm controlled by that of $f$ in $\cLdot^\gamma_{\beta}$. We observe that $\cK^{+}_{\beta} S \in \L^2_{t}\Xdot^{\gamma}_\beta \subset \L^2_{t,x}\Hdot^\gamma_{v}$ by Proposition~\ref{p:boundsKolmo} and $\gamma\ge 0$. As $ \cK^{+}_{\beta} S$ is also a distributional solution of the equation satisfied by $f$ from Lemma~\ref{lem:distributionalsolutions}, it agrees with $f$ by uniqueness from Lemma~\ref{lem:uniqueness}.

This establishes continuity and bijectivity of $\cK^+_\beta$ and that its inverse is $ (\partial_{t}+v\cdot\nabla_{x}) + (-\Delta_{v})^{\beta}$. That the latter is continuous from $\cLdot^\gamma_{\beta}$ into $\L^2_{t}\Xdot^{\gamma-2\beta}_\beta + \L^1_{t} \Xdot^{\gamma-\beta}_{\beta}$ has been proven above. Hence, the isomorphism property is proved.
 
Now we consider the case $\gamma=\beta \ge d/2$. The definition of the kinetic space replacing 
 \eqref{eq:Ldotgammabeta}, seen already in Corollary~\ref{cor:boundedness}, is
 \begin{align}
 \label{eq:Ldotgammabetaged/2}
 \cLdot^\beta_{\beta}= \{ f \in \cD'(\Omega)\, ; \, f \in \Udot^\beta \ \& \ (\partial_t + v \cdot \nabla_x)f \in \L^2_{t} \Xdot^{ -\beta}_\beta+\L^1_{t} \L^2_{x,v}\}
\end{align}
with 
\begin{align*}
 \|f\|_{\cLdot^\beta_{\beta}}= \| D_{v}^{\beta} f\|_{\L^2_{t,x,v}} + \|(\partial_t + v \cdot \nabla_x)f\|_{\L^2_{t}\Xdot^{-\beta}_\beta+\L^1_{t}\L^2_{x,v}}.
 \end{align*}
 Then the same argument as above applies replacing $\gamma$ by $\beta$ and $\L^2_{t,x} \Hdot^{\gamma}_{\vphantom{t,x} v}$ by $\Udot^\beta$ when they appear. This proves the isomorphism property.
\end{proof}

We mention an immediate consequence, which is not direct from the definition of kinetic spaces. 

\begin{cor}\label{cor:complete} 
In the range of exponents prescribed in Lemma~\ref{lem:isom-embed}, the kinetic spaces $\cFdot^\gamma_{\beta}$, $\cGdot^\gamma_{\beta}$ and $\cLdot^\gamma_{\beta}$ (with the modification of definition when $\gamma=\beta\ge d/2$) are complete normed spaces. 
\end{cor}

 \subsubsection{Proof of the embeddings and transfers of regularity}
  
\begin{proof}[Proof of Theorem~\ref{thm:homkinspace}] 
The first two items (i) and (ii) follow from Lemma~\ref{lem:isom-embed} and Proposition~\ref{p:boundsKolmo}, or Corollary~\ref{cor:Kbeta} if $\gamma=\beta$.

\bigskip

\paragraph{Proof of (iii):} By Lemma~\ref{lem:density},  the space $\cS_{K}$ which is dense in any $\L^2_{t}\Xdot^{\gamma-2\beta}_{\beta}+ \L^1_{t} \Xdot^{\gamma-\beta}_{\beta}$ and has range under $\cK^{+}_{\beta}$  contained in $\L^2_{t,x,v}$ (and, if we replace it by $\cS_{K}$ of Lemma~\ref{lem:strongerdensity} if $d\ge 2$, contained  in $ \cS(\Omega)$). We conclude using the isomorphism property in Lemma~\ref{lem:isom-embed} that this range is dense in each of our kinetic spaces. The same holds with the range under $\cK^{-}_{\beta}$ of this dense subspace.

\bigskip

\paragraph{Proof of (iv):} When $\gamma-2\beta <d/2$, the Fourier transform in the $(x,v)$ variables provides us with  isomorphisms 
\begin{align*}
 \L^2_{t,x} \Hdot^{\gamma -2\beta}_{\vphantom{t,x} v} &\to \L^2_{t,\varphi,\xi}(|\xi|^{2\gamma-4\beta} \dt \dd \varphi \dd\xi), \\
 \L^2_{t} \Xdot^{\gamma -2\beta}_{{\beta}} &\to \L^2_{t,\varphi,\xi}(w^{2\gamma-4\beta} \dt \dd \varphi \dd\xi),
\end{align*}
where we recall that $w(\varphi,\xi)=\sup(|\xi|, |\varphi|^{\frac{1}{1+2\beta}})$.
Hence, complex interpolation for weighted $\L^2$ spaces, see \cite{MR0482275}, implies that the right-hand spaces in the two lines have the complex interpolation property, so the same holds for the corresponding left-hand spaces. We conclude from the isomorphism 
properties of $\cK^{+}_{\beta}$ (Lemma~\ref{lem:isom-embed}) in the allowed range of $\gamma$'s that the scales $\cFdot^\gamma_{\beta}$ and $\cGdot^\gamma_{\beta}$ of kinetic spaces have the complex interpolation property, respectively.
\end{proof}
\begin{rem}
We also have the isomorphism
\begin{equation*}
	\L^2_{t}\Xdot^{\gamma-2\beta}_\beta + \L^1_{t} \Xdot^{\gamma-\beta}_{\beta} \to \L^2_{t}\L^2_{\varphi,\xi}(w^{2\gamma-4\beta} \dd \varphi \dd\xi)+ \L^1_t\L^2_{\vphantom{t} \varphi,\xi}(w^{2\gamma-2\beta} \dd \varphi \dd \xi )
\end{equation*}
but it is not clear whether the scale of spaces on the right-hand side is an interpolation family, whence we cannot deduce the complex interpolation property for $\cLdot^\gamma_{\beta}$. 
\end{rem}
\begin{rem}
When $\gamma=\beta$, we can obtain multiplicative inequalities by scale invariance for $f\in \cFdot^\beta_{\beta}$. Applying (i) to $f_\delta(t,x,v) = f(\delta t, x, \delta v)$ and optimizing the choice of $\delta>0$ we obtain the multiplicative inequality
\[
	\|D_{x}^{\frac{\beta}{2\beta+1}} f\|_{\L^2_{t,x,v}} \lesssim_{d,\beta} \|D_{\vphantom{x} v}^\beta f\|_{\L^2_{t,x,v}} ^{\frac{\beta+1}{2\beta+1}}\|(\partial_t +v \cdot \nabla_x) f\|_{{\L^2_{t,x} \Hdot_{\vphantom{t,x} v}^{-\beta}}}^{\frac{\beta}{2\beta+1}}.
\]
 Similarly, applying (ii) to $f_\delta(t,x,v) = f( t, \delta x, \delta v)$ and optimizing in $\delta>0$, we obtain the multiplicative inequality 
\[
 \sup_{t\in \R} \|f(t)\|_{\L^2_{x,v}} \le \big(2 \|D_{v}^\beta f\|_{{\L^2_{t,x,v}}} \|(\partial_t + v \cdot \nabla_x)f\|_{{\L^2_{t,x} \Hdot_{\vphantom{t,x} v}^{-\beta}}}\big)^{\frac 1 2}.
\]
The space $\cFdot^\beta_{\beta}$ has been defined for $\beta<d/2$. For $\beta\ge d/2$, we set $\cFdot^\beta_{\beta}=\{f \in \Udot^\beta\, ;\,  (\partial_t + v \cdot \nabla_x)f \in {\L^2_{t,x} \Hdot_{\vphantom{t,x} v}^{-\beta}}\}$.
 \end{rem}

 \subsubsection{Proof of Theorem~\ref{thm:optimalregularity}.}

We turn to the proof of our main tool in this work.
 
\begin{proof}[Proof of Theorem~\ref{thm:optimalregularity}]
We assume $f\in \Udot^\beta$
 and $(\partial_{t}+v\cdot\nabla_{x})f =S_{1}+S_{2}+S_{3}$ which belongs to $\L^2_{t,x}\Hdot^{-\beta}_{\vphantom{t,x} v} + \L^2_{\vphantom{t,x} t,v}\Hdot_{\vphantom{t,x} x}^{-\frac{\beta}{2\beta+1}}
+ \L^1_{\vphantom{t,x} t}\L^2_{\vphantom{t,x} x,v}$.
  Thus $S:= S_{1}+ (-\Delta_{v})^{\beta}f+S_{2}+S_{3}$ is in the same space and 
$(\partial_{t}+v\cdot\nabla_{x})f + (-\Delta_{v})^{\beta}f= S.$
 By Corollary~\ref{cor:Kbeta}, 
$\cK^{+}_{\beta} S \in \L^2_{t,x}\Hdot^\beta_{v}\cap \L^2_{t,v}\Hdot_{\vphantom{t,v} x}^{ \frac{\beta}{2\beta+1}} \cap \C^{}_{0}(\R^{}_{t}\, ; \L^2_{x,v} )$. Hence, $\cK^{+}_{\beta} S\in \Udot^\beta$ and it is a distributional solution of the same equation as $f$ (Lemma~\ref{lem:distributionalsolutions} in the case $\gamma=\beta$). It is the only one by uniqueness in Lemma~\ref{lem:uniqueness}. It follows that $f= \cK^{+}_{\beta} S,$ and the estimate on $f$ follows from the ones on $\cK^{+}_{\beta} S$.

Let us prove the absolute continuity on $\R$. We just showed that $f= \cK^{+}_{\beta} S$ with $S= S_{1}+ (-\Delta_{v})^{\beta}f+S_{2}+S_{3}$. 
Define $f_{1}=\cK^{+}_{\beta} (S_{1}+ (-\Delta_{v})^{\beta}f)$, $f_{2}=\cK^{+}_{\beta}S_{2}$, $f_{3}=\cK^{+}_{\beta}S_{3}$, so that $f=f_{1}+f_{2}+f_{3}$. By Lemma~\ref{lem:density}, consider approximations $(S_{1}+ (-\Delta_{v})^{\beta}f)_{k}$, $(S_{2})_{k}$ and $(S_{3})_{k}$ of $S_{1}+ (-\Delta_{v})^{\beta}f$, $S_{2}$ and $S_{3}$ in $\L^2_{t,x}\Hdot^{-\beta}_{\vphantom{t,x} v}, \L^2_{\vphantom{t,x} t,v}\Hdot_{\vphantom{t,x} x}^{-\frac{\beta}{2\beta+1}}, \L^1_{\vphantom{t,x} t}\L^2_{\vphantom{t,x} x,v}$ respectively. Set $(f_{1})_{k}=\cK^{+}_{\beta} (S_{1}+ (-\Delta_{v})^{\beta}f)_{k}$, $(f_{2})_{k}=\cK^{+}_{\beta}(S_{2})_{k}$ and $(f_{3})_{k}=\cK^{+}_{\beta}(S_{3})_{k}$. 
We have, when $k\to \infty$, and $i=1,2,3$, 
\begin{align*}
& (S_{1}+ (-\Delta_{v})^{\beta}f)_{k} \to S_{1}+ (-\Delta_{v})^{\beta}f, \quad \mathrm{in}\ \L^2_{t,x}\Hdot^{-\beta}_{v},
 \\
& (S_{2})_{k} \to S_{2}, \quad \mathrm{in}\ \L^2_{t,v}\Hdot_{\vphantom{t,v} x}^{-\frac{\beta}{2\beta+1}},
\\
& (S_{3})_{k} \to S_{3}, \quad \mathrm{in}\ \L^1_{t}\L^2_{x,v},
 \\ 
& (f_{i})_{k} \to f_{i}, \quad \mathrm{in}\ \L^2_{t,x}\Hdot^{\beta}_{\vphantom{t,x} v}\cap \L^2_{\vphantom{t,x} t,v}\Hdot_{\vphantom{t,x} x}^{\frac{\beta}{2\beta+1}} \cap \C^{}_{0}(\R^{}_{t}\, ; \L^2_{\vphantom{t,x} x,v} ).
\end{align*}
Setting $g_{k}=(f_{1})_{k}+(f_{2})_{k}+(f_{3})_{k}$, we know that $\Gamma g_{k}$ and $\Gamma((\partial_t + v \cdot \nabla_x)g_{k})= \partial_{t}\Gamma g_{k}$ belong to $\L^2_{t,x,v}$. 
We can thus write for all $s<t$, 
\begin{align*}
 \|g_{k}(t)\|^2_{\L^2_{x,v}}-\|g_{k}(s)\|^2_{\L^2_{x,v}}
 &= \iint_{\R^{2d}} |(\Gamma g_{k})(t,x,v)|^2 \dx \dv-  \iint_{\R^{2d}}|(\Gamma g_{k})(s, x,v)|^2 \dx \dv
 \\
 &= 2\Re \iint_{\R^{2d}}\int_{s}^t \partial_{t}(\Gamma g_{k})(\tau,x,v) \ \overline{\Gamma g_{k}}(\tau,x,v)  \dd\tau \dx \dv
 \\
 &=
 2\Re \int_{s}^t \iint_{\R^{2d}} \partial_{t}(\Gamma g_{k})(\tau,x,v)\ \overline{\Gamma g_{k}}(\tau,x,v)  \dx \dv \dd \tau
 \\
 &= 2\Re \int_{s}^t \iint_{\R^{2d}} ((\partial_t + v \cdot \nabla_x)g_{k})(\tau,x,v)\ \overline{ g_{k}}(\tau,x,v) \dx \dv \dd \tau
\end{align*}
where we used Fubini's theorem, and change of variables $\Gamma(t), \Gamma(s)$ and $\Gamma(\tau)$ in the $(x,v)$-integrals.  The left hand-side converges to $ \|f(t)\|^2_{\L^2_{x,v}}-\|f(s)\|^2_{\L^2_{x,v}}$. For the right-hand side, we split the integral using the decomposition of $(\partial_t + v \cdot \nabla_x) g_{k}$ into three terms and argue differently for each. Realising the $v$-integral as the duality bracket between 
$\Hdot^{-\beta}_{v}$ and $\Hdot^\beta_{v}$, we have, using the dominated convergence theorem for the $(\tau,x)$-integral,
\begin{align*}
\int_{s}^t \iint_{\R^{2d}} &((S_{1}+ (-\Delta_{v})^{\beta}f)_{k}-(-\Delta_{v})^{\beta}g_{k})\ \overline{ g_{k}} \dx \dv \dd \tau 
\\
&= 
\int_{s}^t \int_{\R^{d}} \angle{(S_{1}+ (-\Delta_{v})^{\beta}f)_{k}-(-\Delta_{v})^{\beta}g_{k}} { g_{k}} \dx \dd \tau 
\to \int_{s}^t \int_{\R^{d}} \angle{S_{1}} {f} \dx \dd \tau. 
\end{align*}
Similarly, realizing the $x$-integral as the duality bracket between $\Hdot_{x}^{-\frac{\beta}{2\beta+1}}$ and $\Hdot_{x}^{\frac{\beta}{2\beta+1}}$ and using dominated convergence in the $(t,v)$-integral 
\begin{align*}
\int_{s}^t \iint_{\R^{2d}} (S_{2})_{k} \ \overline{ g_{k}} \dx \dv \dd\tau 
= 
\int_{s}^t \int_{\R^{d}} \angle{(S_{2})_{k}} { g_{k}} \dv \dd \tau 
\to 
\int_{s}^t \int_{\R^{d}} \angle{S_{2}} {f} \dv \dd\tau.
\end{align*}
Eventually, using merely dominated convergence for the $(t,x,v)$-integral 
\[
\int_{s}^t \iint_{\R^{2d}} (S_{3})_{k}\ \overline{ g_{k}}  \dx \dv \dd\tau 
\to \int_{s}^t \iint_{\R^{d}} S_{3}\ \overline{ f} \dx \dv \dd\tau. \qedhere
\]
\end{proof}

The following result is a generalisation of Theorem~\ref{thm:optimalregularity}. The $d/2$ constraint is likely inessential with an appropriate definition of the spaces $\cLdot^{\gamma}_{\beta}$, which we have chosen not to develop here.
\begin{thm}[Kinetic embedding and transfer of regularity - polarized version]\label{thm:optimalregularitygeneralisation} Let $-\beta\le \varepsilon\le \beta$ such that $\beta\pm\varepsilon<d/2$.
 Let $f, \tilde f\in \cD'(\Omega)$ be such that $f\in \cLdot^{\beta+\varepsilon}_{\beta}$ and $\tilde f\in \cLdot^{\beta-\varepsilon}_{\beta}$. 
 Then 
 \begin{enumerate}
\item $f\in \L^2_{t,v}\Hdot_{\vphantom{t,v} x}^{\frac{\beta+\varepsilon}{2\beta+1}} \cap \C^{}_{0}(\R^{}_{t}\, ; \, \Xdot^{\varepsilon}_{\beta})$ with 
\[
\|D_{x}^{\frac{\beta+\varepsilon}{2\beta+1}} f \|_{\L^2_{t,x,v}}+ \sup_{t\in \R} \|f(t)\|_{\Xdot^{\varepsilon}_{\beta}} \\\lesssim_{d,\beta, \varepsilon} \|f\|_{\cLdot^{\beta+\varepsilon}_{\beta}}
\]
\item $\tilde f\in \L^2_{t,v}\Hdot_{\vphantom{t,v} x}^{\frac{\beta-\varepsilon}{2\beta+1}} \cap \C^{}_{0}(\R^{}_{t}\, ; \, \Xdot^{-\varepsilon}_{\beta})$ with 
\[
\|D_{x}^{\frac{\beta-\varepsilon}{2\beta+1}} {\tilde f} \|_{\L^2_{t,x,v}}+ \sup_{t\in \R} \|{\tilde f}(t)\|_{\Xdot^{-\varepsilon}_{\beta}} \\\lesssim_{d,\beta, \varepsilon} \|{\tilde f}\|_{\cLdot^{\beta-\varepsilon}_{\beta}}.
\]
\item The map $t\mapsto\angle{f(t)}{\tilde f(t)}$, where the bracket is for the duality between $\Xdot^{\varepsilon}_{\beta}, \Xdot^{-\varepsilon}_{\beta}$ is absolutely continuous on $\R$ and its almost everywhere derivative can be computed as follows. Write 
$$
(\partial_{t}+v\cdot\nabla_{x})f =S_{1}+S_{2}+S_{3}
$$
with $S_1 \in \L^2_{t,x}\Hdot^{-\beta+\varepsilon}_{\vphantom{t,x} v}$, $S_2 \in \L^2_{t,v}\Hdot_{\vphantom{t,x} x}^{\frac{-\beta+\varepsilon}{2\beta+1}}$ and $S_3 \in \L^1_{\vphantom{t,x} t}\Xdot_{\vphantom{t,x} \beta}^{-\beta+\varepsilon}$, and 
$$
(\partial_{t}+v\cdot\nabla_{x})\tilde f =\widetilde S_{1}+ \widetilde S_{2}+\widetilde S_{3}
$$
with $\widetilde S_1 \in \L^2_{t,x}\Hdot^{-\beta-\varepsilon}_{\vphantom{t,x} v}$, $\widetilde S_2 \in \L^2_{t,v}\Hdot_{\vphantom{t,x} x}^{\frac{-\beta-\varepsilon}{2\beta+1}}$ and $\widetilde S_3 \in \L^1_{\vphantom{t,x} t}\Xdot_{\vphantom{t,x} \beta}^{-\beta-\varepsilon}$.
Then for a.e. $t\in \R$
\begin{multline*}
\frac{\mathrm{d} }{\mathrm{d}t }\angle{f(t)}{\tilde f(t)} =  \int_{\R^d} (\angle {S_{1}}{\tilde f} + \angle { f}{\widetilde S_{1}}) \dx + \int_{\R^d} (\angle {S_{2}}{\tilde f}+ \angle { f}{\widetilde S_{2}}) \dv + \angle  {S_{3}}{\tilde f} + \angle  {f} {\widetilde S_{3}},
\end{multline*}
where each bracket denotes a different sesquilinear duality extending the $\L^2$ inner product by tracing the spaces involved. Putting absolute values around each bracket gives an integrable function in the missing variables. 
\end{enumerate}
\end{thm} 
\begin{proof}
 Items (i) and (ii) are already in Theorem~\ref{thm:homkinspace}.
Item (iii) is the polarised version \eqref{e:abs-cont} of Theorem~\ref{thm:optimalregularity} where the exponents are symmetric with respect to $\beta$ for duality reasons. Details are left to the reader. 
\end{proof}


\section{Construction of weak solutions to rough equations when $\beta<d/2$}
\label{s:weak-rough}

This section is devoted to the construction of weak solutions when $\beta<d/2$. It relies on a theorem due to J.-L.~Lions \cite[Chap.~II, Thm.~1.1]{Lions}, which can be used in this range for $\beta$. Since this article is written in French, we recall the statement before proving Theorem~\ref{thm:HomExUn}. 

\subsection{Existence and uniqueness of weak solutions}
\label{sec:weaksol}

The proof of the existence of weak solutions to Kolmogorov--Fokker--Planck equations with rough coefficients uses the following theorem.
\begin{thm}[Lions -- \cite{Lions}] \label{thm:lions}
 Let $F$ be a Hilbert space equipped with a scalar product $\langle \cdot, \cdot \rangle_F$ and the associated norm $\|\cdot\|_F$. Let $\mathcal{H} \subset F$ be a prehibertian space equipped with a scalar product $\langle \cdot, \cdot \rangle_{\mathcal{H}}$ and the associated norm $\|\cdot\|_{\mathcal{H}}$. Assume that there exists $c_1>0$ such that for all $f \in \mathcal{H}$, $\|f\|_F \le c_1 \|f\|_{\mathcal{H}}$. Let $E$ be a sesquilinear form on $F \times \mathcal{H}$ such that
 \begin{align*}
  \mathrm{(H1)} & \quad \text{for all } h \in \mathcal{H}, \text{ the linear form } f \mapsto E (f,h) \text{ is continuous on } F, \\
  \mathrm{(H2)} & \quad \text{there exists } \alpha>0 \text{ such that } \Re E(h,h) \ge \alpha \|h\|_{\mathcal{H}}^2 \text{ for all } h \in \mathcal{H}.
 \end{align*}

 Then for any semi-linear continuous form $L$ on $\mathcal{H}$, there exists $f \in F$ such that $E(f,h) = L(h)$ for all $h \in \mathcal{H}$. 
\end{thm}

The following lemma is a consequence of the transfer of regularity obtained in Theorem~\ref{thm:optimalregularity} and is valid for all $\beta>0$. 
\begin{lem}[Energy equality]
\label{lem:energyequalityweaksol} Let $\beta>0$ and
$S=S_{1}+S_{2}+S_{3}$ with $S_1 \in \L^2_{t,x}\Hdot^{-\beta}_{\vphantom{t,x} v}$ and $S_2 \in \L^2_{\vphantom{t,x} t,v}\Hdot_{\vphantom{t,x} x}^{-\frac{\beta}{2\beta+1}}$
and $S_3 \in \L^1_{\vphantom{t,x} t}\L^2_{\vphantom{t,x} x,v}$. Any weak solution to $(\partial_t + v \cdot \nabla_x) f + \cA f= S$ in the sense of Definition~\ref{defn:weaksol} satisfies the energy equality: for any $s,t \in \R$ with $s<t$, 
\begin{align} \label{e:energy}
\nonumber \|f(t)\|^2_{\L^2_{x,v}} + 2\Re& \int_{s}^t\int_{\R^d} a_{\tau,x}({f},{f})\dx\dd \tau = 
\\ \|f(s)\|^2_{\L^2_{x,v}}  &+ 2\Re \int_{s}^t\bigg(\int_{\R^d} \angle {S_{1}}{f} \dx + \int_{\R^d} \angle {S_{2}}{f} \dv + \iint_{\R^{2d}} {S_{3}}\overline f \dx \dv \bigg)\dd \tau.
 \end{align} 
\end{lem}
\begin{proof} As $(\partial_t + v \cdot \nabla_x) f = S-\cA f$, we know from Theorem~\ref{thm:optimalregularity} that 
 $t\mapsto \|f(t)\|^2_{\L^2_{x,v}}$ is absolutely continuous, with limit 0 at $\pm\infty$ and, for $s<t$, 
 $ \|f(t)\|^2_{\L^2_{x,v}}-\|f(s)\|^2_{\L^2_{x,v}}$ is equal to
\begin{align*}
 2\Re \int_{s}^t \bigg(\int_{\R^d} \angle {S_{1}-\cA f}{f} \dx + \int_{\R^d} \angle {S_{2}}{f} \dv + \iint_{\R^{2d}} {S_{3}}\overline f \dx \dv \bigg)\dd\tau.
 \end{align*}
 The bracket inside the first integral denotes the one expressing the $\Hdot^{-\beta}_{v}, \Hdot^{\beta}_{\vphantom {v}v}$ duality as a function of $(\tau,x)$, so by definition and writing the variables for clarity
 \[
  \angle {\cA f}{f}(\tau,x)=a_{\tau,x}({f(\tau,x)},{f(\tau,x)})= a_{\tau,x}({f},{f})
 \]
by convention in the last equality.
\end{proof}

\begin{proof}[Proof of Theorem~\ref{thm:HomExUn} when $\beta<d/2$]
 Lemma~\ref{lem:energyequalityweaksol} already proves that any weak solution satisfies the energy equality. 
\bigskip

In the second step, we prove uniqueness. Assume that $f\in \Fdot^\beta$ with $ (\partial_t + v \cdot \nabla_x)f + \cA f=0$ (in the sense of Definition~\ref{defn:weaksol}). 
The energy equality reads when $s<t$,
\begin{align*}
 \|f(t)\|^2_{\L^2_{x,v}}-\|f(s)\|^2_{\L^2_{x,v}}=- 2\Re \int_{s}^t\int_{\R^d} a_{\tau,x}({f},{f})\dx\dd \tau .
\end{align*}
Letting now $s\to -\infty$ and $t\to+\infty$ and using that limits of $\L^2_{x,v}$ norms are zero, we obtain 
$\Re \int_{-\infty}^{+\infty}\int_{\R^d} a_{\tau,x}({f},{f}) \dx\dd \tau=0$. Ellipticity implies $f=0$. 
\bigskip

In a third step, we prove the existence in the case where $S\in \L^2_{t,x}\Hdot^{-\beta}_{v}=(\Fdot^{\beta})'$ by use of Theorem~\ref{thm:lions}. This is the only place where we need the assumption $\beta<d/2$ in this proof. Indeed, we recall that when $\beta<d/2$,  $\Hdot^{\beta}_{v}$ is a Hilbert space, so is $\Fdot^\beta=\L^2_{t,x}\Hdot^{\beta}_{v}$. 

To verify the hypotheses there, we let $\cH=\cD(\Omega)$ equipped with the prehilbertian norm $\|h\|_{\cH}=\|h\|_{\Fdot^\beta}$. 
In particular, the inclusion $\cH\subset \Fdot^\beta$ is continuous.

Next, consider the sesquilinear form $E$ on $\Fdot^\beta \times \cH$ defined by
\[ E(f,h) = - \iiint_{\Omega} f (\partial_t+ v \cdot \nabla_x) \bar h \dt \dx \dv + a (f,h),\]
where $a$ is defined in \eqref{e:a-defi}.

We need to check (H1), that is: for all $h\in \cH$, $f\mapsto E(f,h)$ is continuous on $\Fdot^\beta$.
To this end, we show that $|E(f,h)| \lesssim \|f\|_{\Fdot^\beta}\|h\|_{\cFdot^\beta_{\beta}}$, where the $\cFdot^\beta_{\beta}$ norm is defined in \eqref{eq:Fdotgammabetanorm}. This implies (H1) as 
$ \|h\|_{\cFdot^\beta_{\beta}}<\infty$ for $h\in \cH$.
Let $h\in \cH$ and $f\in \Fdot^\beta$. Then, by H\"older inequality and Sobolev embedding, we have 
\begin{align*}
 \bigg|\iiint_{\Omega} f (\partial_{t}+v\cdot\nabla_{x}) \overline h \dt \dx \dv\bigg|
 &\le \|f\|_{\L^2_{t,x}\L^q_{\vphantom{t,x} v}}\ \|(\partial_{t}+v\cdot\nabla_{x})h\|_{\L^2_{t,x}\L^{q'}_{\vphantom{t,x} v}}\\
 & \lesssim \|f\|_{\L^2_{t,x}\Hdot^\beta_{\vphantom{t,x} v}}\ \|(\partial_{t}+v\cdot\nabla_{x})h\|_{\L^2_{t,x}\Hdot^{-\beta}_{\vphantom{t,x} v}}.
\end{align*}
For the other term, we use the assumption on $a$ in order to get,
\[
 |a(f,h)| \le \Lambda\|f\|_{\Fdot^\beta}\|h\|_{\Fdot^\beta}. 
\]
Gathering the two terms using the definition of the $\cFdot^\beta_{\beta}$ norm in \eqref{eq:Fdotgammabetanorm} gives us the conclusion.

We next check (H2), namely that $\Re E(h,h)\ge \lambda \|h\|_{\cH}^2$ for all $h\in \cH$. Indeed, by integrating by parts, 
\begin{align}
 \label{eq:ReY=0}
 - \Re \iiint_{\Omega} h (\partial_{t}+v\cdot\nabla_{x})\overline h \dt \dx \dv =0.
\end{align}
Thus, $\Re E(h,h)= \Re a(h,h) \ge \lambda \|h\|_{\Fdot^\beta}^2=\lambda\|h\|_{\cH}^2$. 

As $S\in (\Fdot^{\beta})'$, the map $h\mapsto \iint_{\R\times \R^d} \angle{S} {h}\dx \dt $,
where the brackets inside the integral are the 
$\Hdot^{-\beta}_{v}, \Hdot^{\beta}_{\vphantom{v} v}$ duality as usual, is a continuous semi-linear functional on $\cH$. Applying the Theorem~\ref{thm:lions}, there exists $f\in \Fdot^\beta$ such that
\[ E(f,h)= \iint_{\R\times \R^d} \angle{S} {h}\dx \dt \ \text{ for all } h\in\cH. \]
In our language, this means that $f$ is a weak solution to $(\partial_t + v \cdot \nabla_x)f+\cA f=S$ in the sense of Definition~\ref{defn:weaksol}. And since $(\partial_t + v \cdot \nabla_x)f=S-\cA f \in (\Fdot^{\beta})'$, we have that $f\in \cFdot^\beta_{\beta}$. 

The second and third steps show the isomorphism property of $(\partial_t + v \cdot \nabla_x)+\cA$ from $\cFdot^\beta_{\beta}$ onto $(\Fdot^{\beta})'$. 
\bigskip

In the fourth step, 
we assume $S\in \L^2_{t,v}\Hdot_{\vphantom{t,v} x}^{-\frac{\beta}{2\beta+1}}+\L^1_{t}\L^2_{x,v}$. We prove the existence of a weak solution by a duality scheme and the use of kinetic embedding and transfer of regularity in Theorem~\ref{thm:optimalregularity}. The same argument as above applies to $((\partial_t + v \cdot \nabla_x)+\cA)^*= -(\partial_t + v \cdot \nabla_x)+\cA^*$, because \eqref{eq:ReY=0} is not changed in changing the sign of the field and $\cA^*$ is also elliptic. It is also an isomorphism from $\cFdot^\beta_{\beta}$ onto $(\Fdot^{\beta})'$. In particular, Theorem~\ref{thm:optimalregularity} implies that
\[(-(\partial_t + v \cdot \nabla_x)+\cA^*)^{-1}: (\Fdot^\beta)'\to \cFdot^\beta_{\beta}\hookrightarrow \cLdot^\beta_{\beta} \hookrightarrow \L^2_{t,v}\Hdot_{\vphantom{t,v} x}^{\frac{\beta}{2\beta+1}} \cap \C^{}_{0}(\R^{}_{t}\, ; \L^2_{x,v}).\]
Hence, the operator $\cT: {\L^2_{t,v}\Hdot_{\vphantom{t,v} x}^{-\frac{\beta}{2\beta+1}}+}\L^1_{t}\L^2_{x,v} \to \Fdot^\beta$ defined by 
\[
\angle{\cT S}h= \angle{S}{(-(\partial_t + v \cdot \nabla_x)+\cA^*)^{-1}h}, \ \text{ for } h\in (\Fdot^\beta)' \text{ and } S\in \L^2_{t,v}\Hdot_{\vphantom{t,v} x}^{-\frac{\beta}{2\beta+1}}+\L^1_{t}\L^2_{x,v}
\]
is continuous.

If $S\in \L^2_{t,v}\Hdot_{\vphantom{t,v} x}^{-\frac{\beta}{2\beta+1}}+\L^1_{t}\L^2_{x,v}$ satisfies the extra assumption $S\in (\Fdot^\beta)'$, then $\cT S$ agrees with $((\partial_t + v \cdot \nabla_x)+\cA)^{-1}S$ as previously defined, so the function $f=\cT S$ satisfies $E(f,h)=\angle{S}{h}$ for all $h\in \cH =\cD(\Omega)$.

We now remove this extra assumption. Approximate $S\in \L^2_{t,v}\Hdot_{\vphantom{t,v} x}^{-\frac{\beta}{2\beta+1}}+\L^1_{t}\L^2_{x,v}$ by $S_{j}\in \cS_{K}$ as in Lemma~\ref{lem:density}, and thus, letting $f_{j}=\cT S_{j}$, we have $E(f_{j},h)= \angle {S_{j}} h $ when $h\in \cD(\Omega)$. For fixed $h$, since $S_{j} \to S$ in $\L^2_{t,v}\Hdot_{\vphantom{t,v} x}^{-\frac{\beta}{2\beta+1}}+\L^1_{t}\L^2_{x,v}$, using the proper meaning of $\angle {S} h $ explained in Theorem~\ref{thm:optimalregularity}, $\angle {S_{j}} h \to \angle {S} h$.  As 
$f_{j}\to \cT S$ in $\Fdot^\beta$, $E(f_{j},h)\to E(\cT S,h)$ using the estimate proved for (H1) above. We have shown that $\cT S$ is a weak solution to \eqref{eq:weaksol}. Eventually, the fact that $\cT S\in \L^2_{t,v}\Hdot_{\vphantom{t,v} x}^{\frac{\beta}{2\beta+1}} \cap \C^{}_{0}(\R^{}_{t}\, ; \L^2_{x,v} )$ follows from Theorem~\ref{thm:optimalregularity}.  
 \end{proof}

\section{Kinetic Cauchy problems}
\label{s:cauchy}

This section is devoted to the construction of weak solutions to kinetic Cauchy problems on strips for time intervals
 of the form $[0,\infty)$ or $[0,T]$ with $T$ finite. In order to construct solutions
 on strips, we will
 introduce inhomogeneous kinetic spaces and prove an embedding theorem. We begin by working on infinite intervals and obtain homogeneous estimates. 

\subsection{The kinetic Cauchy problem in homogeneous spaces}

We turn our attention to the kinetic Cauchy problem
 \begin{align}
 \label{eq:CP}
 \begin{cases}
 (\partial_{t}+v\cdot\nabla_{x})f + \cA f= S, & \\
 f(0)=\psi, &
\end{cases}
\end{align}
For the kinetic Cauchy problem on an infinite time interval, we replace $\Omega$ by $\Omega_{+}= (0,\infty)\times \R^d\times \R^d$ and one can again use the homogeneous spaces on $\Omega_{+}$ with the same notation, see Section~\ref{sec:homweakkineticspaces}.  We use the same notation as well for the operator $\cA$.

\begin{defn}[Weak solutions to the kinetic Cauchy problem]
\label{defn:weaksol-cauchy}
Let $\beta >0$. Let $S \in \cD' (\Omega_+)$ and $\psi \in \L^2_{x,v}$. The distribution $S$ is assumed to extend to a continuous linear functional on $\cD (\overline{\Omega_+})$.
A distribution $f \in \cD' (\Omega_+)$ is said to be a \emph{weak solution} to \eqref{eq:CP} if $f\in \Udot^\beta$
and for all $h\in \cD (\overline{\Omega_+})$, 
\[
 -\iiint_{\Omega_{+}} f (\partial_{t}+v\cdot\nabla_{x}) \overline h \, \dt \dx \dv + a(f,h) = \angle{S}{h} + \iint_{\R^{2d}} \psi (x,v) \overline h (0,x,v) \dx \dv.
\]
\end{defn} 

Again, $f\in \Udot^\beta$ 
implies $f\in \L^2_{t,x}\cS'_{v}\cap \L^2_{t,x}\L^2_{\loc,v}$ with our definition. When $h$ is smooth, the expression $(\partial_{t}+v\cdot\nabla_{x}) \overline h$ is nothing but the pointwise partial derivative calculation. Thus, the left-hand side is well-defined; in the right-hand side, $\angle{S}{h}$ expresses the ad-hoc duality for the extension of $S$ (it will be clear in the following). 
The function $\psi$ could also be a distribution in $\R^{2d}$, in which case the corresponding term should be a sesquilinear duality bracket $\angle{\psi}{h(0)}$.
When $h\in \cD (\Omega_+)$, the last term disappears and this means in particular that $f$ is a weak solution to the first equation $(\partial_{t}+v\cdot\nabla_{x})f + \cA f= S$ in the sense of Definition~\ref{defn:weaksol} on $\Omega_{+}$. The second equation in \eqref{eq:CP} is encoded by allowing test functions that do not vanish at $t=0$. 

\begin{thm}[Existence and uniqueness of weak solutions to the kinetic Cauchy problem]
\label{thm:homCP} Let $\beta>0$.
 Let $S\in \L^2_{t,x}\Hdot^{-\beta}_{v} + \L^2_{t,v}\Hdot_{\vphantom{t,v} x}^{-\frac{\beta}{2\beta+1}}
+ \L^1_{t}\L^2_{x,v}$ and $\psi\in \L^2_{x,v}$. There exists a unique weak solution $f\in \Udot^\beta$ to the kinetic Cauchy problem
in the sense of Definition~\ref{defn:weaksol-cauchy}. 

Moreover, $f\in \L^2_{t,v}\Hdot_{\vphantom{t,v} x}^{\frac{\beta}{2\beta+1}}\cap \C^{}_{0}([0,\infty)\, ; \L^2_{x,v} )$, hence there is $\L^2_{x,v}$ convergence of $f(t)$ to $\psi$ as $t\to 0$, and 
 \begin{multline*}
 \| D_{v}^{\beta} f\|_{\L^2_{t,x,v}}+ \|D_{ x}^{\frac{\beta}{2\beta+1}} f\|_{\L^2_{t,x,v}}+ \sup_{t\in [0,\infty)} \|f(t)\|_{\L^2_{x,v}}\\
  \lesssim \|S \|_{\L^2_{t,x}\Hdot^{-\beta}_{\vphantom{t,x} v} + \L^2_{t,v}\Hdot_{\vphantom{t,v} x}^{-\frac{\beta}{2\beta+1}}+ \L^1_{t}\L^2_{x,v}} +\|\psi\|_{\vphantom{\L^2_{t,x}\Hdot^{-\beta}_{\vphantom{t,x} v} + \L^2_{t,v}\Hdot_{\vphantom{t,v} x}^{-\frac{\beta}{2\beta+1}}+ \L^1_{t}\L^2_{x,v}}\L^2_{x,v}}
 \end{multline*}
 for an implicit constant depending only on $d,\beta, \lambda, \Lambda$. In addition, $f$ satisfies the energy equality \eqref{e:energy} for non-negative times. 
\end{thm}
\begin{proof}[Proof of Theorem \ref{thm:homCP} in the case $\beta<d/2$] The strategy has several steps.
\medskip

\paragraph{Step 1:} The estimates of the backward Kolmogorov operator $\cK_{\beta}^-$ in Corollary~\ref{cor:Kbeta} 
hold on $\Omega_{+}$, by restriction to the time interval $(0,\infty)$ with $\C^{}_{0}(\R^{}_{t}\, ; \L^2_{x,v} )$ replaced by $\C^{}_{0}([0,\infty)\, ; \L^2_{x,v} )$.

\medskip

\paragraph{Step 2:} The uniqueness statement in Lemma~\ref{lem:uniqueness} holds for the backward Kolmogorov operator on $\Omega_{+}$ when $\gamma=\beta$.

\medskip

\paragraph{Step 3:} Combining Steps 1 and 2, the space $\cLdot^\beta_{\beta}$ on $\Omega_{+}$ embeds into $\C^{}_{0}([0,\infty)\, ; \, \L^2_{x,v})$.
In particular, for all $f \in \cLdot^\beta_{\beta}$, 
\[ \|f\|_{\C_0([0,\infty)\, ; \, \L^2_{x,v})} \lesssim_{d,\beta} \|f\|_{\cLdot^\beta_\beta}.\]

\paragraph{Step 4:} More generally, the statement of Theorem~\ref{thm:optimalregularity} holds on $\Omega_{+}$ with $\C^{}_{0}(\R^{}_{t}\, ; \L^2_{x,v} )$ replaced by $\C^{}_{0}([0,\infty)\, ; \L^2_{x,v} )$.  
The absolute continuity on $[0,\infty)$ is proved similarly using the backward Kolmogorov operator to construct approximations for the proof. In particular, any weak solution to the kinetic Cauchy problem as in the statement satisfies the energy equality \eqref{e:energy} of Lemma~\ref{lem:energyequalityweaksol} for non-negative times and converges in the sense of $\L^2_{x,v}$ limit at $t=0$.	

\medskip

\paragraph{Step 5:} The uniqueness in the kinetic Cauchy problem proceeds with the energy equality as in the proof of Theorem~\ref{thm:HomExUn}.

\medskip

\paragraph{Step 6:} We use $\beta<d/2$ here. Existence when $S\in \L^2_{t,x}\Hdot^{-\beta}_{v}=(\Fdot^\beta)' $, and $\psi\in \L^2_{x,v}$
also proceeds as in the proof of Theorem~\ref{thm:HomExUn} with modifications as follows. We change here $\cH$ to $\cH_{+}=\cD(\overline{\Omega_{+}})$ equipped with prehilbertian norm
$\|h\|_{\cH_{+}}^2=\|h\|_{\Fdot^\beta}^2+ \|h(0)\|_{\L^2_{x,v}}^2$. Clearly $\cH_{+}$ is continuously included in the Hilbert space $\Fdot^\beta$ on $\Omega_{+}$. 

Setting 
\[{E_{+}}(f,h):=
 -\iiint_{\Omega_{+}} f (\partial_{t}+v\cdot\nabla_{x}) \overline h \, \dt \dx \dv + a(f,h),
 \]
the proof of condition (H1) with $|{E_{+}}(f,h)| \lesssim \|f\|_{\Fdot^\beta}\|h\|_{\cFdot^\beta_{\beta}}$ for $f\in \Fdot^\beta$ and $h\in \cH_{+}$ is unchanged.
Now, for condition (H2), if $h\in \cH_{+}$ then integration by parts in $t$ and $x$ shows that
\[
 - \Re \iiint_{\Omega_{+}} h (\partial_{t}+v\cdot\nabla_{x})\overline h \dt \dx \dv = \frac 1 2 \|h(0)\|_{\L^2_{x,v}}^2.
\]
Thus, 
\[
 \Re {E_{+}}(h,h) =\frac 1 2 \|h(0)\|_{\L^2_{x,v}}^2+ \Re a(h,h) \ge \min\bigg(\lambda, \frac 1 2\bigg)\|h\|_{\cH_{+}}^2.
\]

Finally, the map $\cH_{+}\ni h\mapsto L(h)= \iint_{(0,\infty) \times \R^d} \angle{S} {h}\dx \dt + \angle{\psi}{h(0)}$,
where the brackets inside the integral are the $\Hdot^{-\beta}_{v}, \Hdot^{\beta}_{\vphantom{v} v}$ duality as usual,
  is a continuous semi-linear functional on $\cH_{+}$. For the first term, this is as before, and for the second, this is simply th e Cauchy--Schwarz inequality,
\[
 |\angle{\psi}{h(0)}| \le \|\psi\|_{\L^2_{x,v}} \|h(0)\|_{\L^2_{x,v}} \le \|\psi\|_{\L^2_{x,v}} \|h\|_{\cH_{+}}.
\]
Applying Theorem~\ref{thm:lions} again, there exists $f \in \Fdot^\beta$ such that  for all $h\in \cD(\overline{\Omega_{+}})$, ${E_{+}}(f,h)=L(h)$, that is explicitly
\begin{multline}
\label{eq:CPweaksol}
-\iiint_{\Omega_{+}} f (\partial_{t}+v\cdot\nabla_{x}) \overline h \dt \dx \dv 
 + \iint_{(0,\infty)\times \R^d} a_{t,x}(f,h)\dx \dt  \\
 =\iint_{(0,\infty) \times \R^d} \angle{S} {h}\dx \dt + \angle{\psi}{h(0)}. 
\end{multline}
 In particular, $f$ is a weak solution on $\Omega_{+}$ to $(\partial_{t}+v\cdot\nabla_{x})f + \cA f= S$. 
It remains to show $f(0)=\psi$ as we already know that $t\mapsto f(t)$ is continuous on $[0,\infty)$ in $\L^2_{x,v}$.
We remark again that $t\mapsto \angle{f(t)}{h(t)}$ is absolutely continuous and 
\begin{align*}
\angle{f(t)}{h(t)} &- \angle{f(0)}{h(0)}= - \int_{s}^t\int_{\R^d} a_{\tau,x}(f,h)\dx\dd\tau \\
& + \int_{0}^t\int_{ \R^d} \angle{S} {h}\dx\dd\tau +  \int_{0}^t\iint_{ \R^{2d}} f (\partial_{t}+v\cdot\nabla_{x}) \overline h\dx \dv \dd \tau.
 \end{align*}
 Letting $t\to +\infty$, the right hand side tends to $-\angle{\psi}{h(0)}$ by \eqref{eq:CPweaksol} and the left hand side to 
 $- \angle{f(0)}{h(0)}$ from Step 3. As $h(0)$ can be arbitrary in $\cD(\R^{2d})$, we obtain $u(0)=\psi$. 

\bigskip

\paragraph{Step 7:} It remains to prove existence when $S\in \L^2_{t,v}\Hdot_{\vphantom{t,v} x}^{-\frac{\beta}{2\beta+1}}+ \L^1_{t}\L^2_{x,v}$ and $\psi=0$.
In order to do so, we extend $S$ by $0$ for $t<0$ and take the weak solution $f$ to $(\partial_{t}+v\cdot\nabla_{x})f+\cA f=S$ on $\R\times \R^d \times \R^d$ of Theorem~\ref{thm:HomExUn}. Using the energy equality, we have when $s<t\le 0$, 
\[
 \|f(t)\|^2_{\L^2_{x,v}}-\|f(s)\|^2_{\L^2_{x,v}} = -2\Re \int_{s}^t\int_{\R^d} a_{t,x}(f,f)\dx \dt ,
\]
and letting $s\to-\infty$, we obtain $\|f(t)\|^2_{\L^2_{x,v}}\le 0$. Therefore, $f$ restricted to $\Omega_{+}$ is a weak solution on $\Omega_{+}$ with $f(0)=0$.
\end{proof}

\begin{cor}[Causality principle]
\label{cor:causality} Let $\beta>0$. Let $S\in \L^2_{t,x}\Hdot^{-\beta}_{\vphantom{t,x} v} + \L^2_{\vphantom{t,x} t,v}\Hdot_{\vphantom{t,x} x}^{-\frac{\beta}{2\beta+1}}+ \L^1_{t}\L^2_{x,v}$ on $\Omega_{+}$. If $ S_{0}$ is the zero extension of $S$ for $t<0$, then $\cK_{\cA}^+S_{0}$, with $\cK_{\cA}^+$ as in \eqref{eq:KAdefintro}, vanishes for $t\le 0$ and its restriction to $\Omega_{+}$ is the solution to the kinetic Cauchy problem \eqref{eq:CP} with zero initial value.
\end{cor}

\begin{proof}[Proof in the case $\beta<d/2$]
 This is the argument in Step 7.
\end{proof}

\begin{rem}
 Following the same strategy, one can also solve the backward kinetic Cauchy problem for $-(\partial_{t}+v\cdot \nabla_{x})+\cA^*$ on intervals $(-\infty,0]$. Step 2 uses uniqueness for distributional solutions of the forward Kolmogorov operator on $(-\infty,0)\times \R^{2d}$. As integration is on this set, the change of sign for the kinetic field gives the condition (H2) in Step 6. The other steps are rigorously the same. \end{rem}

\section{Statements on inhomogeneous spaces}

It should be clear that inhomogeneous situations yield similar results and come with slight modifications in statements and proofs. We explain this in this section with a few details when necessary. 

\subsection{Embeddings for inhomogeneous kinetic spaces} 
 The theory of homogeneous kinetic spaces adapts to inhomogeneous kinetic spaces using the inhomogeneous Sobolev spaces $\H^\gamma_{v}$ and $  \H^{\frac{\gamma}{2\beta+1}}_{\vphantom{v} x}$ for all $\beta>0$ and $\gamma\in \R$ (no restriction $\gamma<d/2$ here). It amounts to adding the $\L^2_{t,x,v}$ norm to the homogeneous ones. This leads us to define on $\Omega$ the inhomogeneous spaces $X^\gamma_{\beta}$ as in \eqref{eq:dotXgammabeta} and \eqref{eq:dotX-gammabeta}, 
 $ \cF^\gamma_{\beta}$ as in \eqref{eq:Fdotgammabeta}$, \cG^\gamma_{\beta}$ as in \eqref{eq:Kdotgammabeta}, and $ \cL^\gamma_{\beta}$ as in \eqref{eq:Ldotgammabeta} for which the analog of Theorem~\ref{thm:homkinspace} holds. 
 \begin{thm}[Kinetic embeddings and transfer of regularity: inhomogeneous case]
\label{thm:inhomkinspace}
Assume $\gamma \in [0,2\beta]$.
\begin{enumerate} 
\item  $\cL^\gamma_{\beta} \hookrightarrow \L^2_{\vphantom{\beta} t,v}\H_{\vphantom{\beta} x}^{\frac{\gamma}{2\beta+1}}$ with 
\(
\|
f\|_{\L^2_{\vphantom{\beta} t,v}\H_{\vphantom{\beta} x}^{\frac{\gamma}{2\beta+1}}} \lesssim_{d,\beta,\gamma} \|f\|_{\cL^\gamma_{\beta}}
\)
\item $\cL^\gamma_{\beta} \hookrightarrow \C^{}_{0}(\R^{}_{t}\, ; \, \X^{\gamma-\beta}_{\beta} )$ with 
\(
\sup_{t\in \R} \|f(t)\|_{\X^{\gamma-\beta}_{\beta}} \lesssim_{d,\beta,\gamma} \|f\|_{\cL^\gamma_{\beta}}.
\) 
\item There is a subspace of $\L^2_{t,x,v}$ densely contained in all  $\cF^\gamma_{\beta}, \cG^\gamma_{\beta}$ and $ \cL^\gamma_{\beta}$.
\item The families of spaces $(\cF^\gamma_{\beta})_{\gamma}$ and $(\cG^\gamma_{\beta})_{\gamma}$ have the complex interpolation property. 
\end{enumerate}
\end{thm}
\begin{rem}
To the knowledge of the authors, the first estimate (i) for $\cG^\gamma_{\beta}$ in the special case $\gamma = \beta = 1$ appears first in \cite{MR1949176}. 
\end{rem}

 This is proved using the inhomogeneous Kolmogorov operators $\cK_{1,\beta}^+ $ corresponding to $\big(\pm (\partial_{t}+v\cdot\nabla_{x}) + (-\Delta_{v})^{\beta}+1 \, \big)^{-1}$ 
 as in Section~\ref{sec:estimates} and where the density with the analogs of Lemma~\ref{lem:density} and Proposition~\ref{p:boundsKolmo} without any restriction on $\gamma$. The proofs using weighted estimates with 
 weight $\sup(1, |\xi-t\varphi|, |\varphi|^{\frac{1}{2\beta+1}})$ are similar and we skip the details. Theorem~\ref{thm:optimalregularity} and its generalisation Theorem~\ref{thm:optimalregularitygeneralisation} are also true with inhomogeneous assumptions and conclusions without the condition $\beta\pm\varepsilon<d/2$.
\subsection{Kinetic Cauchy problems on strips}
 
 We now come to equations with relaxed ellipticity conditions, still imposing measurability. We assume that there exist $0<\Lambda<\infty$ and $0\le c^0<\infty$, such that uniformly for $(t,x)\in \R\times \R^d$, 
 \begin{equation}
\label{e:ellip-upper-inhom}
\forall f,g \in \L^2_{t,x}\H^{\beta}_{\vphantom{t,x} v}, \quad |a_{t,x}(f,g)|\le \Lambda (\|f\|_{\Hdot^\beta_{v}}+c^0\|f\|_{\L^2_{v}})( \|g\|_{\Hdot^\beta_{v}}+c^0\|g\|_{\L^2_{v}})
\end{equation}
  and that there exists $\lambda>0, c_{0}\ge 0$ such that for all $f\in \H^\beta_{v}$ and uniformly for $(t,x)\in \R\times \R^d$,
\begin{align}
\label{eq:weakelliptic-bis}
 \Re a_{t,x}(f,f) \ge \lambda \|f\|_{\Hdot^\beta_{v}}^2 -c_{0}\|f\|^2_{\L^2_{v}}.
 \end{align}
 Note that this allows the possibility of including lower-order operators acting in velocity with suitable coefficients.
 
 The inhomogeneous version of existence and uniqueness of weak solutions defined in $\F^\beta= \L^2_{t,x}\H^{\beta}_{\vphantom{t,x} v}$ on $\R\times \R^d\times \R^d$ of $(\partial_{t}+v\cdot\nabla_{x})f+\cA f+cf=S$ is completely similar with $c>c_{0}$. That is, the source and the solution belong to the corresponding inhomogeneous spaces. The inhomogeneous kinetic operator $\cK_{c,A}=(\partial_{t}+v\cdot \nabla_{x}+\cA+ c)^{-1}$ has the inhomogeneous boundedness corresponding to 
 \eqref{eq:KAintro}. 

For the kinetic Cauchy problem, we may work on a finite or infinite time interval. Define all the inhomogeneous kinetic spaces as before on $\Omega_{I}=I\times \R^d \times \R^d$ where $I$ is an open interval. The embeddings are proved as follows.

For $I$ infinite, say $(0,\infty)$, use the extension to the inhomogeneous backward Kolmogorov operator of estimates in Proposition~\ref{p:boundsKolmo} and uniqueness in Lemma~\ref{lem:uniqueness}  on $\Omega_{+}$. For $(-\infty,T)$, use the same arguments for the forward Kolmogorov operator. 

For $I$ finite, say $I=(0,T)$, use cut-off in time: if $\chi\colon \R \to [0,1]$ is Lipschitz and supported in $[0,T)$ and is 1 near 0, and $f\in \cL^\gamma_{\beta}$ on $\Omega_{(0,T)}$, then $\chi f\in \cL^\gamma_{\beta}$ on $\Omega_{+}$ so that we have continuity of $f$ in $\X^{\gamma-\beta}_{\beta}$ in any subinterval $[0,T']$ of $[0,T)$. Also doing the same thing at the other endpoint of $(0,T)$ and gluing the two information, we conclude that $f\in \C([0,T]\, ;\, \X^{\gamma-\beta}_{\beta})$. If $\gamma=\beta$, we also obtain the absolute continuity of $t\mapsto\|f(t)\|_{\L^2_{x,v}}^2$ on $[0,T]$ and its extension to duality brackets $\angle{f(t)}{\tilde f(t)}$ of pairs $(f,\tilde f)\in  \cL^{\beta+\varepsilon}_{\beta}\times \cL^{\beta-\varepsilon}_{\beta}$, $-\beta\le \varepsilon\le \beta$. 

Moreover, $\L^2_{t}( I\, ; \, \cS(\R^{2d}))$ is dense in all of them, and one can also interpolate the $\cF$ and 
$\cG$ scales.

 The inhomogeneous version of Theorem~\ref{thm:homCP} is the following.
\begin{thm}[Weak solutions for the kinetic Cauchy problem in inhomogeneous spaces]
 \label{thm:CP0T}
Let $\beta>0$ and $0<T<\infty$.  Assume that the sesquilinear form $a$ satisfies \eqref{e:ellip-upper-inhom} and \eqref{eq:weakelliptic-bis}  on $(0,T)\times \R^d\times \R^d$.
 Let $S\in \L^2_{t,x}\H^{-\beta}_{v} + \L^2_{t,v}\H_{\vphantom{t,v} x}^{-\frac{\beta}{2\beta+1}}
+ \L^1_{t}\L^2_{x,v}$ and $\psi\in \L^2_{x,v}$. There exists a unique $f\in \L^2_{t,x}\H^{\beta}_{v}$ such that 
 \begin{align}
 \label{eq:CP0T}
 \begin{cases}
 (\partial_{t}+v\cdot\nabla_{x})f + \cA f= S, &
 \\
 f(0)=\psi &
\end{cases}
\end{align}
in the sense of Definition~\ref{defn:weaksol-cauchy} for test functions in $\cD([0,T)\times \R^d\times \R^d)$.

Moreover, $f\in \L^2_{t,v}\H_{\vphantom{t,v} x}^{\frac{\beta}{2\beta+1}}\cap \C([0,T]\, ; \L^2_{x,v} )$, $f(t)$ converges to $\psi$ in $ \L^2_{x,v}$ as $t\to 0$, and 
 \begin{align*}
\sup_{t\in [0,T]} \|f(t)\|_{\L^2_{x,v}} + \| D_{v}^{\beta} f\|_{\L^2_{t,x,v}}&+ \|  D_{\vphantom{D_{v}} x}^{\frac{\beta}{2\beta+1}} f\|_{\L^2_{t,x,v}} + \|  f\|_{\L^2_{t,x,v}}\\
&\lesssim \|S \|_{\L^2_{t,x}\H^{-\beta}_{v} + \L^2_{t,v}\H_{\vphantom{t,v} x}^{-\frac{\beta}{2\beta+1}}+ \L^1_{t}\L^2_{x,v}} +\|\psi\|_{\vphantom{\L^2_{t,x}\H^{-\beta}_{v} + \L^2_{t,v}\H_{\vphantom{t,v} x}^{-\frac{\beta}{2\beta+1}}+ \L^1_{t}\L^2_{x,v}} \L^2_{x,v}}
\end{align*}
for an implicit constant depending on $d,\beta, \lambda, \Lambda, c_{0}, T$. Furthermore, $f$ satisfies the energy equality \eqref{e:energy} for all times $0 \le s < t \le T$.
\end{thm}

\begin{proof} 
For existence, extend $a_{t,x}$ by the form associated with $\lambda(-\Delta_{v})^\beta$. We may call it the canonical extension, and note that it preserves the ellipticity constants $\lambda,\Lambda$. Extend $S$ by $0$ for $t\ge T$. 
Use the exponential change of variable trick $f \to e^{-ct}f$ to go back to the kinetic Cauchy problem  $(\partial_{t}+v\cdot\nabla_{x})f+ \cA f+cf=e^{-ct}S$ with initial data $\psi$.  Solve this problem with $c>c_{0}$ on $[0,\infty)$ exactly as in Theorem~\ref{thm:homCP} but with inhomogeneous spaces, then restrict to $[0,T]$ and undo the change of variable mentioned above.
 
 For uniqueness, recall that any $f\in \cL^\beta_{\beta}$ on $\Omega_{(0,T)}$ belongs to $\C([0,T]\, ; \, \L^2_{x,v})$ with the absolute continuity of $t\mapsto \|f(t)\|^2_{\L^2_{x,v}}$. If $f\in \F^\beta$ and $(\partial_{t}+v\cdot\nabla_{x})f+ \cA f=0$ on $\Omega_{(0,T)}$ with $f(0)=0$, then $f\in \cF^\beta_{\beta}\subset \cL^\beta_{\beta}$ on $\Omega_{(0,T)}$. Thus $f$ satisfies the energy equality \eqref{e:energy}:  for any $t\le T$, $\|f(t)\|_{\L^2_{x,v}}^2 + 2\Re \int_{0}^t \int_{\R^d} a_{\tau,x}(f,f)\dx\dd\tau = \|f(0)\|_{\L^2_{x,v}}^2=0$, and it follows that $f=0$.
\end{proof}

\begin{rem}
 The same arguments show that one can also solve the backward kinetic Cauchy problem for $-(\partial_{t}+v\cdot \nabla_{x})+\cA^*$ on intervals $[0,T]$ with final value at $T$, proceeding with Theorem~\ref{thm:homCP} on $(-\infty, T]$) for operators $-(\partial_{t}+v\cdot \nabla_{x})+\cA^*+c$.  \end{rem}
 
\section{Proof of Theorem~\ref{thm:HomExUn} and Theorem~\ref{thm:homCP} when $\beta\ge d/2$.}
 \label{sec:Proofswhenbetaged/2}
 Basically, the main technical difference is the construction of weak solutions having the desired homogeneous estimates since we work in $\Omega$ or $\Omega_{+}$. Recall the notation 
 $$\Ydot^{\beta}=\L^2_{t,x}\Hdot^{\beta}_{v} \cap \L^2_{t,v}\Hdot_{\vphantom{t,v} x}^{\frac{\beta}{2\beta+1}} \cap \C^{}_{0}(\R_{\vphantom{t,v} t}\, ; \L^2_{x,v})$$ and 
 $$\Zdot^{\beta}= \L^2_{t,x}\Hdot^{-\beta}_{v} + \L^2_{t,v}\Hdot_{\vphantom{t,v} x}^{-\frac{\beta}{2\beta+1}}
+ \L^1_{t}\L^2_{x,v}$$ used when working on $\R$. We shall do the proofs at the same time, assuming that the form associated to $\cA$ is defined for $t\in \R$ (we know we can always extend it anyway) and divide the argument into 6 steps.

\begin{enumerate}
    \item[Step 1:] Uniqueness in $\Udot^\beta$ on $\Omega_{+}$ and $\Omega$.
    \item[Step 2:] Existence in $\Udot^\beta$ for $S\in \L^2_{t,x}\Hdot^{-\beta}_{v}$ on $\Omega_{+}$. 
    \item[Step 3:] Existence in $\Udot^\beta$ for $S\in \L^2_{t,x}\Hdot^{-\beta}_{v}$ on $\Omega$.
    \item[Step 4:] On $\Omega$, $\cK_{\cA}^+: \L^2_{t,x}\Hdot^{-\beta}_{v} \to \Ydot^{\beta}$. 
    \item[Step 5:] On $\Omega$, $\cK_{\cA}^+: \Zdot^{\beta} \to \Ydot^{\beta}$ (end of proof of Theorem~\ref{thm:HomExUn}).
    \item[Step 6:] Causality (end of proof of Theorem~\ref{thm:homCP}).
\end{enumerate}

\smallskip
\paragraph{Step 1:} Uniqueness is a consequence of the energy equality for weak solutions. The one in $\Omega$ is proved in Lemma~\ref{lem:energyequalityweaksol}. For the one in $\Omega_{+}$, one follows the first steps of the proof of Theorem~\ref{thm:homCP} when $\beta<d/2$. Indeed, all the results used there are proved for $\beta>0$. 
\medskip
\paragraph{Step 2:} Assume $S\in \L^2_{t,x}\Hdot^{-\beta}_{v} $ and $\psi\in \L^2_{x,v}$ and $0<T<\infty$. As $\L^2_{t,x}\Hdot^{-\beta}_{v} \subset \L^2_{t,x}\H^{-\beta}_{v}$,
  we may apply Theorem~\ref{thm:CP0T} with inhomogeneous spaces. Hence, there exists a weak solution  $f_{\varepsilon}\in \L^2_{t,x}\H^{\beta}_{v} \cap \L^2_{t,v}\H_{\vphantom{t,v} x}^{\frac{\beta}{2\beta+1}}\cap \C([0,T]\, ; \L^2_{x,v} )$ of the kinetic Cauchy problem on $(0,T)$ 
  $$
\begin{cases} (\partial_{t}+v\cdot\nabla_{x})f_{\varepsilon}+ \cA f_{\varepsilon} +\varepsilon f_{\varepsilon}= e^{-\varepsilon t}S
\\
f_{\varepsilon}(0)=\psi.
\end{cases}
$$ 
Rewriting the equation as $(\partial_{t}+v\cdot\nabla_{x})(e^{\varepsilon t}f_{\varepsilon})+ \cA (e^{\varepsilon t}f_{\varepsilon})= S$ and using the energy equality on all subintervals $[0,\tau]$ leads to 
the bound 
 $$\|D_{v}^\beta (e^{\varepsilon t}f_{\varepsilon})\|^2_{\L^2(0,T\, ;\,\L^2_{x,v})} + \sup_{t\in [0,T]} \|e^{\varepsilon t}f_{\varepsilon}(t)\|^2_{\L^2_{x,v}} \lesssim \|S\|^2_{\L^2_{t,x}\Hdot^{-\beta}_{\vphantom{t,x}v}} +\|\psi\|^2_{\L^2_{x,v}}
 $$
 with uniform implicit constant with respect to $\varepsilon>0$ and $T$. 
 Hence, we can extract when $\varepsilon\to 0$ a weak$^*$-limit for the sequence $e^{\varepsilon t}f_{\varepsilon} \in \L^\infty_{t}\L^2_{x,v}\cap \L^2_{t,x}\Hdot^{\beta}_{v}$ and obtain a weak solution $f^T\in \Udot^\beta$ in $\Omega_{(0,T)}$ of $(\partial_{t}+v\cdot\nabla_{x})f^T+ \cA f^T = S$ with same initial data $\psi$. Observe that this solution satisfies
 $$\|D_{v}^\beta f^T\|^2_{\L^2(0,T\, ;\,\L^2_{x,v})} + \sup_{t\in [0,T]} \|f^T(t)\|^2_{\L^2_{x,v}} \lesssim \|S\|^2_{\L^2_{t,x}\Hdot^{-\beta}_{\vphantom{t,x}v}} +\|\psi\|^2_{\L^2_{x,v}}
 $$
 with constant uniform with respect to $T$. Then, the compatibility for different $T$ implied by uniqueness allows us to define $f$ in $\Omega_{+}$ as $f=f^T$ on $\Omega_{(0,T)}$. This $f$ is a weak solution on $\Omega_{+}$ with initial data $\psi$, and it satisfies the above bound for $T=\infty$ (by a monotone convergence argument as $T\uparrow \infty$). In particular, $f\in \Udot^\beta$ in $\Omega_{+}$. 
 \medskip 
\paragraph{Step 3:}
 Assume $S\in \L^2_{t,x}\Hdot^{-\beta}_{v}$ on $\Omega$ and $\psi=0$. We can apply Step 2 on any interval $[-k,\infty)$, $ k\in \IN$, with zero initial condition and push backward the initial time $-k$ to $-\infty$ to obtain, again with a weak$^*$-limit argument, using uniform bounds with respect to $k$, a weak solution $f \in \Udot^\beta$ in $\Omega$. 
 \medskip 
\paragraph{Step 4:} Existence and uniqueness show that one can define $\cK_{\cA}^+: \L^2_{t,x}\Hdot^{-\beta}_{v} \to \Udot^\beta$ as $\cK_{\cA}^+ S$ is the unique solution $f\in \Udot^\beta$ to 
$(\partial_{t}+v\cdot\nabla_{x})f+ \cA f = S$ on $\Omega$. 
  It follows that $f\in \cLdot^\beta_{\beta}$. By Theorem~\ref{thm:homkinspace}, $f\in \L^2_{t,x}\Hdot^{\beta}_{v} \cap \L^2_{t,v}\Hdot_{\vphantom{t,v} x}^{\frac{\beta}{2\beta+1}} \cap \C^{}_{0}(\R_{\vphantom{t,v} t}\, ; \L^2_{x,v})$. This establishes the desired boundedness 
  \begin{align*}
 \cK_{\cA}^+\colon \L^2_{t,x}\Hdot^{-\beta}_{v} \to \Ydot^{\beta}.
\end{align*}
 \medskip 
\paragraph{Step 5:} We wish to apply the duality scheme as in an earlier argument, but we have to do it with care. The result of Step 4 applies to the backward equation for $\cA^*$ and yields
\begin{align*}
 \cK_{\cA^*}^-: \L^2_{t,x}\Hdot^{-\beta}_{v} \to \Ydot^\beta.
\end{align*}
Note that if $S\in \cS_{K}$ and
$h\in \L^2_{t,x}\Hdot^{-\beta}_{v}$, then 
\[ |\angle{\cK_{\cA}^+S}{h}|= |\angle{S}{\cK_{\cA^*}^-h}|\le \|S\|_{\Zdot^{\beta}}\|\cK_{\cA^*}^-h\|_{\Ydot^\beta} \lesssim \|S\|_{\Zdot^{\beta}}\|h\|_{\L^2_{t,x}\Hdot^{-\beta}_{v}}
\]
using Step 4 in the last inequality. Consequently, we can interpret $\cK_{\cA}^+ S$ as a bounded linear functional on $\L^2_{t,x}\Hdot^{-\beta}_{v}$, hence in $(\L^2_{t,x}\Hdot^{-\beta}_{\vphantom{t,x} v})'$, with norm controlled by $ \|S\|_{\Zdot^{\beta}}$.
Since $\cS_{K}\subset \L^2_{t,x}\Hdot^{-\beta}_{v}$, we know by Step 4 that $\cK_{\cA}^+ S\in \Ydot^\beta\subset \Udot^\beta$, the completion of which for the semi-norm  $\|\cdot\|_{\Udot^\beta}=\|\cdot \|_{\L^2_{t,x}\Hdot^{\beta}_{\vphantom{t,x}v}}$ is that dual space. In particular, for any $f\in \Udot^\beta$, $\|f \|_{\Udot^\beta} = \|f\|_{(\L^2_{t,x}\Hdot^{-\beta}_{\vphantom{t,x} v})'}$ under identification. It follows that 
 \[
 \|\cK_{\cA}^+ S\|_{\Udot^{\beta}}  \lesssim \|S\|_{\Zdot^{\beta}}.
\]
From the equation we deduce that $\cK_{\cA}^+ S \in \cLdot^\beta_{\beta}$ with
\[
\|\cK_{\cA}^+ S\|_{\cLdot^\beta_{\beta}}  \lesssim \|S\|_{\Zdot^{\beta}}.
\]
Now, as noted in Corollary~\ref{cor:complete}, $\cLdot^\beta_{\beta}$ is a complete normed space, so 
arguing by density of $\cS_{K}$ in $\Zdot^{\beta}$, we conclude that $\cK_{\cA}^+$ extends boundedly from $\Zdot^{\beta}$ to $\cLdot^\beta_{\beta}$, which is contained in $\Udot^\beta$. Eventually, this extension maps $\Zdot^{\beta}$ into weak solutions by passing to the limit in the weak formulation of the equation. This completes the proof of existence on $\Omega$ for source terms in $\Zdot^{\beta}$. 
\medskip 
\paragraph{Step 6:} The causality principle follows again from the energy equality and allows one to construct weak solutions on $\Omega_{+}$ as in Step 7 of the proof of Theorem~\ref{thm:homCP} when $\beta<d/2$.
This concludes the argument, and also proves Corollary~\ref{cor:causality} when $\beta\ge d/2$.

\begin{rem} 
As we have now fully proved Theorem~\ref{thm:homCP} when $\beta >0$, we may go back to Theorem~\ref{thm:CP0T}
in the case $c_{0}=c^0=0$ in \eqref{e:ellip-upper-inhom} and \eqref{eq:weakelliptic-bis}. Assume $S\in \L^2_{t,x}\Hdot^{-\beta}_{v} + \L^2_{t,v}\Hdot_{\vphantom{t,v} x}^{-\frac{\beta}{2\beta+1}}
+ \L^1_{t}\L^2_{x,v}$ on $\Omega_{(0,T)}$ and $\psi\in \L^2_{x,v}$. We have nothing else to do than restrict the construction on $[0,\infty)$ to $[0,T]$, having extended $S$ by 0 for $t\ge T$. The weak solution enjoys the homogeneous estimates in $\L^2_{t,x}\Hdot^{\beta}_{v} \cap \L^2_{t,v}\Hdot_{\vphantom{t,v} x}^{\frac{\beta}{2\beta+1}} \cap \C^{}_{0}([0,T]\, ; \L^2_{x,v})$ with bounds independent of $T$ if $S$ is normalized. It also has an $\L^2_{t,x,v}$ bound on $\Omega_{(0,T)}$ but this bound blows up with $T$.  
\end{rem}

\section{The special case of local diffusion operators}
\label{sec:localcase}

In this section, we focus on $\beta=1$ and give different definitions of the homogeneous spaces that are valid in any dimension $d\ge 1$.  

Define $\Wdot^{1,2}(\R^d)=\{f\in \cD'(\R^d)\, ;\, \nabla f\in \L^2(\R^d)\}$, equipped with the semi-norm $\|\nabla f\|_{\L^2}$. Remark that $\Hdot^1(\R^d) \subset \Wdot^{1,2}(\R^d)$ when $d\ge 3$.  
The inhomogeneous version is $\W^{1,2}(\R^d)=\{f\in\L^2(\R^d)\, ;\, \nabla f\in \L^2(\R^d)\}$ equipped with the norm $(\|f\|_{\L^2}^2+\|\nabla f\|_{\L^2}^2)^{1/2}$. It agrees with $\H^1(\R^d)$ when $d\ge 1$. Hence, our inhomogeneous theory is unchanged and we only have to address the homogeneous one.

\begin{lem}
\label{lem:hombeta=1}
Any element of $ \Wdot^{1,2}(\R^d)$ identifies to an element in $ \L^2_{\loc}(\R^d) \cap \cS'(\R^d)$, and $\cD(\R^d)$ is dense in $\Wdot^{1,2}(\R^d)$ for the semi-norm.\end{lem}

\begin{proof}
This is known. Here is a proof for the comfort of the reader. Let $f\in \Wdot^{1,2}(\R^d)$. As $f$ is a distribution, we may regularise it to $f_{\varepsilon}$ in this space by mollification, and $f_{\varepsilon} \in C^\infty$. Use Poincar\'e inequalities for $f_{\varepsilon}$ and limiting arguments to deduce that $f\in \L^2_{\loc}(\R^d)$. Poincar\'e inequalities again give a polynomial growth in $R$ of the $\L^2$ norms of $f$ on balls $B(0,R)$. Thus $f$ is a tempered distribution as well. Lastly, for the density, we claim that we can approximate $f$ by compactly supported functions in the $\|\nabla f\|_{2}$ semi-norm. Pick $f_{k}(x)=(f(x)-c_{k})\varphi(2^{-k}
x)$, $k\ge 1$, with $\varphi$ a smooth function supported on the ball $B(0,2)$ and that is 1 on the unit ball and where $c_{k}$ is the mean of $f$ on the ball $B(0,2^{k+1})$. It is easy to see that $\nabla f_{k}$ converges weakly to $\nabla f$ in $
\L^2$ and by Mazur's lemma, there exist convex combinations of subsequences of $\nabla f_{k}$ having strong convergence to $\nabla f$ in $\L^2$. From there, use mollification to obtain an approximation in $\cD(\R^d)$. 
\end{proof}

 We introduce the kinetic space
 \begin{align}
 \label{eq:Ldot}
 \cLdot= \{ f \in \cD'(\Omega)\, ; \, f \in \L^2_{t,x}\Wdot^{1,2}_{\vphantom{t,x} v} \ \& \ (\partial_{ t} + v \cdot \nabla_x)f \in \L^2_{\vphantom{t,x} t} \Xdot^{-1}_{\vphantom{t,x} 1}+\L^1_{\vphantom{t,x} t} \L^2_{\vphantom{t,x} x,v}\}
\end{align}
with semi-norm $ \|f\|_{\cLdot}$ defined by
\begin{align*}
 \|f\|_{\cLdot}= \| \nabla_{v} f\|_{\L^2_{t,x,v}} + \|(\partial_t + v \cdot \nabla_x)f\|_{\L^2_{t}\Xdot^{ -1}_{\vphantom{t,x,v} 1}+\L^1_{\vphantom{t,x,v} t} \L^2_{\vphantom{t,x,v} x,v}}.
 \end{align*}
 The spaces with negative indices are defined as before and $\Hdot^{-1}_{v}$ identifies to the dual of $\Wdot^{1,2}_{v}$, so that we also write $\Hdot^{-1}_{v}=\Wdot^{-1,2}_{v}$ and $\Xdot^{-1}_1= \L^2_{x}\Wdot^{-1,2}_{\vphantom{1} v}+ \L^2_{\vphantom{1} v}\Hdot_{\vphantom{1} x}^{-\frac{1}{3}}$. Note that the $f \in \L^2_{t,x}\Wdot^{1,2}$ condition excludes non-zero constants (in the first version of this article, we wrote the opposite, which would be true if we replace this condition by $\nabla f\in \L^2_{t,x,v}$). Hence, all results in this section are modified accordingly.

Redoing for the Kolmogorov operators $\cK^\pm:=\cK_{1}^\pm$ the proof in Corollary~\ref{cor:Kbeta}, we can state the following boundedness result.

\begin{prop}[Boundedness of the Kolmogorov operators] We have the bounded map
\[\cK^\pm\colon \L^2_{t,x}\Wdot^{-1,2}_{v}+\L^2_{t,v}\Hdot_{\vphantom{t,v} x}^{-\frac{1}{3}}+ \L^1_{t} \L^2_{\vphantom{t,x} x,v} \to \L^2_{t,x}\Wdot^{1,2}_{\vphantom{t,x} v}\cap  \L^2_{\vphantom{t,x} t,v}\Hdot_{\vphantom{t,x} x}^{\frac{1}{3}} \cap \C^{}_{0}(\R_{\vphantom{t,x} t}\, ; \,\L^2_{\vphantom{t,x} x,v}). 
\]
\end{prop}

\begin{cor}[Uniqueness]\label{cor:nabla}
Let $f\in \cD'(\Omega)$ for which $\nabla_{v}f\in \L^2_{t,x,v}$. Then $\pm (\partial_{t}+v\cdot\nabla_{x})f -\Delta_{v}f=0$ in the sense of distributions implies that 
 $f$ is constant. If we further impose $f \in \L^2_{t,x}\Wdot^{1,2}_{\vphantom{t,x} v}$ then $f=0$.
 \end{cor} 
 \begin{proof} Assume $ (\partial_{t}+v\cdot\nabla_{x})f -\Delta_{v}f=0$ as the other case is similar. We follow the strategy of Lemma~\ref{lem:uniqueness} and the notation there. By Lemma~\ref{lem:hombeta=1}, $f\in \L^2_{t,x}\cS'_{v}\subset \cS'(\Omega)$ but $\hat f$ may not be locally integrable on $\Omega$.
 Still, one can solve the ordinary differential equation in a distributional sense and obtain that 
 $G= K(t,0)G_{0}$ in $\cD'(\Omega)$ with $G_{0}\in \cD'_{\varphi,\xi}$. Next the condition $(\xi-t\varphi)G\in \L^2_{t,\varphi,\xi}$ implies that the restriction of $G$ to the complement of $\{\xi-t\varphi=0\}$ is locally square integrable. 
This forces $G_{0}$ to be supported at the origin $(\varphi,\xi)=(0,0)$. Indeed,  pick a compact set $C$ in $\R_{\varphi,\xi}^{2d}\setminus \{(0,0)\}$. Then one can find a compact interval $I\subset\R^{}_{t}$ such that $\min |\xi-t\varphi|^2>0$ on $I\times C$. Thus $G\in \L^2(I\times C)$ and, using the equation, we see that $G\in \C(I; \L^2(C))$.  This implies that $G_{0}\in \L^2(C)$. Using again 
 $(\xi-t\varphi)G\in \L^2_{t,\varphi,\xi}$, this gives us 
 \begin{equation*}
 \iiint_{\R\times C} |\xi-t\varphi|^2\exp\bigg({-2\int_{0}^t |\xi-\tau \varphi|^{2}\, d\tau}\bigg)| G_{0}(\varphi,\xi)|^{2} \, \dt \dd \varphi \dd \xi <\infty,
\end{equation*}
 which yields $G_{0}=0$ on $C$. This proves that $G$ is supported in $\R^{}_{t}\times \{(0,0)\}$. This implies $(\xi-t\varphi)G=0$ as a distribution because it is locally integrable in $\Omega$ and supported in the null Lebesgue set $\R^{}_{t}\times \{(0,0)\}$. Thus $\partial_{t}G=0$ from the equation. Differentiating $\xi G= t\varphi G$ with respect to $t$ leads to $\varphi G=0$, and consequently $\xi G=0$. This implies that $G=T_{0}\otimes \delta$ where $\delta$ is the Dirac mass at the origin in $\R_{\varphi,\xi}^{2d}$ and $T_{0}\in \cD'(\R^{}_{t})$. One uses again $\partial_{t}G=0$ to conclude that $T_{0}$ is constant. 
 Going back to $f$, we conclude that $f$ is constant. And this constant must be 0 if $f \in \L^2_{t,x}\Wdot^{1,2}_{\vphantom{t,x} v}$.
 \end{proof}

\begin{prop} The Kolmogorov operators $\cK^\pm$ are isomorphisms from $\L^2_{t,x} \Wdot^{-1,2}_{\vphantom{t,x} v}$ onto $\cFdot$, from  $\L^2_{t} \Xdot^{-1}_{1}$ onto $\cGdot$ and from $\L^2_{t} \Xdot^{-1}_1+ \L^1_{t} \L^2_{x,v}$ onto $\cLdot$, using quotient norms, where the spaces $\cFdot$ and $\cGdot$ are defined similarly to $\cLdot$ replacing $\L^2_{t,x}\Hdot^{1}_{v}$ by $\L^2_{t,x}\Wdot^{1,2}_{v}$. In particular, $\cFdot$, $\cGdot$ and $\cLdot$ are complete normed spaces.

\begin{proof}
 This follows from the estimates and uniqueness, as in Lemma~\ref{lem:isom-embed}. 
\end{proof}

 \end{prop}
 
 \begin{cor}\label{cor:cLemb}
 Let $f\in \cLdot$. Then 
$ f\in \C^{}_{0}(\R^{}_{t}\, ; \,\L^2_{x,v} )\cap \L^2_{t,v}\Hdot_{\vphantom{t,v} x}^{\frac{1}{3}}$ with estimate 
\begin{equation*}
\| f\|_{\L^2_{t,v}\Hdot_{\vphantom{t,v} x}^{\frac{1}{3}}}+ \sup_{t\in \R} \|f(t)\|_{\L^2_{x,v}} \lesssim_{d} \|f\|_{\cLdot}.\end{equation*} 
 The space in item~(iii) of Theorem~\ref{thm:homkinspace} is dense in $\cLdot$. Moreover, for $f\in \cLdot$, the absolute continuity stated in Theorem~\ref{thm:optimalregularity} (and its generalization Theorem~\ref{thm:optimalregularitygeneralisation}) holds.
\end{cor}

\begin{proof}
Follow the strategy of proof of Theorem~\ref{thm:homkinspace}. Details are left to the reader. 
\end{proof}
 
\begin{rem} One can now see that 
 $ \cLdot=\cLdot^{1}_{1}$ when $d\ge 3$. It is only because we have not defined $\Hdot^{1}_v$ when $d=1,2$ that we do not have a statement for $\cLdot^{1}_{1}$ in that case, see \cite{AN}.
\end{rem}

Let us turn to the construction of weak solutions. The difficulty with this homogeneous space here is that $\L^2_{t,x}\Wdot^{1,2}_{v}$ is not a Hilbert space, so that we cannot apply the Lions theorem directly. We use an approximation procedure starting from the inhomogeneous case. We let $\mathbf{A}=\mathbf{A}(t,x,v)$ be a bounded and coercive matrix with measurable coefficients (real and symmetric is not necessary) in such a way that 
\begin{equation*}
a_{t,x}(f,g)= \int_{\R^d}\angle{\mathbf{A}(t,x,v)\nabla_{v} f(v)}{\nabla_{v}g(v)}\, \dv, \quad f,g\in \Wdot^{1,2}_{v},
\end{equation*}
satisfies \eqref{e:ellip-upper} and \eqref{e:ellip-lower} on $\Wdot^{1,2}_{v}$. We define $\cA$ as in \eqref{e:fk} with the form $a$ defined by \eqref{e:a-defi} on $\L^2_{t,x}\Wdot^{1,2}_{v}$. Here we assume either $t\in \R$ or $t\in (0,\infty)$.

\begin{thm} \label{thm:beta1}
\begin{enumerate}
\item On $\R$, if $S\in \L^2_{t}\Xdot^{-1}_1+\L^1_{t} \L^2_{x,v}$, then there is a unique weak solution $f \in \L^2_{t,x}\Wdot^{1,2}_{v}$ to  $(\partial_{t}+v\cdot\nabla_{x})f + \cA f= S$. We have $f\in \C^{}_{0}(\R^{}_{t}\, ; \,\L^2_{x,v} )\cap \L^2_{t,v}\Hdot_{\vphantom{t,v} x}^{\frac{1}{3}}$ and satisfies the energy equality. 
\item On $[0,\infty)$, if $S\in \L^2_{t}\Xdot^{-1}_1+\L^1_{t} \L^2_{x,v}$ and $\psi\in \L^2_{x,v}$, then there is a unique weak solution in $\L^2_{t,x}\Wdot^{1,2}_{v}$ to the kinetic Cauchy problem  $(\partial_{t}+v\cdot\nabla_{x})f + \cA f= S$ with $f(0)=\psi$. We have  $f\in \C^{}_{0}([0,\infty)\, ; \,\L^2_{x,v} )\cap \L^2_{t,v}\Hdot_{\vphantom{t,v} x}^{\frac{1}{3}}$ and satisfies the energy equality.
\end{enumerate}
\end{thm}

\begin{proof} 
 We start with uniqueness. In case (i), a weak solution  $f \in \L^2_{t,x}\Wdot^{1,2}_{v}$ to $(\partial_{t}+v\cdot\nabla_{x})f + \cA f= S$ belongs to $\cLdot$. Thus, $f\in \C^{}_{0}(\R^{}_{t}\, ; \,\L^2_{x,v} )$ and satisfies the energy equality. If we assume $S=0$, then for all $s<t$,
\begin{align*}
 \|f(t)\|^2_{\L^2_{x,v}}-\|f(s)\|^2_{\L^2_{x,v}}=- 2\Re \int_{s}^t\int_{\R^d} a_{\tau,x}({f},{f})\dx\dd \tau. 
\end{align*}
In particular, this yields after taking limits at $\pm\infty,$ $\Re \int_{-\infty}^{+\infty}\int_{\R^d} a_{\tau,x}({f},{f}) \dx\dd \tau=0$, hence $\Re a_{\tau,x}({f},{f})=0$ for almost all $(\tau,x)$ (because it is non negative). This yields $\|f(t)\|^2_{\L^2_{x,v}}-\|f(s)\|^2_{\L^2_{x,v}}=0$ for all $s<t$, that is $t\mapsto \|f(t)\|^2_{\L^2_{x,v}}$ is constant. As it has zero limits at $\pm\infty$,  $f=0$.

In case (ii), we can produce the same argument following Steps 1 and 2 in the proof of Theorem~\ref{thm:homCP} to show that on $\Omega_{+}$,
for all $f \in \cLdot$,  
\[ \|f\|_{\C_0([0,\infty)\, ; \, \L^2_{x,v})} \lesssim_{d} \|f\|_{\cLdot}.\]
Now consider a weak solution $f$ with $S=0$ and $\psi=0$. As we assume that $f(0)=0$ at time 0, the energy equality yields for $t>0$, $\|f(t)\|^2_{\L^2_{x,v}}=- 2\Re \int_{0}^t\int_{\R^d} a_{\tau,x}({f},{f})\dx\dd \tau\le 0$, thus $f(t)=0$.

Now we turn to existence. It follows the argument done in Section~\ref{sec:Proofswhenbetaged/2}. We give details for completeness. We begin with (ii) in the case where $S\in \L^2_{t,x}\Wdot^{-1,2}_{v}$ and $\psi\in \L^2_{x,v}$. We prove existence on any finite strip $\Omega_{(0,T)}=(0,T)\times \R^d\times \R^d$ of a weak solution $f^T$ with $\nabla_{v}f^T\in \L^2(\Omega_{(0,T)}) $ and $f^T\in \C([0,T]\, ; \, \L^2_{x,v})$. Assume this is done. Such solutions also verify the energy equality. If $T<T'$ and $g=f^{T'} - f^T$, then $g$ is a weak solution in the same class on $[0,T]$ with zero source and zero initial data. The energy equality yields when $t\in [0,T]$, $\|g(t)\|^2_{\L^2_{x,v}}=- 2\Re \int_{0}^t\int_{\R^d} a_{\tau,x}({g},{g})\dx\dd \tau\le 0$, thus $g(t)=0$. Since $f^T$ and $f^{T'}$ agree on the smaller strip, this allows us to define $f$ on $\Omega_{+}$ and $f$ is a weak solution to the kinetic Cauchy problem in (ii) with the desired properties. 

To construct $f^T$, we argue using the inhomogeneous setup on $[0,T]$ that is now fixed, and we work on this interval. For $\varepsilon>0$, using Theorem~\ref{thm:CP0T} with $\cA+\varepsilon$, we solve in $\L^2_{t,x}\W^{1,2}_{v}$
 the kinetic Cauchy problem $(\partial_{t}+v\cdot\nabla_{x})f_{\varepsilon}+ \cA f_{\varepsilon} +\varepsilon f_{\varepsilon}= e^{-\varepsilon t}S$ with $f_{\varepsilon}(0)=\psi$. Note that $(\partial_{t}+v\cdot\nabla_{x})(e^{\varepsilon t}f_{\varepsilon})+ \cA(e^{\varepsilon t}f_{\varepsilon}) = S$ and $(e^{\varepsilon t}f_{\varepsilon})(0)=\psi$. Using the energy equality for $e^{\varepsilon t}f_{\varepsilon}$ (absolute continuity follows from that of $f_{\varepsilon}$) and the fact that $S\in \L^2_{x}\Wdot^{-1,2}_{v}$, we can see that 
 $\|\nabla_{v}(e^{\varepsilon t}f_{\varepsilon})\|_{\L^2_{t,x,v}}$ and $\sup_{t\in [0,T]}\|e^{\varepsilon t}f_{\varepsilon}(t)\|_{\L^2_{x,v}}$ are uniformly bounded with respect to $\varepsilon$. Hence, a classical weak$^*$-limit argument furnishes a weak solution $f^T \in \L^\infty_{t}\L^2_{x,v}\cap \L^2_{t,x}\W^{1,2}_{v}$. As $f^T$ belongs to the inhomogeneous version $\cL$ of $\cLdot$ on $\Omega_{(0,T)}$, continuity in time follows. 
 
 Now, we turn to existence in (i) when $S\in \L^2_{t,x}\Wdot^{-1,2}_{v}$ on $\Omega$ this time.  By applying (ii) with time interval $[-n,\infty)$, $n\in \IN$, instead on $[0,\infty)$,  consider the weak solution $f_{n}$ on $[-n,\infty)$ to the kinetic Cauchy problem $(\partial_{t}+v\cdot\nabla_{x})f_{n}+ \cA f_{n} = S$ and $f_{n}(-n)=0$. The extension of this solution by 0 when $t\le -n$ belongs to
 $\L^2_{t,x}\Wdot^{1,2}_{v} \cap \C^{}_{0}(\R^{}_{t}\, ; \,\L^2_{x,v} )\cap \L^2_{t,v}\Hdot_{\vphantom{t,v} x}^{\frac{1}{3}}$ on $\Omega$, uniformly in $n$. Extract a weak$^*$-limit in $\L^2_{t,x}\Wdot^{1,2}_{v} \cap \L^\infty_{t}\L^2_{x,v} \cap \L^2_{t,v}\Hdot_{\vphantom{t,v} x}^{\frac{1}{3}}$. This weak$^*$-limit is a weak solution on $\R$ and as it is already a bounded function of time valued in $\L^2_{x,v}$, it further belongs to $\C^{}_{0}(\R^{}_{t}\, ; \,\L^2_{x,v} )$ using the embedding in Corollary~\ref{cor:cLemb}. 
 
 It remains to solve the problems for source terms in $\L^2_{t,v}\Hdot_{\vphantom{t,v} x}^{-\frac{1}{3}}\cap \L^1_{t}\L^2_{x,v}$. For this, we use the duality scheme in Section~\ref{sec:Proofswhenbetaged/2} first on $\R$ starting from the operator 
 \begin{align*}
 \tilde\cK_{\cA^*}^-: \L^2_{t,x}\Wdot^{-1,2}_{v} \to \Ydot^\beta \hookrightarrow \cLdot
\end{align*}
defined as $\tilde\cK_{\cA^*}^-S$ is the weak solution for the backward equation, as we just did for the forward equation, and using that $\cLdot$ is a complete normed space.
Next, on $(0,\infty)$ with zero initial data, we use this previous case and the causality principle
 as in Step 7 of the proof of Theorem~\ref{thm:homCP}.
 \end{proof}

\begin{rem}
 When $d\ge 3$, as $\L^2_{t,x}\Hdot^{1}_{v}\subset \L^2_{t,x}\Wdot^{1,2}_{v}$, uniqueness implies that the weak solutions in Theorem~\ref{thm:beta1}, (ii), and in Theorem~\ref{thm:homCP} when $\beta=1$ are the same. Similarly, the weak solution $f$ in Theorem~\ref{thm:beta1}~(i) is the same as the one in Theorem~\ref{thm:HomExUn}. This holds as well when $d=1,2$ provided one defines $\Hdot^{1}_v$ appropriately, see \cite{AN}.
\end{rem}
 
 Let us illustrate our results by a consequence for optimal regularity, to be compared with \cite[Theorem 1.5]{MR1949176} where $f$ and $\nabla_{v}f$ were both supposed to belong to $\L^2_{t,x,v}$ for the conclusion $\Delta_{v}f\in \L^2_{t,x,v}$. 

\begin{cor}[Optimal regularity]
\label{cor:optimalregularity} Let $f\in \cD'(\Omega)$ with $\sup_{\sigma\in \R}\|\nabla_{v}f\|_{\L^2(\Omega_{(\sigma,\sigma+1)})}<\infty$ $($in particular, if
$\nabla_{v}f\in \L^2_{t,x,v})$ be such that $S= (\partial_{t}+v\cdot\nabla_{x})f -\Delta_{v}f \in \L^2_{t,x,v}$. Then, $(\partial_{t}+v\cdot\nabla_{x})f, \Delta_{v}f\in \L^2_{t,x,v}$ and $\nabla_{v}f\in \C^{}_{0}(\R^{}_{t}\, ; \,\L^2_{x,v} )$. 
 \end{cor}
 
\begin{proof} 
Again, the estimates obtained using Fourier methods in Proposition~\ref{prop:estimatesT} imply in particular the boundedness
\[\cK^\pm\colon \L^2_{t,x,v}\to \L^2_{ t,x}\Wdot^{2,2}_{ v}
\cap \C^{}_{ 0 \vphantom{x}}(\R^{}_{\vphantom{x} t}\, ; \,\L^2_{\vphantom{t,x,v} x}\Wdot^{1,2}_{ v}) \]
where $\Wdot^{2,2}(\R^d)=\{f\in \cD'(\R^d)\, ;\, \nabla^2 f\in \L^2(\R^d)\}.$
Thus, $\cK^+S$ is a tempered distribution solution of the same equation as $f$. We show that $f-\cK^+S$ is constant, which concludes the proof from the properties of $\cK^+S$.

Remark that from the above bounds $\nabla_{v}\,\cK^+S$ satisfies the same condition as $\nabla_{v}f$ in the hypothesis. Applying the method of proof of Corollary~\ref{cor:nabla} to $f-\cK^+S$ and calling $G$ the Fourier transform of $\Gamma(f-\cK^+S)$, the only thing that changes is the condition $(\xi-t\varphi)G\in \L^2_{t,\varphi,\xi}$ replaced by $(\xi-t\varphi)G \in \L^2(\sigma,\sigma+1; \L^2_{\varphi,\xi})$ uniformly in $\sigma\in \R$. This is enough to conclude that the restriction of $G_{0}$ to any compact set $C\subset \R^{2d}\setminus \{(0,0)\}$ is square integrable and that 
 \begin{equation*}
\sup_{\sigma\in \R} \iiint_{(\sigma,\sigma+1)\times C} |\xi-t\varphi|^2\exp\bigg({-2\int_{0}^t |\xi-\tau \varphi|^{2}\, d\tau}\bigg)| G_{0}(\varphi,\xi)|^{2} \, \dt \dd \varphi \dd \xi <\infty.
\end{equation*} 
By taking $\sigma\to -\infty$, this is only possible if $G_{0}=0$ almost everywhere on $C$. The rest of the proof is verbatim the one of Corollary~\ref{cor:nabla}. 
\end{proof}

\begin{rem}
 At the expense of more technicalities, one can make the same discussion with the higher-order homogeneous space
$\Wdot^{m,2}(\R^d)=\{f\in \cD'(\R^d)\, ;\, \nabla^m f\in \L^2(\R^d)\}$, $m$ integer, which contains polynomials of degree less than $m$, applying this to higher order Kolmogorov operators $\cK_{m}^\pm$. For weak solutions, ellipticity \eqref{e:ellip-lower} should be in the sense of a G\aa rding inequality on $\Wdot^{m,2}_{v}$. Embedding (on $\R$ or $[0,\infty)$) and uniqueness of weak solutions on $\R$ can be proved. 
\end{rem}


\section{Link with other kinetic embeddings} 
\label{sec:kinmax}

As mentioned in the introduction, there is a notion of kinetic spaces developed in \cite{MR4350284} where embeddings similar to ours are proved without restriction on $\gamma$, while we are limited to $0\le \gamma\le 2\beta$. It is worth explaining the links between the two works. For this, we need to introduce more spaces and make some extensions (mostly without proofs, which are adaptations of the previous arguments).

\subsection{Homogeneous source spaces revisited}
We first extend the range of parameter $\gamma$ in our spaces $\Xdot^{\gamma}_{{\beta}}$. We just mention at this stage that we did not do it in the article because it would not have helped to prove the existence of weak solutions in Theorem~\ref{thm:HomExUn}, where we only impose ellipticity in the $v$-variable.

In the proof of Lemma~\ref{l:galilean}, we established the following isomorphism via the Fourier transform in $\R^{2d}$
\begin{align*}
 \Xdot^{\gamma}_{{\beta}} \to \L^2_{\varphi,\xi}(w^{2\gamma} \dd \varphi \dd\xi),
\end{align*}
where $w(\varphi,\xi)=\sup(|\xi|, |\varphi|^{\frac{1}{1+2\beta}})$, provided 
$\gamma<d/2$. When $\gamma>0$, the only reason for the restriction $\gamma<d/2$ is our definition of $\Xdot^{\gamma}_{{\beta}}$ as the intersection of two spaces, separating the variables in some sense. We remark that a simple calculation shows for $\gamma\in \R$, 
\[ w^{-\gamma} \in \L^2_{\loc}(\R^{2d}) \Longleftrightarrow \gamma< (\beta+1)d =\frac{(2\beta+2)d}{2}.
\]
Note that $(2\beta+2)d$ is the anisotropic homogeneous dimension for the scaling $(x,v)\mapsto (\delta^{2\beta+1}x, \delta v)$.
This allows us to define $\Xdot^{\gamma}_{{\beta}}$ for $\gamma< (\beta+1)d$ in the spirit of homogeneous Sobolev spaces by
\[ \Xdot^{\gamma}_{{\beta}}= \{f\in \cS'(\R^{2d})\, ; \, \exists g \in \L^2(\R^{2d}),  \hat f = w^{-\gamma}\hat g\}
\]
with norm $\|f \|_{\Xdot^{\gamma}_{{\beta}}}=\|g\|_{\L^2_{\varphi,\xi}}$. In this range, it is a well-defined Hilbert space, contained in $ \cS'(\R^{2d})$ with elements having Fourier transforms in $\L^1_{\loc}(\R^{2d})$.\footnote{One can give a definition a a space of tempered distribution also when $\gamma\ge (\beta+1)d$; see \cite{AN}.} When $\gamma<d/2$, this space is equal, as a set and with equivalent norm, to the one defined in \eqref{eq:dotXgammabeta} and \eqref{eq:dotX-gammabeta}. 

With this definition, our proofs on the Fourier transform side (where there is no restriction on exponents) show that the statements of Lemma~\ref{l:galilean}, 
Proposition~\ref{p:boundsKolmo} hold for $\gamma<(\beta+1)d$. Also the uniqueness as in Lemma~\ref{lem:uniqueness} applies with $f\in \L^2_{t}\Xdot^{\gamma}_{\beta}$ and $\gamma<(\beta+1)d$, thanks to Remark~\ref{rem:uniqueness}.

\subsection{More homogeneous kinetic spaces}

We remark that even with this new definition of $\Xdot^{\gamma}_{{\beta}}$, the upper bound on $\gamma$ in the definitions of our kinetic spaces $\cFdot^\gamma_{\beta}$, $\cGdot^\gamma_{\beta}$, $\cLdot^\gamma_{\beta}$ is $d/2$ because we impose them to be subspaces of 
$\L^2_{t,x}\Hdot^{\gamma}_{v}$. Nevertheless, changing this condition to $f\in \L^2_{t}\Xdot^{\gamma}_{\beta}$ allows to introduce when $\gamma<(\beta+1)d$,
\begin{align}
 \label{eq:Tdotgammabeta}
 \cTdot^\gamma_{\beta}= \{ f \in \cD'(\Omega)\, ; \, f \in \L^2_{t}\Xdot^{\gamma}_{\beta} \ \& \ (\partial_t + v \cdot \nabla_x)f \in \L^2_{t} \Xdot^{\gamma -2\beta}_\beta\}
\end{align}
with Hilbertian norm $ \|f\|_{\cTdot^\gamma_{\beta}}$ defined by
\begin{align*}
 \|f\|_{\cTdot^\gamma_{\beta}}^2= \| f\|_{\L^2_{t}\Xdot^{\gamma}_{\beta}}^2 + \|(\partial_t + v \cdot \nabla_x)f\|_{\L^2_{t}\Xdot^{\gamma-2\beta}_\beta}^2, 
 \end{align*}
 and the larger space
 \begin{align}
 \label{eq:Udotgammabeta}
 \cUdot^\gamma_{\beta}= \{ f \in \cD'(\Omega)\, ; \, f \in \L^2_{t}\Xdot^{\gamma}_{\beta} \ \& \ (\partial_t + v \cdot \nabla_x)f \in \L^2_{t} \Xdot^{\gamma -2\beta}_\beta+\L^1_{t} \Xdot^{\gamma-\beta}_{\beta}\}
\end{align}
with norm $ \|f\|_{\cUdot^\gamma_{\beta}}$ defined by
\begin{align*}
 \|f\|_{\cUdot^\gamma_{\beta}}= \| f\|_{\L^2_{t}\Xdot^{\gamma}_{\beta}} + \|(\partial_t + v \cdot \nabla_x) f\|_{\L^2_{t}\Xdot^{\gamma-2\beta}_\beta+\L^1_{t} \Xdot^{\gamma-\beta}_{\beta}}.
 \end{align*}
 Note that the $x$-regularity is already included in the definition of these spaces, hence $ \cTdot^\gamma_{\beta} \subset \cGdot^\gamma_{\beta}$ and $ \cUdot^\gamma_{\beta} \subset \cLdot^\gamma_{\beta}$ when $0\le \gamma<d/2$. When $\gamma\le 0$, the opposite inclusions hold. 
 
 There is an immediate observation that $(-\Delta_{v})^{\beta}$ maps $\L^2_{t}\Xdot^{\gamma}_{\beta}$ into $\L^2_{t}\Xdot^{\gamma-2\beta}_\beta$ when $\gamma<(\beta+1)d$. Indeed, writing $ \hat f = w^{-\gamma}\hat g$ with $\hat g\in \L^2_{t,\varphi,\xi}$, we have $|\xi|^{2\beta} \hat f= w^{-\gamma+2\beta} (|\xi|^{2\beta}w^{-2\beta}\hat g)$. As $|\xi|^{2\beta}w^{-2\beta}\hat g$ is in $\L^2_{t,\varphi,\xi}$ and $\gamma-2\beta<\gamma$, we have that $|\xi|^{2\beta} \hat f$ is a tempered distribution with inverse (partial) Fourier transform in $\L^2_{t}\Xdot^{\gamma-2\beta}_\beta$. 

 \subsubsection{Another homogeneous kinetic embedding}
 
In this scale, the kinetic embedding statement becomes the following one. 
\begin{thm}[Homogeneous kinetic embedding: extension]\label{thm:homkinspaceT}
Assume $\gamma <(\beta+1)d$. 
\begin{enumerate} 
\item $\cTdot^\gamma_{\beta} \subset \cUdot^\gamma_{\beta} \hookrightarrow \C^{}_{0}(\R^{}_{t}\, ; \, \Xdot^{\gamma-\beta}_{\beta} )$ with 
\(
\sup_{t\in \R} \|f(t)\|_{\Xdot^{\gamma-\beta}_{\beta}} \lesssim_{d,\beta,\gamma} \|f\|_{\cUdot^\gamma_{\beta}}.
\) 
\item The subspaces  in Theorem~\ref{thm:homkinspace}, (iii), are dense in all $\cTdot^\gamma_{\beta} $ and $\cUdot^\gamma_{\beta}$.
\item The family of spaces $(\cTdot^\gamma_{\beta})_{\gamma}$  has the complex interpolation property. 
\end{enumerate}
\end{thm}
 Again, this is only the definition of spaces that imposes the upper bound on $\gamma$. The inclusion (i) in this statement is weaker than (ii) in Theorem~\ref{thm:homkinspace}, but for a larger range.

We also have modifications of the statements of Theorem~\ref{thm:optimalregularity} and of Theorem~\ref{thm:optimalregularitygeneralisation} assuming $f\in \cUdot^{\beta+\varepsilon}_{\beta}$ and $\tilde f\in \cUdot^{\beta-\varepsilon}_{\beta}$, with only constraint $\beta\pm\varepsilon<(\beta+1)d$. In the latter statement, as there is no transfer of regularity, (i) and (ii) only contain the sup norm estimates. Next, (iii) has to be modified slightly with assumptions on the sums $S_{1}+S_{2}$, $\widetilde S_1 +\widetilde S_2$ rather than individual assumptions. The dualities in the formula for the derivative also have to be adapted. This is left to the reader. 

The isomorphism statement Lemma~\ref{lem:isom-embed} is modified as follows: when $\gamma<(\beta+1)d$,  
$\cK^\pm_\beta$ are isomorphisms from 
    $\L^2_{t} \Xdot^{\gamma -2\beta}_{\beta}$ onto $\cTdot^\gamma_{\beta}$ and from $\L^2_{t} \Xdot^{\gamma -2\beta}_\beta+ \L^1_{t} \Xdot^{\gamma-\beta}_{\beta}$ onto $\cUdot^\gamma_{\beta}$. This is what induces the interpolation property for the scale of spaces $\cTdot^\gamma_{\beta}$.

The above discussion also extends to homogeneous spaces on $\Omega_{+}$. 

\subsection{Inhomogeneous versions}

Consider the inhomogeneous versions $\cT^\gamma_{\beta}(I)$, $\cU^\gamma_{\beta}(I)$ on $\Omega$, $\Omega_{+}$ or on strips $(0,T)\times\R^d\times \R^d$ where $I=\R,(0,\infty)$ or $(0,T)$, respectively. With these spaces, there are no conditions on $\gamma\in \R$. The natural generalisations to Theorem~\ref{thm:homkinspaceT} and absolute continuity hold. 
In particular, one has 
\begin{equation}
\label{eq:cTembedding}
\cT^\gamma_{\beta}(0,T) \hookrightarrow \C([0,T]\, ; \, \X^{\gamma-\beta}_{\beta} )\ \mathrm{with}\ 
\sup_{t\in [0,T]} \|f(t)\|_{\X^{\gamma-\beta}_{\beta}} \lesssim_{d,\beta,\gamma, T} \|f\|_{\cT^\gamma_{\beta}(0,T)}.
\end{equation}
If $T=\infty$, the homogeneous estimate is the one corresponding to (i) in Theorem~\ref{thm:homkinspaceT} (with $\gamma<(\beta+1)d$). 

\begin{prop}
If $0\le \gamma\le 2\beta$, then $ \cG^\gamma_{\beta}(0,T)= \cT^\gamma_{\beta}(0,T)$. As a consequence, Theorem~\ref{thm:inhomkinspace}(i), and \eqref{eq:cTembedding} are the same statements in this range of $\gamma$.
\end{prop}

\begin{proof}
 When $\gamma\ge 0$, the definition yields the inclusion $ \cT^\gamma_{\beta}(0,T) \subset \cG^\gamma_{\beta}(0,T)$. The inhomogeneous version  of the transfer of regularity contained in Theorem~\ref{thm:inhomkinspace}(ii) shows that if  $0\le \gamma\le 2\beta$ then $ \cG^\gamma_{\beta}(0,T) \subset \cT^\gamma_{\beta}(0,T)$.
\end{proof}

When $\gamma\notin [0,2\beta]$, this equality of spaces does not seem true, hence the kinetic embedding for $ \cG^\gamma_{\beta}(0,T)$ is not clear as well.

\subsection{Comparison with the embeddings in \cite{MR4350284}} We can now make a precise link between the estimate \eqref{eq:cTembedding} and \cite[Theorem~5.13]{MR4350284}). More precisely,  the spaces $\cT^\gamma_{\beta}(0,T)$ are introduced (with different notation and in the range $\gamma\ge \beta$: a close examination shows this restriction was not necessary and their arguments work for all $\gamma$) as part of a more general framework (with temporal weights) for solving the kinetic Cauchy problem for the Kolmogorov equation with constant diffusion ($0<\beta\le 1$) and \eqref{eq:cTembedding} is a key ingredient in the argument. While the techniques used there relied on the Fourier transform as well, the method is different. It is proved by first establishing a weaker embedding in $\C([0,T]\, ; \, \X)$ where $\X$ is a Sobolev space with a large enough negative exponent. This argument is not possible when working with homogeneous spaces. Then, by an interpolation argument, the exact trace space was characterised as $\X^{\gamma-\beta}_\beta$. Some estimates were mentioned to hold when $T=\infty$, but not the homogeneous embedding that would correspond to \eqref{eq:cTembedding}. 

\bibliographystyle{plain}

\end{document}